\documentclass[a4paper,11pt]{article}
\usepackage{inputenc}
\usepackage{authblk}
\usepackage[table]{xcolor}
\usepackage{color}
\usepackage{amsmath}
\usepackage{amssymb}
\usepackage{amsthm}
\usepackage{booktabs}
\usepackage{xcolor,graphicx,float}
\usepackage{hyperref}
\usepackage{soul}
\usepackage{indentfirst}
\usepackage{enumitem}
\setlist[itemize]{itemsep=0pt, topsep=0pt}
\usepackage{tikz}
\usepackage{subcaption}
\usepackage{caption}

\hypersetup{colorlinks=true,
	linkcolor=blue,
	anchorcolor=blue,
	citecolor=blue}
\usepackage[
a4paper,
textwidth=16cm,
textheight=23cm,
centering
]{geometry}

\theoremstyle{plain}
\newtheorem{theorem}{\bf Theorem}[section]
\newtheorem{lemma}[theorem]{\bf Lemma}
\newtheorem{proposition}[theorem]{\bf Proposition}
\newtheorem{corollary}[theorem]{\bf Corollary}
\newtheorem{conjecture}{\bf Conjecture}

\theoremstyle{remark}

\newtheorem{problem}{\bf Problem}
\newtheorem{construction}{\bf Construction}

\newtheorem{assumption}{\bf Assumption}

\numberwithin{equation}{section}

\definecolor{tonglv}{RGB}{43,174,133}

\newcommand{\spn}[2]{\genfrac{\{}{\}}{0pt}{}{#1}{#2}}
\title{\bf A unified approach to cross-intersection problems with applications to Hilton--Milner type theorems and stability}
\author[1]{Jie Wen\thanks{E-mail: \text{jwen@mail.bnu.edu.cn}}}
\author[1]{Benjian Lv\thanks{Corresponding author. E-mail: \text{bjlv@bnu.edu.cn}}}
\affil[1]{\small Laboratory of Mathematics and Complex Systems (Ministry of Education), School of Mathematical Sciences, Beijing Normal University, Beijing 100875, China}
\date{}

\begin{document}
	\renewcommand{\baselinestretch}{1.2}
	\maketitle
	\begin{abstract}
		\medskip
We develop a new approach to cross-intersection problems in extremal set theory. The method builds on the iterative procedure introduced by Kupavskii and Zakharov (2024) and the $t$-cover method. It provides a flexible framework for deriving extremal and stability results for cross $t$-intersecting families. Our approach applies to a variety of combinatorial objects. As an application, we prove a product version of the seminal Erd\H{o}s--Ko--Rado theorem for sufficiently spread set systems, leading to improved bounds for Erd\H{o}s--Ko--Rado type results for signed sets and $k$-partitions.

Two families $\mathcal{F}$ and $\mathcal{G}$ of $k$-subsets of the standard $n$-set $[n]=\{1,2,\ldots,n\}$ are called cross $t$-intersecting if $|F\cap G|\geq t$ for all $F\in\mathcal{F}$ and $G\in\mathcal{G}$. We determine the families maximizing $\min\{|\mathcal{F}|, |\mathcal{G}|\}$ for $n\geq t+6(k-t)\cdot\max\{5(t+2), k\}$ and all $t\ge2$, generalizing results of M\"{o}rs (1985) and F\"{u}redi (1995) for cross $1$-intersecting families. This also yields a new proof of a classical Hilton--Milner type theorem of Frankl (1978). We then establish a product version of the theorem by  determining the families maximizing $|\mathcal{F}||\mathcal{G}|$ under the condition $\max\{|\cap_{F\in\mathcal{F}}F|,|\cap_{G\in\mathcal{G}}G|\}<t$ for $n\geq t+\max\{15(t+2)^2(k-t),10k^2\}$. This improves the bound $n\ge 4(t+2)^2k^2$ obtained by Frankl and Wang (2024), and provides a characterization of extremal configurations. 
		
For a family $\mathcal{F}$ of subsets of $[n]$, we introduce its $t$-diversity $\gamma_t(\mathcal{F})$, defined as the minimum number of sets from $\mathcal{F}$ not containing a fixed $t$-subset. This serves as a natural generalization of the important notion of diversity for $t=1$. We obtain a stability result via $\gamma_t$, and determine the maximum of  $\min\{\gamma_t(\mathcal{F}),\gamma_t(\mathcal{G})\}$ for cross $t$-intersecting families $\mathcal{F}$ and $\mathcal{G}$. These yield new results for $t$-intersecting families, including a stability theorem towards a conjecture of Ellis, Keller and Lifshitz (2019), which may also be regarded as a $t$-intersection version, for large $n$, of an influential theorem of Frankl (1987).
		
		
		Spread approximation;\;Stability;\;$t$-cover;\;$t$-diversity
		
	\end{abstract}
	
	\section{Introduction}
	Extremal set theory is concerned with determining or estimating the size of a family of sets satisfying prescribed restrictions. A cornerstone of the area is the Erd\H{o}s--Ko--Rado theorem \cite{Erdos-Ko-Rado-1961}, which determines, for large $n$ depending on $k$ and $t$, the maximum  size of a family $\mathcal{F}$ of $k$-subsets of an $n$-set with the property that any two of its members have at least $t$ elements in common. Such a family is said to be \emph{$t$-intersecting}. Originating from the celebrated theorem, the study of \emph{intersection problems} has since developed into a longstanding and active direction of research in extremal combinatorics. 
	
	Let us introduce some standard notation. We set $[n]:=\{1,2,\ldots,n\}$, and set $[i,j]:=\{i,i+1,\ldots,j\}$ for $i\leq j$. Denote by $\binom{[n]}{k}$ the family of all $k$-subsets of $[n]$. For a family $\mathcal{F}$ of sets, we use  $\cup\mathcal{F}$ and $\cap\mathcal{F}$ to denote the union and the intersection of its members, respectively. A $t$-intersecting  family  $\mathcal{F}\subseteq\binom{[n]}{k}$ is said to be  \emph{trivial} if $|\cap\mathcal{F}|\geq t$.  Otherwise, it is \emph{non-trivial}. Two families $\mathcal{F},\mathcal{G}$ of subsets of $[n]$ are \emph{isomorphic}, denoted $\mathcal{F}\cong\mathcal{G}$, if they are the same up to a permutation on $[n]$.
	
	Formally, let us present the following \emph{exact Erd\H{o}s--Ko--Rado theorem}.
\begin{theorem}[\cite{Erdos-Ko-Rado-1961,Frankl-1976,Wilson-1984}]\label{EKR}
	Let $k\geq t\geq 1$ and $n\geq(t+1)(k-t+1)$. If $\mathcal{F}\subseteq\binom{[n]}{k}$ is $t$-intersecting, then
	$$|\mathcal{F}|\leq\binom{n-t}{k-t}.$$
	Moreover, if $n>(t+1)(k-t+1)$, then equality holds if and only if $\mathcal{F}\cong\left\{F\in\binom{[n]}{k}:[t]\subseteq F\right\}$.
\end{theorem}	
Determining the least possible value $n_0(k,t)$ for which the upper bound holds was a major open problem.  The original paper \cite{Erdos-Ko-Rado-1961} determined $n_0(k,1)=2k$, and established $n_0(k,t)\leq(k-t)\binom{k}{t}^3+t$. In the seminal paper, the authors introduced the \emph{shifting technique} (see also \cite{Frankl-shifting}), which has been a fundamental tool in extremal set theory. Frankl \cite{Frankl-1976} made a breakthrough by determining $n_0(k,t)=(t+1)(k-t+1)$ for $t\geq 15$. Wilson \cite{Wilson-1984} subsequently showed $n_0(k,t)=(t+1)(k-t+1)$ for $2\leq t\leq14$, with an ingenious algebraic proof valid for all $t\geq1$. A natural problem is what happens when $n<n_0(k,t)$. In this range, the family
\begin{equation}\label{familya}
	\mathcal{A}(n,k,t):=\left\{F\in\binom{[n]}{k}:|F\cap [t+2]|\geq t+1\right\}
\end{equation}
has size larger than $\binom{n-t}{k-t}$. In 1997, Ahlswede and Khachatrian \cite{Ahlswede-Khachatrian-1997} established the famous \emph{complete intersection theorem}, which determines maximum-sized $t$-intersecting families in $\binom{[n]}{k}$ for all values of $n, k$, and $t$. This settled a conjecture of Frankl \cite{Frankl-1976}. 

For $n>n_0(k,t)$, an optimal family in Theorem \ref{EKR} consists of all $k$-subsets of $[n]$ containing a fixed $t$-subset. Such a family is trivial, and usually referred to as a \emph{full $t$-star}. In 1967, Hilton and Milner \cite{Hilton-Milner-1967} determined the maximum-sized non-trivial $1$-intersecting families and, in particular, for $(k,t)\neq(3,1)$, the unique extremal family is isomorphic to $\mathcal{H}(n,k,1)$, where for $t\geq1$ we write
\begin{equation}\label{familyh}
\mathcal{H}(n,k,t):=\left\{F\in\binom{[n]}{k}:[t]\subseteq F,\;F\cap[t+1,k+1]\neq\emptyset\right\}\cup\left\{[k+1]\setminus\{i\}:i\in [t]\right\}.
\end{equation}
The Hilton--Milner theorem initiated the problem of characterizing  the structure of large non-trivial $t$-intersecting families. One major contribution is due to Frankl \cite{Frankl-1976}, who determined non-trivial $t$-intersecting families maximizing the size for $t\geq2$ and large $n$ depending on $k$ and $t$. Ahlswede and Khachatrian \cite{Ahlswede-Khachatrian-1996} determined such families for all possible parameters.
\begin{theorem}[\cite{Hilton-Milner-1967,Frankl-1976,Ahlswede-Khachatrian-1996}]\label{HM}
	Let $k\geq t\geq 1$ and $n\geq(t+1)(k-t+1)$. If $\mathcal{F}\subseteq\binom{[n]}{k}$ is a non-trivial $t$-intersecting family, then
	$$|\mathcal{F}|\leq\max\{|\mathcal{A}(n,k,t)|,  |\mathcal{H}(n,k,t)|\}.$$
	Moreover, if equality holds, then $\mathcal{F}\cong\mathcal{A}(n,k,t)$ or $\mathcal{F}\cong\mathcal{H}(n,k,t)$.
\end{theorem}

Theorem \ref{HM} can be regarded as a \emph{stability result} for the Erd\H{o}s--Ko--Rado theorem: if a $t$-intersecting family is nearly extremal in size, in the sense that its size is larger than that of $\mathcal{A}(n,k,t)$ and $\mathcal{H}(n,k,t)$, then it is also close to the optimum in structure, namely, it is contained in a full $t$-star. 

Another classical problem closely related to stability for the Erd\H{o}s--Ko--Rado theorem concerns how a large $t$-intersecting family can differ from a full $t$-star. A natural way to measure this difference is through degree conditions. In 1987, Frankl \cite{Frankl-1987} established a far-reaching sharpening of the Hilton--Milner theorem, characterizing largest $1$-intersecting families under constraints on the maximum degree. Frankl's theorem has motivated several insightful results over the past decade, which we discuss in Section \ref{subsecdiv}. For some related stability results for the Erd\H{o}s--Ko--Rado theorem, we refer the reader to \cite{Keevash,Dinur-Friedgut,Kupavskii-Zakharov-2018,Ellis-2019,Keevash-Long-2020}.

 Intersection problems arise naturally for a variety of mathematical objects. The study of these problems has given rise to a wealth of results and methods in combinatorics. For a systematic introduction to the field, we refer the reader to the  comprehensive monographs \cite{ Frankl-Tokushige-book, Godsil-Meagher-book} and  survey papers \cite{Deza-Frankl-1983,Frankl-Tokushige-2016,Ellis-book}. 
 
 In 2024, Kupavskii and Zakharov  \cite{Kupavskii-Zakharov-2024} introduced the powerful \emph{spread approximation method}. The method builds on the breakthrough paper \cite{sunflower} of Alweiss, Lovett, Wu and Zhang on the Erd\H{o}s--Rado sunflower conjecture, and has been further developed in a series of works. This has led to progress on a variety of major open problems, including the forbidden intersection problem \cite{Kupavskii-Zakharov-2024}, the  Hajnal--Rothschild problem \cite{Frankl-Kupavskii-2025}, the intersection problem for permutations \cite{Kupavskii-Zakharov-2024,Kupavskii-2024,Saengrungkongka,KKLS} and set partitions \cite{Kupavskii-2026}. 
 
 Another powerful and influential method is the \emph{junta method} introduced by Keller and Lifshitz \cite{Keller-Lifshitz}, which was further developed by Ellis, Keller and Lifshitz \cite{Ellis-2024}. This has been used to resolve several longstanding open problems in combinatorics. We further refer the reader to recent breakthrough works based on hypercontractivity for global functions, due to Keevash, Lifshitz, Long and Minzer \cite{Keevash-2021,Keevash-2023,Keevash-2024} and to Keller, Lifshitz, Minzer and Sheinfeld \cite{KLMS}.
 
 One of the most natural and extensively studied generalizations of the notion of $t$-intersection is that of \emph{cross $t$-intersection}. Two families $\mathcal{F}\subseteq\binom{[n]}{k}$ and $\mathcal{G}\subseteq\binom{[n]}{\ell}$ are said to be \emph{cross $t$-intersecting} if $|F\cap G|\geq t$ for all $F\in\mathcal{F}$ and $G\in\mathcal{G}$. In particular, a $t$-intersecting family is cross $t$-intersecting with itself. Let $\mathcal{F}\subseteq\binom{[n]}{k}$ and $\mathcal{G}\subseteq\binom{[n]}{\ell}$ be cross $t$-intersecting. The pair $(\mathcal{F},\mathcal{G})$ is \emph{maximal} if $\mathcal{F}=\mathcal{F}'$ and $\mathcal{G}=\mathcal{G}'$ whenever $\mathcal{F}'\subseteq\binom{[n]}{k}$ and $\mathcal{G}'\subseteq\binom{[n]}{\ell}$ are cross $t$-intersecting with $\mathcal{F}\subseteq\mathcal{F}'$ and $\mathcal{G}\subseteq\mathcal{G}'$.
 
 In this paper, we present a unified approach to cross $t$-intersection problems, and use it to derive a range of extremal and stability results. It is inspired by the peeling procedure developed by Kupavskii and Zakharov in \cite{Kupavskii-Zakharov-2024}, which forms part of the spread approximation method, and incorporates ideas from the \emph{$t$-cover method} (cf. e.g., \cite{Blokhuis-etal-2010,Lv-2021,Cao-Lu-Lv-Wang-2024,Wen-Lv-2026}). More precisely, given a pair of cross $t$-intersecting families, we generate, via an iterative procedure, a sequence of pairs of families, called \emph{fingerprints}, which encode structural information on the original pair. This enables us to apply the $t$-cover method more effectively, whose key step is precisely to find a suitable collection of $t$-covers, and thereby obtain improved estimates. Let us note two features of our approach. First, complementing existing methods, the approach is useful for characterizing large cross $t$-intersecting families, and provides a flexible framework for proving Hilton--Milner type theorems and stability results. Second, there are a variety of powerful techniques for dealing with cross $1$-intersecting families, many of which rely crucially on the shifting technique and the Hilton lemma (observed by Hilton in \cite{Hilton-1976} as an application of the seminal Kruskal--Katona theorem; see also e.g. \cite{Frankl-1987,Frankl-Tokushige-2016,Kupavskii-Zakharov-2018}). To the best of the authors' knowledge, such methods are  less readily adaptable to the case $t\geq2$, since no analogue of the Hilton lemma is available in this setting.  This makes the present method particularly relevant to the study of cross $t$-intersecting families for $t\geq2$. In the rest of this section, we present several of its applications.
 \subsection{Erd\H{o}s--Ko--Rado theorems for sufficiently spread structures}
 Let $\Omega$ be a finite set. For a family $\mathcal{A}$ of subsets of $\Omega$ and a subset $S$ of $\Omega$, we write
 \begin{align}
 	\mathcal{A}[S]&:=\left\{A\in\mathcal{A}:S\subseteq A\right\}\label{equlink}
 \end{align}
 the collection of members of $\mathcal{A}$ containing $S$. 
 The \emph{maximum $s$-degree} of $\mathcal{A}$, denoted $\Delta_s(\mathcal{A})$, is the maximum number of members of $\mathcal{A}$ that contain a given $s$-subset. Formally, $$\Delta_s(\mathcal{A}):=\max\left\{|\mathcal{A}[S]|:S\in\binom{\Omega}{s}\right\}.$$
 The maximum $1$-degree is simply called the maximum degree, and we use the notation $\Delta(\mathcal{A})$ for that of $\mathcal{A}$. A family $\mathcal{A}$ of subsets of $\Omega$ is said to be \emph{$r$-spread} if $\Delta_s(\mathcal{A})\leq r^{-s}|\mathcal{A}|$ for all non-negative integer $s$. For example, the family $\binom{[n]}{k}\subseteq2^{[n]}$ is $r$-spread whenever $r\leq r_0:=n/k$, due to the fact that $\binom{n-s}{k-s}\leq r_0^{-s}\binom{n}{k}$ for all  $s$. 
 
 Spreadness  plays an important role in the paper \cite{sunflower} of Alweiss, Lovett, Wu and Zhang, as well as in a series of recent major advances in combinatorics. In that breakthrough paper, a key observation concerning the intersection problem from \cite[Theorem 4.2]{sunflower} is that a $(C\log k)$-spread $k$-uniform family, for all sufficiently large constants $C$, indeed contains two disjoint sets, and consequently cannot be $1$-intersecting. As pointed out in \cite{Frankl-Kupavskii-2025,Kupavskii-2026}, a similar observation is also crucial for the spread approximation method.
 
 In a highly spread family, intersecting subfamilies without a large common intersection are unlikely to be large. In this sense, many natural intersection settings involve sufficiently spread ambient families, in which large intersecting families are therefore expected to be star-like. This naturally leads to the problem of establishing an Erd\H{o}s--Ko--Rado type theorem for sufficiently spread settings. A family $\mathcal{A}$ of subsets of $\Omega$ is 
 \emph{weakly $(r,t)$-spread} if $$\Delta_{t+s}(\mathcal{A})\leq r^{-s}\Delta_{t}(\mathcal{A})\;\;\mbox{for}\;\;s=0,1,2,\ldots\;.$$
 The notion of weak  spreadness was introduced by Kupavskii  \cite{Kupavskii-2026} to settle intersection problems for some relatively `irregular' structures, including several classes of set partitions with restricted sizes of blocks. Our first main result gives such a theorem for sufficiently spread families.
 \begin{theorem}\label{thmcrossekr}
 	Let $k\geq\ell\geq t+1$ be positive integers and  $r\geq1$. Let $\Omega$ be a finite set and let $\mathcal{A}\subseteq\binom{\Omega}{k}$ and $ \mathcal{B}\subseteq\binom{\Omega}{\ell}$ be two weakly $(r,t)$-spread families. The following hold.
 	\begin{itemize}
 		\item[\rm(i)]If $r\geq2ek$ and  $\mathcal{F}\subseteq\mathcal{A}$ is $t$-intersecting, then $|\mathcal{F}|\leq\Delta_{t}(\mathcal{A})$, with equality if and only if  $\mathcal{F}=\mathcal{A}[X]$ for some $X\in\binom{\Omega}{t}$ with $|\mathcal{A}[X]|=\Delta_{t}(\mathcal{A})$.
 		\item[\rm(ii)]If $r\geq2e\cdot\max\{(t+2)^2,k/(1-t/k)\}$ and $\mathcal{F}\subseteq\mathcal{A}, \mathcal{G}\subseteq\mathcal{B}$ are cross $t$-intersecting, then
 		$|\mathcal{F}||\mathcal{G}|\leq\Delta_{t}(\mathcal{A})\Delta_{t}(\mathcal{B})$, 
 		with equality if and only if  $\mathcal{F}=\mathcal{A}[X]$ and $\mathcal{G}=\mathcal{B}[X]$ for some $X\in\binom{\Omega}{t}$ with $|\mathcal{A}[X]|=\Delta_{t}(\mathcal{A})$ and $|\mathcal{B}[X]|=\Delta_{t}(\mathcal{B})$.
 	\end{itemize}
 \end{theorem}
 A feature of this result is that an Erd\H{o}s--Ko--Rado type theorem holds whenever the families under consideration are sufficiently spread, regardless of their precise structure. We note that Theorem \ref{thmcrossekr} (i) was proved by Kupavskii  \cite{Kupavskii-2026} for $r\geq24k$, and our contribution here is limited to improving the constant. We discuss in Section \ref{secrmk}  some applications of the result on $t$-intersecting families for signed sets and $k$-partitions.
\subsection{Structure of large cross $t$-intersecting families}
A natural and longstanding parameter for a pair $(\mathcal{F},\mathcal{G})$ of cross $t$-intersecting families is
$$\min\{|\mathcal{F}|, |\mathcal{G}|\}.$$
In 1967, Hilton and Milner conjectured in the classical paper \cite{Hilton-Milner-1967} that if $n\geq2k$ and $\mathcal{F},\mathcal{G}\subseteq\binom{[n]}{k}$ are cross $1$-intersecting, then $\min\{|\mathcal{F}|, |\mathcal{G}|\}\leq\binom{n-1}{k-1}$. In 1968, Kleitman settled the conjecture in a stronger form that if $n\geq k+\ell$, and $\mathcal{F}\subseteq\binom{[n]}{k},\mathcal{G}\subseteq\binom{[n]}{\ell}$ are cross $1$-intersecting, then either $|\mathcal{F}|\leq\binom{n-1}{k-1}$, or $|\mathcal{G}|\leq\binom{n-1}{\ell-1}-\binom{n-k-1}{\ell-1}$. In particular, when $k=\ell$, this yields $
\min\{|\mathcal F|,|\mathcal G|\}\le \binom{n-1}{k-1}$, and the bound is sharp, as witnessed by two full $1$-stars with the same center.  This motivates the natural problem of determining the maximum of $\min\{|\mathcal{F}|, |\mathcal{G}|\}$ under the condition that $\cap_{F\in\mathcal{F}\cup\mathcal{G}}F=\emptyset$. The problem was settled by M\"{o}rs \cite{Mors} in 1985 and by F\"{u}redi \cite{Furedi-1995} in 1994.
\begin{theorem}[\cite{Mors,Furedi-1995}]
	Let $n>k+\ell$. If  $\mathcal{F}\subseteq\binom{[n]}{k}$ and $\mathcal{G}\subseteq\binom{[n]}{\ell}$ are  cross $1$-intersecting families with $\cap_{F\in\mathcal{F}\cup\mathcal{G}}F=\emptyset$, then either $$|\mathcal{F}|\leq\binom{n-1}{k-1}-\binom{n-\ell-1}{k-1}+1\;\;\mbox{or}\;\;|\mathcal{G}|\leq\binom{n-1}{\ell-1}-\binom{n-k-1}{\ell-1}+1.$$ 
	Moreover, for $k,\ell\geq4$, if $|\mathcal{F}|=|\mathcal{H}(n,k,1)|$ and $|\mathcal{G}|=|\mathcal{H}(n,\ell,1)|$, then $\mathcal{F}=\mathcal{H}(\{x\},K,L)$ and  $\mathcal{G}=\mathcal{H}(\{x\},L,K)$ for some $x\in[n]$, $K\in\binom{[n]}{k+1}$ and  $L\in\binom{[n]}{\ell+1}$ with $x\in K\cap L$ and $|K\cap L|\geq2$.
\end{theorem}
Let us note that in \cite{Furedi-1995}, the author also  characterized the extremal families for $k\leq3$ or $\ell\leq3$. We refer the reader to \cite{Huang-2019,Frankl-Kupavskii-2020} for other results on the topic. In this paper, we establish this max-min theorem for all $t\geq2$. We remark that, to avoid lengthy calculations, we restrict our attention
to the uniform setting $\binom{[n]}k$, which already reveals the essence of
our method. Before presenting it, let us introduce the following families.
\begin{construction}\label{constructionh}
	Let $X\in\binom{[n]}{t}$ and $K,L\in\binom{[n]}{k+1}$ with $X\subseteq K\cap L$. Define
	\begin{equation*}
		\mathcal{H}(X,K,L)=\left\{F\in\binom{[n]}{k}:X\subseteq F,\;F\cap(L\setminus X)\neq\emptyset\right\}\cup\left\{K\setminus\{x\}:x\in X\right\}.
	\end{equation*}
\end{construction}
\begin{construction}
	Let $Z\in\binom{[n]}{t+2}$. Define
	\begin{equation*}
		\mathcal{A}(Z)=\left\{F\in\binom{[n]}{k}:|F\cap Z|\geq t+1\right\}.
	\end{equation*}
\end{construction}
We note that $\mathcal{H}(n,k,t)=\mathcal{H}([t],[k+1],[k+1])$ and $\mathcal{A}(n,k,t)=\mathcal{A}([t+2])$.
\begin{theorem}\label{thmmin-max}
	Let $k\geq t+2\geq4$ and $\mathcal{F},\mathcal{G}\subseteq \binom{[n]}{k}$ be cross $t$-intersecting with $|\cap_{F\in\mathcal{F}\cup\mathcal{G}}F|<t$. If $n\geq t+(k-t)\cdot\max\{30(t+2),2ek\}$, then 
	$$\min\{|\mathcal{F}|,  |\mathcal{G}|\}\leq\max\{|\mathcal{A}(n,k,t)|,  |\mathcal{H}(n,k,t)|\}.$$
	Moreover, if equality holds, then either 
	\begin{itemize}
		\item[\rm(i)]there exist $X\in\binom{[n]}{t}$ and $K,L\in\binom{[n]}{k+1}$ with $X\subseteq K\cap L$ and $|K\cap L|\geq t+2$ such that $\mathcal{F}=\mathcal{H}(X,K,L)$ and $\mathcal{G}=\mathcal{H}(X,L,K)$, or
		\item[\rm(ii)]there exists $Z\in\binom{[n]}{t+2}$ such that  $\mathcal{F}=\mathcal{G}=\mathcal{A}(Z)$.
	\end{itemize}
\end{theorem}
Let us note that, for $k\leq2t+1$, the families described in (ii) are precisely all those achieving the upper bound (see Lemma \ref{lemmasizeha}), while for $k\geq2t+2$, we do not attempt to determine which of
the two constructions, or whether both, is optimal. By applying the theorem above to $\mathcal{F}=\mathcal{G}$, one obtains Theorem \ref{HM} for large $n$. We also remark that if $k-t=\Theta(t)$, then the method gives a new proof of Theorem \ref{HM} in the range $n=\Omega((t+1)(k-t+1))$. In \cite{Balogh-Mubayi}, Balogh and Mubayi gave a short proof of Theorem \ref{HM} for $k\leq2t+1$ using the shifting technique and Theorem \ref{EKR}. 

In 2023, Frankl and Wang \cite{Frankl-Wang-2023} proved a product version of the Hilton--Milner theorem. The result states that, for $n\geq4k$, a pair $(\mathcal{F},\mathcal{G})$ of cross $1$-intersecting families in $\binom{[n]}{k}$ has product of sizes at most $|\mathcal{H}(n,k,1)|^2$. Very recently, Frankl and Wang \cite{Frankl-Wang-2026+} improved this result to \(n\ge 2k+1\) and \(k\ge8\). Clearly, the bound can be achieved by $\mathcal{F}=\mathcal{G}=\mathcal{H}(n,k,1)$. For general $t$, they obtained the following result.
\begin{theorem}[\cite{Frankl-Wang-2024}]\label{thmFW}
	Suppose $n\geq4(t+2)^2k^2$ and $k\geq5$. If $\mathcal{F},\mathcal{G}\subseteq \binom{[n]}{k}$ are cross $t$-intersecting with $|\cap\mathcal{F}|<t$ and $|\cap\mathcal{G}|<t$, then
	$$|\mathcal{F}||\mathcal{G}|\leq\max\{|\mathcal{A}(n,k,t)|^2,|\mathcal{H}(n,k,t)|^2\}.$$
\end{theorem}
 Our next main result provides an improved bound on $n$, and gives a characterization of extremal configurations. 
\begin{theorem}\label{thmcrosshm}
	Suppose $k\geq t+2\geq4$ and $n\geq t+\max\{15(t+2)^2(k-t),10k^2\}$. Let  $\mathcal{F},\mathcal{G}\subseteq \binom{[n]}{k}$ be cross $t$-intersecting families with $|\cap\mathcal{F}|<t$ and $|\cap\mathcal{G}|<t$. The following hold.
\begin{itemize}
	\item[\rm(i)]If $k\geq2t+2$, then $|\mathcal{F}||\mathcal{G}|\leq|\mathcal{H}(n,k,t)|^2$, with equality if and only if $\mathcal{F}=\mathcal{H}(X,K,L)$ and $\mathcal{G}=\mathcal{H}(X,L,K)$ for some   $X\in\binom{[n]}{t}$ and $K,L\in\binom{[n]}{k+1}$ with $X\subseteq K\cap L$ and $|K\cap L|\geq t+2$.
	\item[\rm(ii)]If $k\leq2t+1$, then $|\mathcal{F}||\mathcal{G}|\leq|\mathcal{A}(n,k,t)|^2$, with equality if and only if $\mathcal{F}=\mathcal{G}=\mathcal{A}(Z)$ for some $Z\in\binom{[n]}{t+2}$.
\end{itemize}	
\end{theorem}
\subsection{Stability via diversity}\label{subsecdiv}
As we have mentioned, the influential Frankl theorem characterizes large $1$-intersecting families with bounded maximum degree.  In order to state the theorem, we define
\begin{equation*}\label{familyl}
	\mathcal{L}(n,k,i):=\left\{F\in\binom{[n]}{k}:1\in F,\;F\cap[2,1+i]\neq\emptyset\right\}\cup\left\{F\in\binom{[n]}{k}:[2,1+i]\subseteq F\right\}
\end{equation*}
for $2\leq i\leq k$. Note that $\mathcal{L}(n,k,2)=\mathcal{A}(n,k,1)$ and $\mathcal{L}(n,k,k)=\mathcal{H}(n,k,1)$. 
\begin{theorem}[\cite{Frankl-1987}]\label{F}
	Suppose that $n>2k$, $2\leq i\leq k$ and $\mathcal{F}\subseteq\binom{[n]}{k}$ is $1$-intersecting with  $\Delta(\mathcal{F})\leq\Delta(\mathcal{L}(n,k,i))$. Then
	$$|\mathcal{F}|\leq|\mathcal{L}(n,k,i)|.$$
	Moreover, if equality holds, then either $\mathcal{F}\cong\mathcal{L}(n,k,i)$,  or $i=3$ and $\mathcal{F}\cong\mathcal{L}(n,k,2)$.
\end{theorem}
The theorem sharpens the Hilton--Milner theorem. To see this, note that if $\mathcal{F}\subseteq\binom{[n]}{k}$ is non-trivial  $1$-intersecting, then for every $x\in[n]$ there exists a set $F\in\mathcal{F}$ that does not contain $x$, and hence every member containing $x$ must intersect $F$, yielding $\Delta(\mathcal{F})\leq\binom{n-1}{k-1}-\binom{n-k-1}{k-1}=\Delta(\mathcal{H}(n,k,1))$. In addition, 
as shown by the theorem, a large  $1$-intersecting family admits a `popular' element and is therefore close to a full $1$-star.

 We now turn to another useful parameter revealing the phenomenon. For a family $\mathcal{F}\subseteq2^{[n]}$, its \emph{diversity}, denoted $\gamma(\mathcal{F})$, is defined to be the minimal number of sets from  $\mathcal{F}$ not containing an element with largest degree. Precisely, $$\gamma(\mathcal{F}):=|\mathcal{F}|-\Delta(\mathcal{F}).$$
The diversity \(\gamma(\mathcal F)\) can be regarded as the distance from
\(\mathcal F\) to a trivial family. It is easy to see that a trivial family has diversity zero, and $\mathcal{H}(n,k,1)$ has diversity one as $[2,k+1]$ is the unique member that does not contain $1$. In 2018, Kupavskii and Zakharov \cite{Kupavskii-Zakharov-2018} established a diversity version of Frankl's theorem, with an elegant proof involving matchings in regular  bipartite graphs. In 2021, Frankl and Kupavskii \cite{Frankl-Kupavskii-2021} proved the following  cross-intersection version of the Kupavskii--Zakharov theorem, which is also a sharp stability result for cross-intersecting families in terms of diversity. 
\begin{theorem}[\cite{Frankl-Kupavskii-2021}]\label{thmFK}
	Suppose $n\geq k+\ell$ and  $\mathcal{F}\subseteq\binom{[n]}{k}$ and $\mathcal{G}\subseteq\binom{[n]}{\ell}$ are cross $1$-intersecting. Suppose that 
	\begin{align*}
		|\mathcal{F}|&\geq\binom{n-1}{k-1}-\binom{n-v-1}{k-1}+\binom{n-u-1}{n-k-1}\;\mbox{and}\\
		|\mathcal{G}|&\geq\binom{n-1}{\ell-1}-\binom{n-u-1}{\ell-1}+\binom{n-v-1}{n-\ell-1}
	\end{align*}
	for some real $3\leq u\leq k,\;3\leq v\leq \ell$ and, moreover, that at least one of the displayed 
	inequalities is strict. Then   $$\gamma(\mathcal{F})<\binom{n-u-1}{n-k-1}\;\;\mbox{and}\;\;\;\gamma(\mathcal{G})<\binom{n-v-1}{n-\ell-1}.$$
	Moreover, both families have the same (unique) element of the largest degree.
\end{theorem}
Another fundamental problem concerning the diversity is to determine the maximum possible diversity of a $1$-intersecting family. This problem has attracted considerable attention in the past decade. It has been shown that, for $n\geq f(k)$ every $1$-intersecting family in $\binom{[n]}{k}$ has diversity at most $
\gamma(\mathcal{A}(n,k,1))=\binom{n-3}{k-2}$. 
It was shown by Lemons and Palmer \cite{Lemons-Palmer} that $f(k)\leq6k^3$. This bound was subsequently improved in \cite{Frankl-2017,Kupavskii-2018,Frankl-2020}, and  the previous best bound $f(k)\leq36k$ is due to Frankl and Wang \cite{Frankl-Wang-2024-diversity}. We also note that counterexamples constructed by Huang \cite{Huang-2019} yield $f(k)\geq(2+\sqrt{3})k$.

In this paper, we establish corresponding results for all $t\geq2$. For a family $\mathcal{F}$ of subsets of $[n]$, we define its \emph{$t$-diversity} by
\begin{equation*}
	\gamma_t(\mathcal{F}):=|\mathcal{F}|-\Delta_t(\mathcal{F}).
\end{equation*}
This parameter measures the distance from a family to the full $t$-star centered at its most popular $t$-subset. It is natural to ask which $t$-intersecting family in $\binom{[n]}{k}$ has largest $t$-diversity. A candidate is $\mathcal{A}(n,k,t)$, for which
$$\gamma_t(\mathcal{A}(n,k,t))= t\binom{n-t-2}{k-t-1}.$$
Our next main result is a max-min theorem for the $t$-diversities of cross $t$-intersecting families. As a direct corollary, we settle the problem for large $n$.
\begin{theorem}\label{thmmaxdiv}
	Let $k\geq t+2\geq4$ and $n\geq t+(k-t)\cdot\max\{18(t+2),2ek\}$. If  $\mathcal{F},\mathcal{G}\subseteq \binom{[n]}{k}$ are cross $t$-intersecting, then 
	$$\min\{\gamma_t(\mathcal{F}),\;  \gamma_t(\mathcal{G})\}\leq\gamma_t(\mathcal{A}(n,k,t)).$$
	If equality holds, then  $\mathcal{F}, \mathcal{G}\subseteq\mathcal{A}(Z)$ for some $Z\in\binom{[n]}{t+2}$.
\end{theorem}
\begin{corollary}\label{corothmmaxdiv}
	Let $k\geq t+2\geq4$ and $n\geq t+(k-t)\cdot\max\{18(t+2),2ek\}$. If  $\mathcal{F}\subseteq \binom{[n]}{k}$ is  $t$-intersecting, then $$\gamma_t(\mathcal{F})\leq\gamma_t(\mathcal{A}(n,k,t)).$$ If equality holds, then  $\mathcal{F}\subseteq\mathcal{A}(Z)$ for some $Z\in\binom{[n]}{t+2}$.
\end{corollary}
To present our remaining results, let us introduce the following construction.
\begin{construction}
Let $X\in\binom{[n]}{t}$, $U\in\binom{[n]}{t+u}$ and $V\in\binom{[n]}{t+v}$ with $X\subseteq U\cap V$. Define
\begin{align*}
\mathcal{L}(X,U,V)=&\left\{F\in\binom{[n]}{k}:X\subseteq F,\;F\cap(V\setminus X)\neq\emptyset\right\}\\
&\cup\left\{F\in\binom{[n]}{k}:F\cap U=U\setminus\{x\}
\text{ for some } x\in X\right\}.
\end{align*} 
\end{construction}

Set $\mathcal{L}(n,k,t,u,v)=\mathcal{L}([t],[t+u],[t+v])$. A routine computation gives $$|\mathcal{L}(n,k,t,u,v)|=\binom{n-t}{k-t}-\binom{n-v-t}{k-t}+t\binom{n-u-t}{k-u-t+1}.$$
In particular, $\mathcal{L}(X,U,V)=\mathcal{H}(X,U,V)$ when $U, V\in\binom{[n]}{k+1}$, where \(\mathcal H\) is the family defined in Construction \ref{constructionh}. Note that
$$\gamma_t(\mathcal{L}(n,k,t,u,v))=t\binom{n-u-t}{k-u-t+1}.$$

In terms of $t$-diversity, Theorem \ref{thmmin-max} implies that, if $\mathcal{F},\mathcal{G}\subseteq\binom{[n]}{k}$ are cross $t$-intersecting with $\min\{|\mathcal{F}|, |\mathcal{G}|\}>\max\{|\mathcal{H}(n,k,t)|,|\mathcal{A}(n,k,t)|\}$, then  $\mathcal{F}$ and $\mathcal{G}$ are trivial, that is,  $\gamma_t(\mathcal{F})=\gamma_t(\mathcal{G})=0$. The next main result refines this implication by giving a stability result for large $n$ depending on $k$ and $t$. It may also be viewed as a cross $t$-intersection version of Theorem \ref{thmFK}.
\begin{theorem}\label{thmdiv}
	Let $k\geq t+2\geq4$, $n\geq7(k-t+1)k^2$, and let  $\mathcal{F},\mathcal{G}\subseteq\binom{[n]}{k}$ be cross $t$-intersecting. Suppose that 
	\begin{align*}
		|\mathcal{F}|\geq|\mathcal{L}(n,k,t,u,v)|&\;\;\mbox{and}\;\;\;|\mathcal{G}|\geq|\mathcal{L}(n,k,t,v,u)|\\
		\intertext{for some integers $u,v\in[3,k-t+1]$ with $(u,v)\neq(3,3)$. Then}
		\gamma_t(\mathcal{F})\leq\gamma_t(\mathcal{L}(n,k,t,u,v))&\;\;\mbox{and}\;\;\;\gamma_t(\mathcal{G})\leq\gamma_t(\mathcal{L}(n,k,t,v,u)),
	\end{align*}
 or $u,v\leq t+2$ and $\mathcal{F}, \mathcal{G}\subseteq\mathcal{A}(Z)$ for some $Z\in\binom{[n]}{t+2}$. Moreover, in the former case, if $\gamma_t(\mathcal{F})=\gamma_t(\mathcal{L}(n,k,t,u,v))$ or $\gamma_t(\mathcal{G})=\gamma_t(\mathcal{L}(n,k,t,v,u))$, then $\mathcal{F}=\mathcal{L}(X,U,V)$ and $\mathcal{G}=\mathcal{L}(X,V,U)$ for some  $X\in\binom{[n]}{t}$, $U\in\binom{[n]}{t+u}$ and $V\in\binom{[n]}{t+v}$ with $X\subseteq U\cap V$ and $|U\cap V|\geq t+2$.
\end{theorem}
Let us note that the remaining case $(u,v)=(3,3)$ is treated separately in Proposition~\ref{prop}. We also remark that Theorem \ref{thmdiv} fails for $n=O((k-t)t^2)$ (see Section \ref{secrmk} for details). So for $k=O(t)$, the order of the present lower bound on $n$ is best possible.

Our argument also yields stability results for $t$-intersecting families. In 2019, Ellis, Keller and Lifshitz \cite{Ellis-2019} established the following  stability version for Theorem \ref{EKR}. 
 \begin{theorem}[\cite{Ellis-2019}]\label{thmEKL}
 For any $t\in\mathbb{N}$ and $\eta>0$, there exists $\delta_0=\delta(t,\eta)$ such that the following holds. Let $n,k\in\mathbb{N}$ with $k\leq(\frac{1}{t+1}-\eta)n$, and let $d\in\mathbb{N}$. If  $\mathcal{F}\subseteq\binom{[n]}{k}$ is  $t$-intersecting with
 \begin{align*}
 |\mathcal{F}|>\max\left\{(1-\delta_0)\binom{n-t}{k-t},\binom{n-t}{k-t}-\binom{n-t-d}{k-t}+(2^t-1)\binom{n-t-d}{k-t-d+1}\right\},
 \end{align*}
 then there exists a $t$-subset $X$ such that
 \begin{equation*}
|\mathcal{F}\setminus\mathcal{F}[X]|\leq(2^t-1)\binom{n-t-d}{k-t-d+1}.
 \end{equation*}
 \end{theorem}
 The approach presented in \cite{Ellis-2019}, which employs isoperimetric inequalities for subsets of the hypercube, is deep and has proved to be a fruitful  framework for studying stability results in extremal set theory. In that paper, they proposed the following conjecture as a potential strengthening of Theorem \ref{thmEKL}.
 \begin{conjecture}[\cite{Ellis-2019}]
	Let $n\geq(t+1)(k-t+1)$ and let $d\in\mathbb{N}$.  If $\mathcal{F}\subseteq\binom{[n]}{k}$ is $t$-intersecting with
\begin{align*}
	|\mathcal{F}|\geq
	\max\Bigg\{(t+2)\binom{n-t-2}{k-t-1}&-(t+1)\binom{n-t-2}{k-t-2},\\&\binom{n-t}{k-t}-\binom{n-t-d}{k-t}+t\binom{n-t-d}{k-t-d+1}\Bigg\},
\end{align*}
 then there exists a $t$-subset $X$ such that $|\mathcal{F}\setminus\mathcal{F}[X]|\leq t\binom{n-t-d}{k-t-d+1}$.
 \end{conjecture}
There seems to be a minor slip in the printed formulation, as the expected strengthening of Theorem \ref{thmEKL} should concern the structure of large $t$-intersecting families whose size is larger than $|\mathcal{A}(n,k,t)|=(t+2)\binom{n-t-1}{k-t-1}-(t+1)\binom{n-t-2}{k-t-2}$ and close to $|\mathcal{H}(n,k,t)|$. The former restriction is natural, since the family $\mathcal{A}(n,k,t)$ has both large $t$-diversity and large size. In fact, for $3\leq d\leq t+2$, this family gives a counterexample to the conjecture in its  printed form.

Our next result proves the expected form of the conjecture for large $n$, in a stronger form. 
\begin{theorem}\label{thmstability-t-div}
	Let $k\geq t+2\geq4$, $n\geq7(k-t+1)k^2$ and $d\geq2$. If $\mathcal{F}\subseteq\binom{[n]}{k}$ is a $t$-intersecting family with
	\begin{align*}
		|\mathcal{F}|&\geq\binom{n-t}{k-t}-\binom{n-t-d}{k-t}+t\binom{n-t-d}{k-t-d+1},
	\end{align*}
	then either $\gamma_t(\mathcal{F})\leq t\binom{n-t-d}{k-t-d+1}$, that is, there exists a $t$-subset $X$ such that $|\mathcal{F}\setminus\mathcal{F}[X]|\leq t\binom{n-t-d}{k-t-d+1}$, or $d\leq t+2$ and $\mathcal{F}\subseteq\mathcal{A}(Z)$ for some $Z\in\binom{[n]}{t+2}$.  Moreover, if $2\leq d\leq k-t+1$, then in the former case, equality holds if and only if $\mathcal{F}\cong\mathcal{L}(n,k,t,d,d)$.
\end{theorem}
 Let us note that the following hold for $n\geq(t+1)(k-t)+3$.
 \begin{align*}
 	|\mathcal{L}(n,k,t,3,3)|&<|\mathcal{L}(n,k,t,4,4)|<\cdots<|\mathcal{L}(n,k,t,k-t+1,k-t+1)|.\\
 	\gamma_t(\mathcal{L}(n,k,t,3,3))&>\gamma_t(\mathcal{L}(n,k,t,4,4))>\cdots>\gamma_t(\mathcal{L}(n,k,t,k-t+1,k-t+1)).
 \end{align*}
  The theorem above can be interpreted as a stability result for the Erd\H{o}s--Ko--Rado theorem for large $n$, that is, every $t$-intersecting family of size at least  $|\mathcal{L}(n,k,t,d,d)|$ can be made a subfamily of a full $t$-star by removing at most  $t\binom{n-t-d}{k-t-d+1}$ sets, unless it is isomorphic to a subfamily of $\mathcal{A}(n,k,t)$. It may also be regarded as a $t$-intersection analogue (for large $n$) of Frankl's theorem (Theorem \ref{F}) as well as its diversity version due to Kupavskii and Zakharov \cite{Kupavskii-Zakharov-2018}. 
\subsection{Organization}
The rest of this paper is organized as follows. In the next section, we present a key ingredient of our approach: an algorithm that produces sequences of $t$-covers. Several important properties of the resulting families are then derived, and as a first application we prove Theorem \ref{thmcrossekr}. Sections \ref{secfinstru1} and \ref{secfinstru2} focus on cross-intersection problems for classical set systems. In Section \ref{secfinstru1}, we first collect several frequently used inequalities, and then prove the key lemma for the $t$-cover method, Lemma \ref{lemmakey}. We also establish a simple but crucial observation (Lemma \ref{lemmamaximal}) concerning the initial fingerprint in this setting, which enables a more effective application of our argument. We then prove two structural lemmas (Lemmas \ref{lemmafin-stru1} and \ref{lemmafin-stru2}), followed by the proofs of Theorems \ref{thmmin-max} and \ref{thmcrosshm}. In Section \ref{secfinstru2}, we prove Theorems \ref{thmmaxdiv}, \ref{thmdiv} and \ref{thmstability-t-div} on the $t$-diversity. Finally, Section \ref{secrmk} presents some applications of Theorem \ref{thmcrossekr} and concludes with several remarks.
\section{Fingerprints and the iteration procedure}
Let us begin with several notations. Let $\Omega$ be a finite set. For simplicity, $\binom{\Omega}{\leq s}$ stands for the set of all subsets of $\Omega$ with size at most $s$. For families $\mathcal{A}$ and $\mathcal{S}$ of subsets of $\Omega$ and a subset $X$ of $\Omega$, we write
\begin{align*}
	\mathcal{A}(X)&:=\{A\setminus X:A\in\mathcal{A},X\subseteq A\}.\\
	\mathcal{A}[\mathcal{S}]&:=\left\{A:A\in\mathcal{A}, S\subseteq A\;\mbox{for some}\;S\in\mathcal{S}\right\}.
\end{align*}
We note that $\mathcal{A}(X)$ is a family of subsets of $\Omega\setminus X$, and has the same size as   $\mathcal{A}[X]$, as defined in (\ref{equlink}). Also note that $\mathcal{A}[\mathcal{S}]=\bigcup_{S\in\mathcal{S}}\mathcal{A}[S]$. We set $\mathcal{A}[\emptyset]=\emptyset$. Here is a simple yet useful estimate. \begin{equation}\label{equltrt}
	\binom{a-c}{b-c}\leq(a-b+1)^{b-c}\;\;\mbox{for}\;\;a\geq b\geq c.
\end{equation}
This can be easily verified as  $\frac{y}{x}<\frac{y-1}{x-1}$ for $y>x>1$. 

Let us recall an inductive property of $r$-spreadness (see e.g. \cite{Kupavskii-2026}). 
We present the proof for completeness.
\begin{lemma}\label{lemmaspreadness-size}
	If $\mathcal{F}\subseteq\binom{\Omega}{k}$ and $|\mathcal{F}|>r^k$ for some $r>1$, then there is a subset $X$ of $\Omega$ with size smaller than $k$ such that $\mathcal{F}(X)$ is $r$-spread.
\end{lemma}
\begin{proof}
	If $\mathcal{F}$ itself is $r$-spread, then pick $X=\emptyset$. Suppose it is not $r$-spread, then we choose $X$ to be an  inclusion-maximal subset $X$ with $|\mathcal{F}(X)|>r^{-|X|}|\mathcal{F}|$. Since $|\mathcal{F}|>r^k$, we have $|\mathcal{F}(K)|\leq1<r^{-k}|\mathcal{F}|$ for each $k$-subset $K$, and hence such a subset $X$ has size smaller than $k$. Now for all $\emptyset\neq Y\subseteq\Omega\setminus X$, it follows that 
	$$|(\mathcal{F}(X))(Y)|=|\mathcal{F}(X\cup Y)|\leq r^{-(|X\cup Y|)}|\mathcal{F}|<r^{-|Y|}|\mathcal{F}(X)|.$$ Hence $\mathcal{F}(X)$ is $r$-spread.
\end{proof}

Let $\mathcal{F}$ be a family of subsets of $\Omega$.  A \emph{$t$-cover} of $\mathcal{F}$ is a set that intersects every member of $\mathcal{F}$ in at least $t$ elements. If such a set exists, then the \emph{$t$-covering number} of $\mathcal{F}$, denoted $\tau_t(\mathcal{F})$, is defined to be the minimum size of a $t$-cover. In particular, if $\mathcal{S}$ and $\mathcal{T}$ are cross $t$-intersecting, then $\mathcal{S}$ forms a set of $t$-covers of $\mathcal{T}$, and vice versa. For a family $\mathcal{F}$, we say that a $t$-cover $T$ is \emph{minimal} if any of its proper subset is not a $t$-cover, that is, for all $T_0\subsetneqq T$, there exists $F\in\mathcal{F}$ such that $|F\cap T_0|<t$. Let us introduce the notion of a fingerprint for cross $t$-intersecting families.
\begin{lemma}\label{lemmafingerprintdef}
	Let $\Omega$ be a finite set and $(\mathcal{S},\mathcal{T})$ be a pair of cross $t$-intersecting families of subsets of $\Omega$. Then there is a pair of cross $t$-intersecting families $(\mathcal{S}^*,\mathcal{T}^*)$, called a \emph{fingerprint} of $(\mathcal{S},\mathcal{T})$, such that
	\begin{itemize}
		\item[\rm(i)]Every member of $\mathcal{S}^*$ {\rm(}of $\mathcal{T}^*${\rm)} is contained in some of $\mathcal{S}$ {\rm(}of $\mathcal{T}${\rm)};
		\item[\rm(ii)]Every member of $\mathcal{S}$ {\rm(}of $\mathcal{T}${\rm)} contains some of $\mathcal{S}^*$ {\rm(}of $\mathcal{T}^*${\rm)};
		\item[\rm(iii)]Every member of $\mathcal{S}^*$ {\rm(}of $\mathcal{T}^*${\rm)} is a minimal $t$-cover of $\mathcal{T}^*$ {\rm(}of $\mathcal{S}^*${\rm)}.
	\end{itemize}
\end{lemma}
\begin{proof}
	Write $u=|\mathcal{S}|$ and $v=|\mathcal{T}|$ for short. Fix orderings $(S_1,S_2,\ldots,S_u)$ and $(T_1,T_2,\ldots,T_v)$ of members of $\mathcal{S}$ and $\mathcal{T}$, respectively. We first replace each $S_i$ ($i=1,2,\ldots,u$) by an inclusion-minimal subset $S_i^*\subseteq S_i$ that $t$-intersects every set in $\mathcal{T}$, that is, $S_i^*$ intersects every set from $\mathcal{T}$ in at least $t$ elements, and the property does not hold for any of its proper subset. Next, we apply the same procedure to $(T_1,T_2,\ldots,T_v)$ and $(S_1^*,S_2^*,\ldots,S_u^*)$, namely, replace each $T_j$ ($j=1,2,\ldots,v$) by an inclusion-minimal subset $T_j^*\subseteq T_j$ that $t$-intersects each $S_i^*$. Finally, one desired pair of families is obtained by deleting repeated sets in the resulting sequences $(S_1^*,S_2^*,\ldots,S_u^*)$ and  $(T_1^*,T_2^*,\ldots,T_v^*)$.
\end{proof}
As the heart of our argument, the following procedure iteratively produces a sequence of families reflecting structural information about a given pair of cross $t$-intersecting families. This may be regarded as a cross-intersection analogue of the peeling procedure  \cite{Kupavskii-Zakharov-2024}, which was used to  estimate the size of a $t$-intersecting family. \vspace{1em}

\noindent{\bf Algorithm.\label{algo}\;}Suppose that a finite set $\Omega$, integers $q$ and $t$ with $q\geq t$, a pair of cross $t$-intersecting families $(\mathcal{S},\mathcal{T})$ with  $\mathcal{S},\mathcal{T}\subseteq\binom{\Omega}{\leq q}$, and a fingerprint $(\mathcal{S}_0,\mathcal{T}_0)$ of $(\mathcal{S},\mathcal{T})$ are given. Set $N=0$. For $i=0,1,\ldots,q-t$, do the following:

\begin{itemize}
	\item[(a)]Set  $\mathcal{X}_{i}=\mathcal{S}_{i}\cap\binom{\Omega}{q-i}$ and $\mathcal{Y}_{i}=\mathcal{T}_{i}\cap\binom{\Omega}{q-i}$. If $\mathcal{S}_i=\mathcal{X}_i$ or $\mathcal{T}_i=\mathcal{Y}_i$,
	terminate the procedure and set $N=i$.
	\item[(b)]Find a fingerprint of the pair  $(\mathcal{S}_{i}\setminus\mathcal{X}_{i},\mathcal{T}_{i}\setminus\mathcal{Y}_{i})$, and denote it by $(\mathcal{S}_{i+1},\mathcal{T}_{i+1})$.
\end{itemize}

\noindent Output $N$ and $(\mathcal{S}_{i},\mathcal{T}_{i},\mathcal{X}_{i},\mathcal{Y}_i)$ for $i\leq N$.\vspace{1em}

As an example, let us consider the pair
\begin{align}
\mathcal{S}&=\left\{F\in\binom{[n]}{k}:[t]\subseteq F,\;F\cap[t+1,k+1]\neq\emptyset\right\}\;\;\mbox{and}\label{equexam1}\\
\mathcal{T}&=\left\{F\in\binom{[n]}{k}:[t]\subseteq F\right\}\cup\{[k+1]\setminus\{i\}:i\in[t]\}.\label{equexam2}
\end{align}
We set $k\geq t+2$, and pick $q:=k$, and start with the fingerprint $(\mathcal{S}_0,\mathcal{T}_0)$, where
\begin{equation}\label{equexam3}
\mathcal{S}_0=\{[t]\cup\{i\}:i\in[t+1,k+1]\}\;\;\mbox{and}\;\;\mathcal{T}_0=\{[t]\}\cup\{[k+1]\setminus\{i\}:i\in[t]\}.
\end{equation}
Then the first round (see Fig. \ref{figure1})  yields  $\mathcal{X}_0=\emptyset,\mathcal{Y}_0=\mathcal{T}_0\setminus\{[t]\}$, and  $\mathcal{S}_1=\mathcal{T}_1=\{[t]\}$. It terminates with outputting these together with  $\mathcal{X}_i=\mathcal{Y}_i=\emptyset$ and $\mathcal{S}_{i+1}=\mathcal{T}_{i+1}=\{[t]\}$ for  $1\leq i\leq q-t-1$, and $N=q-t$.
\begin{figure}[H]	\centering
\begin{tikzpicture}
	\newcommand{\petal}[1][]{
		\draw[line width=0.7pt,#1]
		plot[domain=-120:0,samples=150] ({1+cos(\x)},{sin(\x)})
		plot[domain=0:98,samples=150] ({1.5+0.5*cos(\x)},{0.5*sin(\x)});
	}
				\newcommand{\petaltypetwo}[1][]{
		\draw[line width=0.7pt,#1]
		plot[domain=90:128,samples=150] ({1.5+0.93*cos(\x)},{0.5*sin(\x)});
	}
	\newcommand{\petaltypethree}[1][]{
		\draw[#1, line width=0.7pt] 
		plot[domain=-80:0,samples=150] ({1+cos(\x)},{sin(\x)})
		plot[domain=0:98,samples=150] ({1.5+0.5*cos(\x)},{0.5*sin(\x)});
	}
	\newcommand{\petaltypefour}[1][]{
		\draw[#1, line width=0.7pt]
		plot[domain=-120:0,samples=150] ({1+cos(\x)},{sin(\x)})
		plot[domain=0:98,samples=150] ({1.5+0.5*cos(\x)},{0.5*sin(\x)})
		plot[domain=-120:-81.5,samples=150] ({cos(60)*(1+cos(\x)) - sin(60)*sin(\x)}, 
		{sin(60)*(1+cos(\x)) + cos(60)*sin(\x)})
		plot[domain=-60:0,samples=200] ({cos(\x)}, {sin(\x)});
	}
	\newcommand{\flower}{
		\foreach \a in {0,60,120,180,240,300}{
			\begin{scope}[rotate=\a]
	\petal
			\end{scope}
		}
		\draw[black,line width=0.7pt,dash pattern=on 2pt off 1.3pt] plot[domain=-1:60, samples=200] ({cos(\x)}, {sin(\x)});
		\draw[black,line width=0.7pt,dash pattern=on 2pt off 1.3pt] plot[domain=0:1, samples=200] (\x, 0);
		\draw[black,line width=0.7pt,dash pattern=on 2pt off 1.3pt] (0,0) -- (0.5, {tan(60)*0.5});}
\begin{scope}[shift={(-2.8,2.2)}]
	\foreach \b in {60}{
	\begin{scope}[rotate=\b]
		\petaltypetwo[draw=tonglv,line width=2pt];
	\end{scope}
}
\foreach \c in {120}{
	\begin{scope}[rotate=\c]
		\petaltypethree;
	\end{scope}
}
\foreach \a in {0,60,180,240,300}{
	\begin{scope}[rotate=\a]
		\ifnum\a=60
		\petal[draw=tonglv,line width=2pt]
		\else
		\petal[draw=black]
		\fi
	\end{scope}
}
\draw[tonglv, line width=2pt] plot[domain=81.2:360, samples=200] ({cos(\x)}, {sin(\x)});
\draw[black,line width=0.7pt,dashed] plot[domain=-1:83, samples=200] ({cos(\x)}, {sin(\x)});
\node at (2.1,1.6) {$\color{tonglv}S\in\mathcal{S}_0$};
\node at (0,0) {$[t]$};
\end{scope}
\node at (1.5,0) {$\mathcal{S}_1=\mathcal{T}_1=\{[t]\}$};
\node at (-3,0) {$\mathcal{S}_0$};
\node at (7,0) {$\mathcal{T}_0\setminus\mathcal{Y}_0=\{[t]\}$};
\begin{scope}[scale=0.4,shift={(16,5)}]
	\begin{scope}[shift={(180:6)}, rotate=120]
		\flower
	\end{scope}
	\begin{scope}[shift={(120:4)}, rotate=60]
		\flower
	\end{scope}
	\begin{scope}[shift={(0:6)}, rotate=180]
		\flower
	\end{scope}
\draw[line width=0.7pt] plot[domain=0:360, samples=200] ({1.2*cos(\x)}, {-2+1.2*sin(\x)});
\foreach \x in {50,60,70}{
	\fill ({3.5*1.3*cos(\x)}, {3.5*sin(\x)}) circle (2pt);
}
\node at (1.8,1.6) {$\mathcal{Y}_0$}; 	 
\end{scope}
\begin{scope}[scale=0.7,shift={(2.3,1.3)}]
\draw[line width=0.7pt] plot[domain=0:360, samples=200] ({0.7*cos(\x)}, {0.7*sin(\x)});
\end{scope}
\draw[->, line width=1pt] (5,1) to[bend left=30] (3,0.5);
\draw[->, line width=1pt] (-1.1,1.1) to[bend left=-20] (0.5,0.6);
\end{tikzpicture}
{\caption{The first round of the algorithm for the families in (\ref{equexam1})--(\ref{equexam3}).}\label{figure1}}
\end{figure}
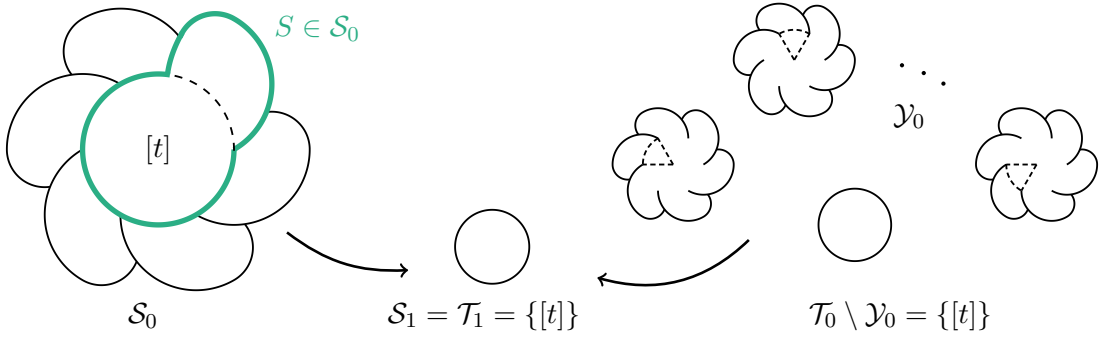
The following lemma records the basic properties of the output of the algorithm. It gives a decomposition of the input families into successive layers and a terminal fingerprint, together with estimates on the sizes of these layers. These properties serve as an important ingredient in our approach.
\begin{lemma}\label{lemmafingerprintproperty}
The following hold for $i=0,1,\ldots,N$.
	\begin{itemize}
		\item[\rm(i)] $\mathcal{S}_i, \mathcal{T}_i\subseteq\binom{\Omega}{\leq(q-i)}$ and $\mathcal{X}_{i}, \mathcal{Y}_i\subseteq\binom{\Omega}{q-i}$, and all of them are antichains, that is, no member is contained in another.
		\item[\rm(ii)]$\mathcal{S}[\mathcal{S}_{i-1}]\subseteq\mathcal{S}[\mathcal{X}_{i-1}]\cup\mathcal{S}[\mathcal{S}_i]$ and $\mathcal{T}[\mathcal{T}_{i-1}]\subseteq\mathcal{T}[\mathcal{Y}_{i-1}]\cup\mathcal{T}[\mathcal{T}_i]$ for $i\geq1$.
		\item[\rm(iii)]$\mathcal{S}\subseteq\left(\bigcup_{j=0}^{i-1}\mathcal{S}[\mathcal{X}_j]\right)\cup\mathcal{S}[\mathcal{S}_i]$ and $\mathcal{T}\subseteq\left(\bigcup_{j=0}^{i-1}\mathcal{T}[\mathcal{Y}_j]\right)\cup\mathcal{T}[\mathcal{T}_i]$ for $i\geq1$.
		\item[\rm(iv)] If $X$ is a subset of $\Omega$ and $\mathcal{L}$ is a subfamily of $\mathcal{S}_i$ or $\mathcal{T}_i$ such that $\mathcal{L}(X)$ is $r$-spread and $|\mathcal{L}(X)|\geq2$, then $r\leq q-t-i+1$.
		\item[\rm(v)]If $\mathcal{H}$ is a $(q-a)$-uniform subfamily of $\mathcal{S}_i$ or $\mathcal{T}_i$, then $|\mathcal{H}[Z]|\leq(q-t-i+1)^{q-t-a}$ for each $t$-subset $Z$ of $\Omega$. 
		\item[\rm(vi)]$\max\{|\mathcal{X}_i|,|\mathcal{Y}_i|\}\leq\binom{q-i}{t}(q-t-i+1)^{q-t-i}$.
	\end{itemize}
\end{lemma}
\begin{proof}
	In each point we give the proof of the statement for $\mathcal{S}_i$ or $\mathcal{X}_i$; the argument for $\mathcal{T}_i$ or $\mathcal{Y}_i$ is identical and is therefore omitted. 
	
	(i) By the choice of $\mathcal{S}_i$ and $\mathcal{X}_i$, it is clear  that $\mathcal{S}_i\subseteq\binom{\Omega}{\leq(q-i)}$ and $\mathcal{X}_i\in\binom{\Omega}{q-i}$. To see that $\mathcal{S}_i$ is an antichain, just note that each of its members is a minimal $t$-cover of $\mathcal{T}_i$, and so no $S_i\in\mathcal{S}_i$ is contained in another one. The same holds for $\mathcal{X}_i$ as  $\mathcal{X}_i\subseteq\mathcal{S}_i$.
	
	(ii) This follows directly from $\mathcal{S}_{i-1}=\mathcal{X}_{i-1}\cup(\mathcal{S}_{i-1}\setminus\mathcal{X}_{i-1})$ and that every set in $\mathcal{S}_{i}$ is contained in some of $\mathcal{S}_{i-1}\setminus\mathcal{X}_{i-1}$.
	
	(iii) This follows by using (ii) repeatedly.
	
	(iv) Since $|\mathcal{L}(X)|\geq2$, the set $X$ is a proper subset of some member of $\mathcal{S}_i$. Then there exists $T\in\mathcal{T}_i$ with $|X\cap T|=:s<t$, and then  $\mathcal{L}[X]=\bigcup_{Z\in\binom{T\setminus X}{t-s}}\mathcal{L}[X\cup Z]$ as $\mathcal{L}$ and $\mathcal{T}_i$ are cross $t$-intersecting. By picking a $(t-s)$-subset $Z_0$ maximizing $|\mathcal{L}[X\cup Z_0]|$, we obtain that 
	$$|\mathcal{L}[X]|\leq\binom{q-i-s}{t-s}|\mathcal{L}[X\cup Z_0]|\leq(q-t-i+1)^{t-s}\cdot r^{-(t-s)}|\mathcal{L}[X]|,$$
	where in the second step we used (\ref{equltrt}) and the $r$-spreadness of $\mathcal{L}(X)$. Thus $r\leq q-t-i+1$.
	
(v) To the contrary, assume that $|\mathcal{H}(Z)|>(q-t-i+1)^{q-t-a}$ for some $Z\in\binom{\Omega}{t}$,  then there is an $r>q-t-i+1\geq1$ such that  $|\mathcal{H}(Z)|>r^{q-t-a}$. By Lemma \ref{lemmaspreadness-size}, there exists a subset $W$ of $\Omega\setminus Z$ with $|W|<q-t-a$ and $\mathcal{H}(Z\cup W)=(\mathcal{H}(Z))(W)$ is $r$-spread. Note that $|Z\cup W|=|Z|+|W|<q-a$, then $Z\cup W\subsetneqq F$ for some $F\in\mathcal{H}[Z\cup W]$, and so $$1=|(\mathcal{H}(Z\cup W))(F)|\leq r^{-(|F|-|Z\cup W|)}|\mathcal{H}(Z\cup W)|\leq r^{-1}|\mathcal{H}(Z\cup W)|,$$ which gives $|\mathcal{H}(Z\cup W)|\geq\lceil r\rceil\geq2$. This contradicts (iv) by identifying $\mathcal{L}$ with $\mathcal{H}$ and $X$ with $Z\cup W$. 

(vi) Pick an arbitrary set, say $T$ in $\mathcal{T}_i$. Then the cross $t$-intersection between $\mathcal{X}_i$ and $\mathcal{T}_i$ gives $|\mathcal{X}_i|\leq\sum_{Z\in\binom{T}{t}}|\mathcal{X}_i(Z)|$, and then the desired upper bound on $|\mathcal{X}_i|$ follows from (v) and $|T|\leq q-i$.
\end{proof}	
\begin{lemma}\label{lemmarqt}
	Suppose $0<\varepsilon<1$ and $r,q,t\geq1$ with $q\geq t+1$ and  $\varepsilon r\geq eq$. Set $T_{j}=\binom{j+t}{t}(j+1)^j/r^j$ for $1\leq j\leq q-t$, and set 
	$$\varphi(r,q,t,m)=\sum_{i=0}^{m}\binom{q-i}{t}\left(\frac{q-t-i+1}{r}\right)^{q-t-i},\;m=0,1,\ldots,q-t-1.$$
	Then $T_{j+1}\leq\varepsilon T_j$ for $j\leq q-t-1$, and $$\varphi(r,q,t,m)<T_{q-t-m}/(1-\varepsilon)\leq2(t+1)\varepsilon^{q-t-m-1}/((1-\varepsilon)r).$$
\end{lemma}
\begin{proof}
First, we have  $$\frac{T_{j+1}}{T_j}=\frac{\binom{j+t+1}{t}(j+2)^{j+1}}{r\binom{j+t}{t}(j+1)^j}=\frac{j+t+1}{r}\cdot\left(1+\frac{1}{j+1}\right)^{j+1}<\frac{q}{r}\cdot e\leq\varepsilon$$
for $1\leq j\leq q-t-1$. Note that the expression above can be rewritten as $\varphi(r,q,t,m)=\sum_{j=q-t-m}^{q-t}T_j$. Then
\begin{align*}
\varphi(r,q,t,m)&<T_{q-t-m}(1+\varepsilon+\cdots+\varepsilon^{m})<T_{q-t-m}/(1-\varepsilon)\\
&<T_1\cdot\varepsilon^{q-t-m-1}/(1-\varepsilon)=2(t+1)\varepsilon^{q-t-m-1}/((1-\varepsilon)r),
\end{align*}
as desired.
\end{proof}
\noindent{\bf Proof of Theorem \ref{thmcrossekr}.}\;To begin with, put
\begin{equation*}
	\varepsilon=0.5\;\;\;\mbox{and}\;\;\; \delta=2(t+1)/((1-\varepsilon)r).
\end{equation*}
By our assumption that $r\geq2ek$, we obtain $\delta\leq2/e$. We need also the expression $\varphi(r,q,t,m)$ as defined in Lemma \ref{lemmarqt}. In this setting, 
$$\varphi(r,q,t,m)=\sum_{i=0}^{m}\binom{q-i}{t}\left(\frac{q-t-i+1}{r}\right)^{q-t-i}<\varepsilon^{q-t-m-1}\delta.$$
We first prove (ii), and then prove (i) by considering a pair consists of a $t$-intersecting family with itself. We indicate the minor changes needed at the end.

 Let $\mathcal{F}\subseteq\mathcal{A}$ and $\mathcal{G}\subseteq\mathcal{B}$ be cross $t$-intersecting. Perform the algorithm with input $\Omega,t,q=k$,  and $(\mathcal{F},\mathcal{G})$ together with an arbitrary fingerprint $(\mathcal{S}_0,\mathcal{T}_0)$ of $(\mathcal{F},\mathcal{G})$. Consider the number $N$ of rounds of iterations.\\
{\bf Case 1.}\;$N=q-t$.

Now  $\mathcal{X}_{q-t-1}\subsetneqq\mathcal{S}_{q-t-1}$ and $\mathcal{Y}_{q-t-1}\subsetneqq\mathcal{T}_{q-t-1}$, implying that both $\mathcal{S}_{q-t-1}$ and $\mathcal{T}_{q-t-1}$ contains some $t$-subset of $\Omega$. Since these two families are cross $t$-intersecting antichains, there is a $t$-subset $X$ such that $\mathcal{S}_{q-t-1}=\mathcal{T}_{q-t-1}=\{X\}$. Then of course  $\mathcal{X}_{q-t-1}=\mathcal{Y}_{q-t-1}=\emptyset$. If $X$ is contained in all the members of $\mathcal{F}$ and $\mathcal{G}$, then 
$$|\mathcal{F}||\mathcal{G}|=|\mathcal{F}[X]||\mathcal{G}[X]|\leq|\mathcal{A}[X]||\mathcal{B}[X]|\leq\Delta_{t}(\mathcal{A})\Delta_{t}(\mathcal{B}),$$
with equality precisely if $|\mathcal{A}[X]|=\Delta_{t}(\mathcal{A}),|\mathcal{B}[X]|=\Delta_{t}(\mathcal{B})$, and both $\mathcal{F}$ and $\mathcal{G}$ consist of all possible sets containing $X$, namely, $\mathcal{F}=\left\{F\in\mathcal{A}:X\subseteq F\right\}$ and $\mathcal{G}=\left\{G\in\mathcal{B}:X\subseteq G\right\}$. 

Now suppose $X\nsubseteq\cap(\mathcal{F}\cup\mathcal{G})$. This assumption forces $q\ge t+2$. Indeed, if $q=t+1$, then
$q-t-1=0$ and hence $\mathcal S_0=\mathcal T_0=\{X\}$, so by the property of the 
fingerprint every member of $\mathcal F\cup\mathcal G$
contains $X$, a contradiction. Without loss of generality, suppose that $X\nsubseteq\cap\mathcal{F}$.  Let us estimate first the size of $\mathcal{F}$. By applying Lemma \ref{lemmafingerprintproperty} (iii) to $i=q-t-2$, and noting that $\mathcal{F}[\mathcal{S}_{q-t-1}]=\mathcal{F}[X]$, we obtain
$$\mathcal{F}\subseteq\left(\bigcup_{i=0}^{q-t-2}\mathcal{F}[\mathcal{X}_i]\right)\cup\mathcal{F}[X]\subseteq\left(\bigcup_{i=0}^{q-t-2}\mathcal{A}[\mathcal{X}_i]\right)\cup\mathcal{A}[X].$$
By combining this with Lemmas \ref{lemmafingerprintproperty} (vi) and  \ref{lemmarqt}, and the weakly $(r,t)$-spreadness of $\mathcal{A}$, we obtain
\begin{align*}
	|\mathcal{F}|&\leq\sum_{i=0}^{q-t-2}\sum_{R\in\mathcal{X}_i}|\mathcal{A}[R]|+|\mathcal{A}[X]|\leq\sum_{i=0}^{q-t-2}|\mathcal{X}_i|\Delta_{q-i}(\mathcal{A})+\Delta_{t}(\mathcal{A})\\
	&\leq\left(1+\sum_{i=0}^{q-t-2}\binom{q-i}{t}\left(\frac{q-t-i+1}{r}\right)^{q-t-i}\right)\cdot\Delta_{t}(\mathcal{A})\\
	&=(1+\varphi(r,q,t,q-t-2))\cdot\Delta_{t}(\mathcal{A})<\left(1+\delta\varepsilon\right)\cdot\Delta_{t}(\mathcal{A}).
\end{align*}

Since $X\nsubseteq\cap\mathcal{F}$, we have $|F_0\cap X|=s\leq t-1$ for some $F_0\in\mathcal{F}$. It follows that $|(G\setminus X)\cap F_0|\geq t-s$ whenever $G\in\mathcal{G}[X]$, and thus, with $R$ ranging over $\binom{F_0\setminus X}{t-s}$,
\begin{align*}
	|\mathcal{G}[X]|&\leq\sum_{R}|\mathcal{G}[X\cup R]|\leq\binom{k-s}{t-s}\cdot\Delta_{2t-s}(\mathcal{B})\\
	&\leq \left(\frac{k-t+1}{r}\right)^{t-s}\cdot\Delta_{t}(\mathcal{B})\leq\frac{k-t+1}{r}\cdot\Delta_{t}(\mathcal{B}),
\end{align*}
where the third inequality follows from (\ref{equltrt}) and the weakly $(r,t)$-spreadness of $\mathcal{B}$. By the same argument, we get  $\mathcal{G}\subseteq\left(\bigcup_{i=0}^{q-t-2}\mathcal{B}[\mathcal{Y}_i]\right)\cup\mathcal{B}[X]$, and hence derive that
$$|\mathcal{G}|\leq\left(\frac{k-t+1}{r}+\delta\varepsilon\right)\cdot\Delta_{t}(\mathcal{B}).$$
Then the product of sizes of $\mathcal{F}$ and $\mathcal{G}$ is less than $\Delta_{t}(\mathcal{A})\Delta_{t}(\mathcal{B})$ as 
$$\left(1+\delta\varepsilon\right)\left(\frac{k-t+1}{r}+\delta\varepsilon\right)<\left(1+\frac{1}{e}\right)\left(\frac{1}{2e}+\frac{1}{e}\right)<1,$$
where the first inequality follows from $r\geq2 ek$ and $\delta\leq2/e$.\\
\noindent{\bf Case 2.}\;$N<q-t$.

In this case, at least one of $\mathcal{S}_N$ and $\mathcal{T}_N$ contains no set of size less than $q-N$. Without loss of generality, suppose $\mathcal{S}_N=\mathcal{X}_N$. It follows from Lemma \ref{lemmarqt} that
\begin{align}
	|\mathcal{F}|&\leq\sum_{i=0}^{N-1}|\mathcal{A}[\mathcal{X}_i]|+|\mathcal{A}[\mathcal{S}_N]|=\sum_{i=0}^{N}|\mathcal{A}[\mathcal{X}_i]|\nonumber\\
	&\leq\varphi(r,q,t,N)\cdot\Delta_{t}(\mathcal{A})<\frac{T_{q-t-N}}{1-\varepsilon}\cdot\Delta_{t}(\mathcal{A}),\label{equthmcrossekr1}
\end{align}
where $T_{j}:=\binom{j+t}{t}(j+1)^j/r^j$. Set $T_{q-t+1}=0$ for short.

We proceed by bounding the size of $\mathcal{G}$. Note that $\mathcal{X}_N$ and $\mathcal{T}_N$ are cross $t$-intersecting, then by picking an arbitrary set in $\mathcal{X}_N$, say $F$, it holds that every set from $\mathcal{T}_N$, and thus every one from  $\mathcal{B}[\mathcal{T}_N]$ contains some $t$-subset of $F$. Therefore, we get
\begin{align}
	|\mathcal{G}|&\leq\sum_{i=0}^{N-1}|\mathcal{B}[\mathcal{Y}_i]|+|\mathcal{B}[\mathcal{T}_N]|\leq\left(\frac{T_{q-t-N+1}}{1-\varepsilon}+\binom{q-N}{t}\right)\cdot\Delta_{t}(\mathcal{B}).\nonumber
\end{align}
By setting $j=q-t-N$, it suffices to prove
\begin{equation}\label{equthmcrossekrgoal}
	f(j):=T_{j}\left(T_{j+1}+(1-\varepsilon)\binom{j+t}{t}\right)<(1-\varepsilon)^2\;\;\;\mbox{for}\;\;1\leq j\leq q-t.
\end{equation}
 From the proof of Lemma \ref{lemmarqt},   $T_j/T_{j+1}>r/(e(j+t+1))$. Fix a $j\in[q-t-1]$, then
\begin{align*}
\frac{f(j)-f(j+1)}{(1-\varepsilon)T_{j+1}}&=\frac{T_j}{T_{j+1}}\left(\frac{T_{j+1}}{1-\varepsilon}+\binom{j+t}{t}\right)-\left(\frac{T_{j+2}}{1-\varepsilon}+\binom{j+t+1}{t}\right)\\
&>\frac{r}{e(j+t+1)}\binom{j+t}{t}-\binom{j+t+1}{t}=\frac{\binom{j+t}{t}}{e(j+t+1)}\left(r-\frac{e(j+t+1)^2}{j+1}\right)>0,
\end{align*}
where in the second step we used $T_j>T_{j+2}$. 
To see the last inequality, note that a routine computation gives $$\frac{(j+t+1)^2}{j+1}\leq\max\{(t+2)^2/2, q^2/(q-t)\},$$
then $r>e(j+t+1)^2/(j+1)$ by our assumption on $r$. Hence the function $f(j)$ is decreasing on $j$, and then we deduce from $r\geq 2e(t+2)^2$ and $\varepsilon=0.5$ that 
\begin{align*}
f(j)&\leq f(1)=T_1(T_2+(1-\varepsilon)(t+1))=\frac{2(t+1)}{r}\left(\frac{9(t+2)(t+1)}{2r^2}+(1-\varepsilon)(t+1)\right)\\
&<\frac{1}{e(t+2)}\left(\frac{9}{8e^2(t+2)^2}+(1-\varepsilon)(t+1)\right)<0.01+0.4(1-\varepsilon)<(1-\varepsilon)^2.
\end{align*} 
Therefore (\ref{equthmcrossekrgoal}) follows, and thus $|\mathcal{F}||\mathcal{G}|<\Delta_{t}(\mathcal{A})\Delta_{t}(\mathcal{B})$. 

Finally, let us prove (i) by setting $\mathcal{F}=\mathcal{G}$ and $\mathcal{A}=\mathcal{B}$. If $|\cap\mathcal{F}|\geq t$, then 
$|\mathcal F|\le |\mathcal A[\cap\mathcal{F}]|\le \Delta_t(\mathcal A)$, with equality
only if $X:=\cap\mathcal{F}$ has size $t$,  $\mathcal F=\mathcal A[X]$ and $|\mathcal A[X]|=\Delta_t(\mathcal A)$. Suppose $|\cap\mathcal{F}|<t$. In Case 1, the estimate used improves to
\[
|\mathcal F|\le
\left(\frac{k-t+1}{r}+\varepsilon\delta\right)\Delta_t(\mathcal A)
<\Delta_t(\mathcal A),
\]
since $r\ge 2ek$. In Case 2, the bound
\eqref{equthmcrossekr1} gives
$$|\mathcal F|\le \varphi(r,q,t,N)\Delta_t(\mathcal A)
<\delta\Delta_t(\mathcal A)<\Delta_t(\mathcal A).$$
This completes the proof.{\hfill$\square$}
\section{Hilton--Milner type results via $t$-covers}\label{secfinstru1}
In what follows, we focus on problems for classical set systems. Let us collect several useful facts. First, the family $\binom{[n]}{k}$ is sufficiently spread, in the sense that  
$$\binom{n-t-s-1}{k-t-s-1}\leq r^{-s}\binom{n-t-s}{k-t-s},\;s=0,1,\ldots,k-t,$$
where $r=(n-t)/(k-t)$. Let us prove several inequalities about binomial coefficients.
\begin{lemma}\label{lemmabinom}
Let $m\geq a+b$. The following hold.
\begin{itemize}
\item[\rm(i)]$(1+ab/m)\binom{m-b}{a}\leq\binom{m}{a}\leq(1+m/(ab))\left(\binom{m}{a}-\binom{m-b}{a}\right).$
\item[\rm(ii)]$\binom{m}{a}-\binom{m-b}{a}\leq b\binom{m-1}{a-1}$ and $\binom{m-b}{a}\geq(1-ab/m)\binom{m}{a}.$
\item[\rm(iii)]$\binom{m}{a}-\binom{m-b}{a}\geq b\binom{m-1}{a-1}-\binom{b}{2}\binom{m-2}{a-2}$.
\end{itemize} 
\end{lemma}
\begin{proof}
(i)\;It suffices to prove $m\binom{m}{a}\geq(m+ab)\binom{m-b}{a}$. Using  $e^{-x}\leq1/(1+x)$, we obtain 
\begin{align*}
\binom{m-b}{a}\bigg/\binom{m}{a}&=\prod_{i=0}^{a-1}(1-b/(m-i))\\
&\leq(1-b/m)^a\leq e^{-ab/m}\leq m/(m+ab).
\end{align*}

(ii)-(iii)\;Note that $\binom{m}{a}-\binom{m-b}{a}$ counts the number of $a$-subsets of $[m]$ that intersect $[b]$. For $1\leq i\leq b$, write  $\mathcal{A}_i=\left\{A\in\binom{[m]}{a}:i\in A\right\}$. Then $\binom{m}{a}-\binom{m-b}{a}=|\cup_{i}\mathcal{A}_i|$. By the union bound, it is at most $b\binom{m-1}{a-1}=\frac{ab}{m}\binom{m}{a}$. So (ii) holds. For (iii), by the Bonferroni inequality, we have  $$|\cup_{i}\mathcal{A}_i|\geq\sum_i|\mathcal{A}_i|-\sum_{i<j}|\mathcal{A}_i\cap\mathcal{A}_j|=b\binom{m-1}{a-1}-\binom{b}{2}\binom{m-2}{a-2},$$
as desired.
\end{proof}
Applying Lemma \ref{lemmabinom} (i) to $m=n-t,a=k-t$ and $b=k-t+1$ gives
\begin{equation}\label{equbinom}
	\binom{n-t}{k-t}\leq(1+r/b)\left(\binom{n-t}{k-t}-\binom{n-k-1}{k-t}\right),
\end{equation}
and then
\begin{equation}\label{equbinom'}
	\binom{n-t-1}{k-t-1}\leq(1/r+1/b)\left(\binom{n-t}{k-t}-\binom{n-k-1}{k-t}\right)
\end{equation}
from $\binom{n-t}{k-t}=r\binom{n-t-1}{k-t-1}$. 
\begin{lemma}\label{lemmaind}
Let $n\geq(k-t+1)(\ell-t+1)+t$ and $\mathcal{F}\subseteq\binom{[n]}{k}$. Suppose that $G\in\binom{[n]}{\ell}$ and $A\in\binom{[n]}{a}$, where $\ell\ge t$ and $t\le a\le k$. If $G$ is a $t$-cover of $\mathcal{F}$ and $|A\cap G|=s<t$, then 
\begin{equation*}
|\mathcal{F}[A]|\leq\binom{n-a}{k-a}-\binom{n-a-(\ell-t+1)}{k-a}\leq(\ell-t+1)\binom{n-a-1}{k-a-1}.
\end{equation*}
In particular, there is an $(a+t-s)$-subset containing $A$ with  $|\mathcal{F}[A]|\leq\binom{\ell-s}{t-s}|\mathcal{F}[A']|$.
\end{lemma}
	\begin{proof}
 Every member of $\mathcal{F}[A]$ shares at least $t$ elements with $G$, and thus intersects $G\setminus A$ in at least $t-s$ elements. It follows that $\mathcal{F}[A]=\bigcup_{H\in\mathcal{H}}\mathcal{F}[H]$, where $\mathcal{H}$ is the collection of $(a+t-s)$-subsets of $G\cup A$ that contain $A$. Then we may choose $A'$ to be such a set with  $|\mathcal{F}[A']|\geq|\mathcal{F}[H]|$ whenever $H\in\mathcal{H}$. The union bound gives
 $$|\mathcal{F}[A]|\leq\binom{\ell-s}{t-s}\binom{n-(a+t-s)}{k-(a+t-s)}.$$
Since $r:=(n-t)/(k-t)>\ell-t+1$, the expression on the right-hand side is increasing on $s$, and so  $|\mathcal{F}[A]|\leq\binom{\ell-t+2}{2}\binom{n-a-2}{k-a-2}$ for $s\leq t-2$. For $s=t-1$, a better bound is
$$|\mathcal{F}[\mathcal{A}]|\leq\left|\left\{F\in\binom{[n]}{k}:A\subseteq F,\;F\cap(G\setminus A)\neq\emptyset\right\}\right|=\binom{n-a}{k-a}-\binom{n-a-(\ell-t+1)}{k-a}.$$
By Lemma \ref{lemmabinom} (ii)-(iii), the right-hand side is at most $(\ell-t+1)\binom{n-a-1}{k-a-1}$ and at least $$(\ell-t+1)\binom{n-a-1}{k-a-1}-\binom{\ell-t+1}{2}\binom{n-a-2}{k-a-2}.$$
 Moreover, it is larger than  $\binom{\ell-t+2}{2}\binom{n-a-2}{k-a-2}$ as
 $\binom{\ell-t+2}{2}+\binom{\ell-t+1}{2}=(\ell-t+1)^2<(\ell-t+1)r$. So the desired bound on $|\mathcal{F}[A]|$ follows.
	\end{proof}
The next lemma is a key ingredient for our method. The technique has proved useful for intersection problems for a variety of objects, see, for example,  \cite{Lv-2021,Cao-Lu-Lv-Wang-2024,Wen-Lv-2026}. 
\begin{lemma}\label{lemmakey}
	Let $n\geq(k-t+1)(\ell-t+1)+t$ and $\mathcal{F}\subseteq\binom{[n]}{k}$. Suppose that $\mathcal{G}$ is a collection of 
	$t$-covers of $\mathcal{F}$, where each has size at most $\ell$. If $\tau_t(\mathcal{G})\geq m$, then  $$|\mathcal{F}[X]|\leq(\ell-t+1)^{m-t}\binom{n-m}{k-m}$$
	for all $X\in\binom{[n]}{t}$. In particular, $|\mathcal{F}|\leq(\ell-t+1)^{m-t}\binom{\tau_t(\mathcal{F})}{t}\binom{n-m}{k-m}.$
\end{lemma}
\begin{proof}
If the displayed inequality holds for all $t$-subset $X$, then the upper bound on the size of $\mathcal{F}$ also holds. Indeed, let $T$ be a $t$-cover of $\mathcal{F}$ of size $\tau_t(\mathcal{F})$. Then clearly $\mathcal{F}=\cup_{S\in\binom{T}{t}}\mathcal{F}[S]$, and the union bound gives $|\mathcal{F}|\leq\binom{\tau_t(\mathcal{F})}{t}(\ell-t+1)^{m-t}\binom{n-m}{k-m}$. 

Fix an $X\in\binom{[n]}{t}$. We may suppose $m\geq t+1$. Since $\tau_t(\mathcal{G})\geq m$, the set $X$ is not a $t$-cover of $\mathcal{G}$. Hence, there exists $G_1\in\mathcal{G}$ with $|X\cap G_1|<t$. Set $H_1=X$. By Lemma \ref{lemmaind}, we can inductively choose $G_1,G_2,\ldots,G_j\in\mathcal{G}$ and sets $H_1\subsetneqq H_2\subsetneqq\cdots\subsetneqq H_j$ such that $|H_i\cap G_i| <t$, $|H_{i+1}|=|H_{i}|+t-|H_i\cap G_i|$ and
\begin{equation*}
	|\mathcal{F}[H_i]|\leq\binom{|G_i|-|H_i\cap G_i|}{t-|H_i\cap G_i|}|\mathcal{F}[H_{i+1}]|
\end{equation*}   
for $1\leq i\leq j-1$, and $|H_{j-1}|<\tau_t(\mathcal{G})\leq|H_j|$.
Note that $|G_i|\leq\ell$ for $1\leq i\leq j$. Therefore,
\begin{align*}
	|\mathcal{F}[H]|&\leq\prod_{i=1}^{j-1}\binom{\ell-|H_i\cap G_i|}{t-|H_i\cap G_i|}|\mathcal{F}[H_j]|\leq\mathcal(\ell-t+1)^{\sum_{i=1}^{j-1}(t-|H_i\cap G_i|)}|\mathcal{F}[H_j]|\\
	&=\mathcal(\ell-t+1)^{\sum_{i=1}^{j-1}(|H_{i+1}|-|H_i|)}|\mathcal{F}[H_j]|=\mathcal(\ell-t+1)^{|H_j|-t}|\mathcal{F}[H_j]|.
\end{align*}
where in the second step we used (\ref{equltrt}) repeatedly. Since $n\geq(k-t+1)(\ell-t+1)+t$, it is easy to check that the expression $(\ell-t+1)^{u-t}\binom{n-u}{k-u}$ is decreasing on $u\in[t,k]$. Then the desired bound follows by combining this with $|\mathcal{F}[H_j]|\leq\binom{n-|H_j|}{k-|H_j|}$ and  $|H_j|\geq\tau_t(\mathcal{G})\geq m$.
\end{proof}
 The key to using Lemma \ref{lemmakey} is to find a suitable set of $t$-covers of the considered family. Given a pair $(\mathcal{F},\mathcal{G})$ of cross $t$-intersecting families, then $\mathcal{G}$ is a natural choice of such a collection for $\mathcal{F}$, and vice versa. In a typical application of the $t$-cover method to characterize the structure of large $t$-intersecting families, a standard and effective step is to show that families with large $t$-covering number have relatively small size. One then finds that many of the remaining families admit a collection of $t$-covers with relatively simple structure, from which the original family can often be recovered, or its size can be estimated more precisely. If we perform the algorithm to the pair, the resulting families encode structural information about the original pair implicitly. To obtain more about them, a key step of the present method is to start with a suitably chosen initial fingerprint. For the pair $(\mathcal{F},\mathcal{G})$ with  $\mathcal{F}\subseteq\binom{[n]}{k}$ and  $\mathcal{G}\subseteq\binom{[n]}{\ell}$, we define
 \begin{align*}
	\mathcal{M}(\mathcal{G})&=\{T\subseteq[n]:T\;\mbox{is a minimal $t$-cover of $\mathcal{G}$},\;|T|\leq k\}\;\;\mbox{and}\\
	\mathcal{M}(\mathcal{F})&=\{T\subseteq[n]:T\;\mbox{is a minimal $t$-cover of $\mathcal{F}$},\;|T|\leq\ell\}.
 \end{align*}
The following lemma suggests that, when the pair is  maximal, we may start with $(\mathcal{M}(\mathcal{G}),\mathcal{M}(\mathcal{F}))$. It turns out to be essential in our approach.
\begin{lemma}\label{lemmamaximal}
Suppose $n\geq k+\ell$, and $\mathcal{F}\subseteq\binom{[n]}{k}$ and  $\mathcal{G}\subseteq\binom{[n]}{\ell}$ are maximal cross $t$-intersecting families, then $(\mathcal{M}(\mathcal{G}),\mathcal{M}(\mathcal{F}))$ is a fingerprint of $(\mathcal{F},\mathcal{G})$. 
\end{lemma}
\begin{proof}
	We verify Lemma \ref{lemmafingerprintdef} (i)-(iii) for $(\mathcal{M}(\mathcal{G}),\mathcal{M}(\mathcal{F}))$. By symmetry, we verify the statements concerning $\mathcal{M}(\mathcal{G})$. First, for all $S\in\mathcal{M}(\mathcal{G})$, since $\mathcal{F}$ and $\mathcal{G}$ are maximal, every $k$-subset containing $S$ lies in $\mathcal{F}$, and certainly $S$ is a subset of some set in $\mathcal{F}$. Every $F\in\mathcal{F}$ is a $t$-cover of $\mathcal{G}$, and so each of them contains a minimal $t$-cover of $\mathcal{G}$, which gives $\mathcal{F}=\mathcal{F}[\mathcal{M}(\mathcal{G})]$. It remains to prove that every set from   $\mathcal{M}(\mathcal{G})$ is a minimal $t$-cover of $\mathcal{M}(\mathcal{F})$. Let $S\in\mathcal{M}(\mathcal{G})$. Since $n\geq k+\ell$, for all $T\in\mathcal{M}(\mathcal{F})$, there exist a $k$-subset $F$ and an $\ell$-subset $G$ with $S\subseteq F$, $T\subseteq G$ and $F\cap G=S\cap T$. By the maximality again, $F\in\mathcal{F}$ and $G\in\mathcal{G}$, and necessarily $|S\cap T|=|F\cap G|\geq t$. Hence $S$ is a $t$-cover of $\mathcal{M}(\mathcal{F})$. To see that $S$ is minimal, note that every proper subset $S_0$ of $S$ is not a $t$-cover of $\mathcal{G}$. Then there exists $G\in\mathcal{G}$ such that $|S_0\cap G|<t$, and so $|S_0\cap W|<t$ whenever  $W\in\mathcal{M}(\mathcal{F})$ is contained in $G$. The existence of such a subset $W$ is guaranteed by  $\mathcal{G}=\mathcal{G}[\mathcal{M}(\mathcal{F})]$.
\end{proof}
The next technical lemma is required in the proof of Lemma \ref{lemmafin-stru1}. 
\begin{lemma}\label{lemmastrue-alike}
	Let $t\geq2$, $n\geq t+3$ and  $\mathcal{S},\mathcal{T}\subseteq\binom{[n]}{t+1}$. Suppose that $(\mathcal{S},\mathcal{T})$ is a fingerprint of a pair of cross $t$-intersecting families, and there does not exist  $Z\in\binom{[n]}{t+2}$ such that $\mathcal{S},\mathcal{T}\subseteq\binom{Z}{t+1}$. Then $|\mathcal{S}|,|\mathcal{T}|\geq2$, and there exists $I\in\binom{[n]}{t-1}$ and two families $\mathcal{P},\mathcal{Q}\subseteq\binom{[n]\setminus I}{2}$ such that 
	$$\mathcal{S}\subseteq\{I\cup P:P\in\mathcal{P}\}\;\;\mbox{and}\;\;\mathcal{T}\subseteq\{I\cup Q:Q\in\mathcal{Q}\},$$
	where $(\mathcal{P},\mathcal{Q})$ is isomorphic to one of the following pairs:	
	\begin{itemize} 
		\item[\rm(i)]$(\{\{1,2\},\{3,4\},\{1,4\}\},\;\{\{1,3\},\{2,4\},\{1,4\}\})$,
		\item[\rm(ii)]$(\{\{1,2\},\{3,4\},\{1,4\},\{2,3\}\},\;\{\{1,3\},\{2,4\}\})$,
		\item[\rm(iii)]$(\{\{1,3\},\{2,4\}\},\;\{\{1,2\},\{3,4\},\{1,4\},\{2,3\}\})$.
	\end{itemize}
\end{lemma}
\begin{proof}
Fix an $S\in\mathcal{S}$. By the minimality of $\mathcal{S}$ and $\mathcal{T}$, there exists $T\in\mathcal{T}$ that is distinct from $S$. Then $X:=S\cap T$ has size $t$. Write  $S=X\cup\{u\}$, $T=X\cup\{v\}$ and $Z=X\cup\{u,v\}$ for short. 

Without loss of generality, suppose $\mathcal{S}\nsubseteq\binom{Z}{t+1}$. Since $X$ is a proper subset of $T=X\cup\{v\}$, by the minimality, $|S'\cap X|<t$ for some $S'\in\mathcal{S}$.  Since $|S'\cap T|\geq t$, it follows that  $S'=(X\setminus\{x_1\})\cup\{v,w_1\}$ for some $x_1\in X$ and $w_1\in[n]\setminus X$. Similarly, there exists $T'=(X\setminus\{x_2\})\cup\{u,w_2\}\in\mathcal{T}$, where $x_2\in X$ and $w_2\in[n]\setminus X$. Now $\{S\}\subseteq\mathcal{S}[X]\subseteq\{S,R\}$, where $R=X\cup\{w_2\}$. 

Assume that $\mathcal{S}[X]\nsubseteq\binom{Z}{t+1}$, then $\mathcal{S}[X]=\{S,R\}$ and certainly $w_2\notin Z$. This yields $\{v,w_1\}\neq\{u,w_2\}$, and then we derive from $|S'\cap T'|\geq t$ that $x_1=x_2$, $w_1\in\{u,w_2\}$ and $\mathcal{T}\setminus\mathcal{T}[X]=\{T'\}$. Let $W\in\mathcal{S}\setminus\mathcal{S}[X]$, then $|W\cap X|=t-1, v\in W$ and $\{u,w_2\}\cap W\neq\emptyset$. Set 
$$x=x_1,\;\;I=X\setminus\{x\},\;\;\mathcal{P}=\{Y\setminus I:Y\in\mathcal{S}\}\;\;\mbox{and}\;\;\mathcal{Q}=\{Y\setminus I:Y\in\mathcal{T}\}.$$
Then it is routine to check that one of the following holds, where we write a $2$-subset $\{f,g\}$ simply as $fg$.
\begin{itemize}
	\item[\rm(a)]$w_1=u$, $\mathcal{P}=\{xu,uv,xw_2\}$ and $\{xv,uw_2\}\subseteq\mathcal{Q}\subseteq\{xv,uw_2,xu\}$.
	\item[\rm(b)]$w_1=u$, $\mathcal{P}=\{xu,uv,xw_2,vw_2\}$ and $\mathcal{Q}=\{xv,uw_2\}$.
	\item[\rm(c)]$w_1=w_2$, $\mathcal{P}=\{xu,vw_1,xw_1\}$ and $\{xv,uw_1\}\subseteq\mathcal{Q}\subseteq\{xv,uw_1,xw_1\}$.
	\item[\rm(d)]$w_1=w_2$, $\mathcal{P}=\{xu,vw_1,xw_1,uv\}$ and $\mathcal{Q}=\{xv,uw_1\}$.
\end{itemize}
We set $w:=w_2$ when $w_1=u$, and set $w:=w_1$ when $w_1=w_2$. If (a) or (c) holds, then (i) is true by identifying $(x,u,v,w)=(1,4,3,2)$ and $(x,u,v,w)=(1,2,3,4)$, respectively. Similarly, if (b) or (d) holds, then (ii) is true by identifying $(x,u,v,w)=(1,4,3,2)$ and $(x,u,v,w)=(1,2,3,4)$,  respectively. By symmetry, if $\mathcal{T}[X]\nsubseteq\binom{Z}{t+1}$, then the lemma also holds, and now $\mathcal{P}$ and $\mathcal{Q}$ conform to the structures in (i) or (iii).

It remains to consider the case that $\mathcal{S}[X],\mathcal{T}[X]\subseteq\binom{Z}{t+1}$. Now $\mathcal{S}\setminus\binom{Z}{t+1}$ is non-empty by our assumption that $\mathcal{S}\nsubseteq\binom{Z}{t+1}$. Fix a set $W$ from there, then $W=(X\setminus\{x\})\cup\{v,w\}$ for some $x\in X$ and $w\in[n]\setminus Z$. Since $W_0:=(X\setminus\{x\})\cup\{v\}$ is not a $t$-cover of $\mathcal{T}$, there exists $Y\in\mathcal{T}$ such that $|Y\cap W_0|<t$. If $X\subseteq Y$, then $Y=X\cup\{w\}$ as $|Y\cap W|\geq t$ and $W_0\nsubseteq Y$, which contradicts that $\mathcal{T}[X]\subseteq\binom{Z}{t+1}$. Hence $X\nsubseteq Y$, and then we derive from $|Y\cap S|\geq t$ and $|Y\cap W|\geq t$ that $Y=(X\setminus\{x\})\cup\{u,w\}$. We also write $Y'=(X\setminus\{x\})\cup\{u,v\}$. Then it is easy to check that $\{S,W\}\subseteq\mathcal{S}\subseteq\{S,W,Y'\}$ and $\{T,Y\}\subseteq\mathcal{T}\subseteq\{T,Y,Y'\}$. 
Hence (i) is true by setting $I=X\setminus\{x\}$ and identifying $(x,u,v,w)=(3,4,1,2)$.
\end{proof}
We are now ready to state the two structural lemmas of the algorithm which will be used throughout the rest of this paper. Lemma
\ref{lemmafin-stru1} deals with the case where the algorithm stops with $N=q-t-1$. We use $(\mathcal{S}_{q-t-1},\mathcal{T}_{q-t-1})$ at the final level to approximate the original pair. Lemma \ref{lemmafin-stru2} treats the case where the algorithm reaches the last
round, that is, $N=q-t$. In this situation, both sides end up with a common $t$-set. The lemma shows that either one of the families is sufficiently small, or
both have covering number exactly $t+1$ and the pair resembles the 
Hilton--Milner type pairs given in Construction \ref{constructionh}. Before stating these lemmas, let us first record the notation used in the following assumption and, as a warm-up, prove a simple yet crucial lemma (Lemma \ref{lemmafin-stru0}), which establishes two basic consequences of the algorithm. The first says that the parts of the families
captured by the early fingerprints (at levels $0,1,\ldots,q-t-2$) are small. The second identifies the
$(t+1)$-element members of the initial fingerprints as the
$(t+1)$-element $t$-covers of the opposite family; this is important to
our argument, especially when characterizing the  structures. 
\begin{assumption}\label{assumption}
Let $k\geq t+2$, $b=k-t+1$,  $r=(n-t)/(k-t)$ and $\varepsilon\in(0,1)$ with $\varepsilon r\geq ek$.  Put $\delta_{k-t-u}=rT_{u}/(1-\varepsilon)$ for $2\leq u\leq k-t$, where $T_{u}=\binom{u+t}{t}(u+1)^u/r^u$. Put $\delta_{-1}=0$ and write $\delta=\delta_{k-t-2}$ and $\delta'=\delta_{k-t-3}$. Suppose  $\mathcal{F},\mathcal{G}\subseteq\binom{[n]}{k}$ are maximal cross $t$-intersecting families with $|\cap\mathcal{F}|<t$ and $|\cap\mathcal{G}|<t$. Perform the algorithm to the pair $(\mathcal{F},\mathcal{G})$, the uniformity $q:=k$ and $(\mathcal{S}_0,\mathcal{T}_0):=(\mathcal{M}(\mathcal{G}),\mathcal{M}(\mathcal{F}))$, and let $N$ be the number of rounds and $(\mathcal{S}_{i},\mathcal{T}_{i},\mathcal{X}_{i},\mathcal{Y}_i)$ $(i\leq N)$ be the output families. Set 
$\mathcal{S}^*=\mathcal{S}_0\cap\binom{[n]}{t+1}\;\mbox{and}\;\mathcal{T}^*=\mathcal{T}_0\cap\binom{[n]}{t+1}$.
\end{assumption}
\begin{lemma}\label{lemmafin-stru0}
	With the notation in Assumption \ref{assumption}, the following hold.
\begin{itemize}
\item[\rm(i)]It holds that $\delta_0<\delta_1<\cdots<\delta_{q-t-2}$, and 
$$\max\left\{\left|\bigcup_{i=0}^{m}\mathcal{\mathcal{F}}[\mathcal{X}_i]\right|, \left|\bigcup_{i=0}^{m}\mathcal{\mathcal{G}}[\mathcal{Y}_i]\right|\right\}\leq\delta_{m}\binom{n-t-1}{k-t-1}\;\;\mbox{for}\;\;m\leq\min\{N, q-t-2\}.$$ 
\item[\rm(ii)]The families $\mathcal{S}^*$ and $\mathcal{T}^*$ are the set of $t$-covers of $\mathcal{G}$ and $\mathcal{F}$ with size $t+1$, respectively.
\end{itemize}
\end{lemma}
\begin{proof}
	(i)\;By Lemma \ref{lemmarqt}, we have $T_{u+1}<\varepsilon T_u$ for $0\leq u\leq q-t-1$, and then $\delta_0<\delta_1<\cdots<\delta_{q-t-2}$. Next, from Lemmas \ref{lemmafingerprintproperty} (vi) and \ref{lemmarqt}, for $m\leq\min\{N, q-t-2\}$, 
\begin{align*}
	\left|\bigcup_{i=0}^{m}\mathcal{\mathcal{F}}[\mathcal{X}_i]\right|\bigg/\binom{n-t-1}{k-t-1}&\leq\sum_{i=0}^{m}|\mathcal{X}_i|\binom{n-(q-i)}{k-(q-i)}\bigg/\binom{n-t-1}{k-t-1}\\
	&\leq r\sum_{i=0}^{m}\binom{q-i}{t}\left(\frac{q-t-i+1}{r}\right)^{q-t-i}\\
	&=r\varphi(r,q,t,m)\leq rT_{q-t-m}/(1-\varepsilon)=\delta_m.
\end{align*}
The same estimation involving $\mathcal{G}$ and the $\mathcal{Y}_i$'s also holds.

(ii)\;First, $\mathcal{S}^*$ is a collection of $t$-covers of $\mathcal{G}$ with size $t+1$. To see that every such a $t$-cover must belong to $\mathcal{S}^*$, just note that $\mathcal{G}$ is non-trivial, then a $t$-cover with size $t+1$ must be minimal, and so lies in $\mathcal{S}^*\subseteq\mathcal{M}(\mathcal{G})$. Similarly, $\mathcal{T}^*$ consists exactly of the $t$-covers of $\mathcal{F}$ with size $t+1$.
\end{proof}
\begin{lemma}\label{lemmafin-stru1}
With the notation in Assumption \ref{assumption}, suppose further that $N=q-t-1$ and $\mathcal{S}_{q-t-1}=\mathcal{X}_{q-t-1}$.  
Then $\mathcal S^*\subseteq\mathcal S_{q-t-1}$, and the following hold.
	\begin{itemize}
		\item[\rm(i)]If  $\mathcal{T}_{q-t-1}=\mathcal{Y}_{q-t-1}$, then one of the following holds.
	\begin{itemize}
	\item[\rm (ia)]$\mathcal{S}_{q-t-1},\mathcal{T}_{q-t-1}\subseteq\binom{Z}{t+1}$ for some $Z\in\binom{[n]}{t+2}$, and either $\mathcal{F}=\mathcal{G}\cong\mathcal{A}(n,k,t)$, or $|\mathcal{F}||\mathcal{G}|\leq(\delta+tb/r+2)(\delta+t+2)\binom{n-t-1}{k-t-1}^2$ and $\min\{|\mathcal{F}|,|\mathcal{G}|\}\leq(\delta+tb/r+2)\binom{n-t-1}{k-t-1}$.
	\item[\rm (ib)]$\mathcal{S}_{q-t-1}\nsubseteq\binom{Z}{t+1}$ or $\mathcal{T}_{q-t-1}\nsubseteq\binom{Z}{t+1}$ whenever $Z\in\binom{[n]}{t+2}$, and $(\mathcal{S}_{q-t-1},\mathcal{T}_{q-t-1})$ is isomorphic to one of the pairs given in Lemma \ref{lemmastrue-alike}. In particular,  $|\mathcal{F}||\mathcal{G}|\leq(3+\delta)^2\binom{n-t-1}{k-t-1}^2$.
	\end{itemize}
	\item[\rm(ii)]If  $\mathcal{T}_{q-t-1}\neq\mathcal{Y}_{q-t-1}$, then
		$$|\mathcal{F}||\mathcal{G}|\leq\max\left\{M_1\binom{n-t-1}{k-t-1}^2,\;M_2\binom{n-t-1}{k-t-1}\left(\binom{n-t}{k-t}-\binom{n-k-1}{k-t}\right)\right\},$$
		where $M_1:=\max\{(\delta+2)(\delta+t+b),(\delta+1)(\delta+3t+3),(\delta+1)(\delta+2t+b)\}$ and $M_2:=(\delta'+1)(t+1+\delta'(1/r+1/b))$. 
		In particular,  $|\mathcal{F}|\leq(2+\delta)\binom{n-t-1}{k-t-1}$.
	\end{itemize}	
\end{lemma}
\begin{proof}
To ease notation, we write  $\mathcal{S}=\mathcal{S}_{q-t-1}$, and write $\mathcal{T},\mathcal{X}$ and $\mathcal{Y}$ in the same way. By Lemma \ref{lemmafingerprintproperty} (iii), we have
\begin{equation*}
\mathcal{F}\subseteq\left(\bigcup_{i=0}^{q-t-2}\mathcal{\mathcal{F}}[\mathcal{X}_i]\right)\cup\mathcal{F}[\mathcal{S}]\;\;\mbox{and}\;\;\mathcal{G}\subseteq\left(\bigcup_{i=0}^{q-t-2}\mathcal{\mathcal{G}}[\mathcal{Y}_i]\right)\cup\mathcal{G}[\mathcal{T}].
\end{equation*}
We will use frequently a bound given in Lemma \ref{lemmafin-stru0} (i), that is, 
\begin{equation}\label{equlemmafin-stru11}
\max\left\{\left|\bigcup_{i=0}^{q-t-2}\mathcal{\mathcal{F}}[\mathcal{X}_i]\right|, \left|\bigcup_{i=0}^{q-t-2}\mathcal{\mathcal{G}}[\mathcal{Y}_i]\right|\right\}\leq\delta\binom{n-t-1}{k-t-1}.
\end{equation}
Since the members of $\mathcal S^*$ have size $t+1$, and since
$q-(N-1)=t+2$, no member of $\mathcal S^*$ is removed before the last step. Then $\mathcal{S}^*\subseteq\mathcal{S}$ as $\mathcal{S}=\mathcal{X}$ contains no subset with size smaller than $t+1$. 	

(i)\;In this case, we find that $\mathcal{S}^*\subseteq\mathcal{S}$ and $\mathcal{T}^*\subseteq\mathcal{T}$. 
	
	Suppose first  $\mathcal{S},\mathcal{T}\subseteq\binom{Z}{t+1}$ for some $Z\in\binom{[n]}{t+2}$. Then necessarily $\mathcal{S}^*,\mathcal{T}^*\subseteq\binom{Z}{t+1}$. If $|\mathcal{S}^*|\geq 3$, then we find that 
	\begin{align*}
		\mathcal{G}\subseteq\left\{G\in\binom{[n]}{k}:|G\cap Z|\geq t+1\right\}=\mathcal{A}(Z).
	\end{align*}
	Indeed, suppose for contradiction that $|G\cap Z|=t$ for some $G\in\mathcal{G}$, then $t\leq|G\cap S|\leq|G\cap Z|=t$ and thus $G\cap S=G\cap Z$ for all $S\in\mathcal{S}^*$. This leads to $|\mathcal{S}^*|\leq2$ as there are only two $(t+1)$-subsets of $Z$ containing $G\cap Z$, which is a contradiction. Similarly, $\mathcal{F}\subseteq\mathcal{A}(Z)$ provided that $|\mathcal{T}^*|\geq 3$. Therefore, if both $\mathcal{S}^*$ and $\mathcal{T}^*$ have at least three members, then $$\mathcal{F}=\mathcal{G}=\mathcal{A}(Z)$$ 
	by the maximality of $\mathcal{F}$ and $\mathcal{G}$. For the other case, suppose by symmetry that $|\mathcal{S}^*|\leq2$. Now every  $W\in\mathcal{S}\setminus\mathcal{S}^*$ is not a $t$-cover of $\mathcal{G}$, and then Lemma \ref{lemmaind} yields $|\mathcal{F}[W]|\leq(k-t+1)\binom{n-t-2}{k-t-2}$. It follows that 
	\begin{align*}
		|\mathcal{F}|&\leq\left|\bigcup_{i=0}^{q-t-2}\mathcal{\mathcal{F}}[\mathcal{X}_i]\right|+|\mathcal{F}[\mathcal{S}]|\leq(\delta+|\mathcal{S}^*|)\binom{n-t-1}{k-t-1}+(|\mathcal{S}|-|\mathcal{S}^*|)b\binom{n-t-2}{k-t-2}\\
		&\leq(\delta+2+tb/r)\binom{n-t-1}{k-t-1},
	\end{align*}
	where the last inequality follows from $|\mathcal{S}|\leq\binom{|Z|}{t+1}=t+2$ and $r\geq b$. For $|\mathcal{G}|$, we simply use the bound $(\delta+t+2)\binom{n-t-1}{k-t-1}$.
	
	Suppose $\mathcal{S}\nsubseteq\binom{Z}{t+1}$ or $\mathcal{T}\nsubseteq\binom{Z}{t+1}$ whenever $Z\in\binom{[n]}{t+2}$. From Lemma \ref{lemmastrue-alike}, either $\max\{|\mathcal S|,|\mathcal T|\}\leq3$ or $\{|\mathcal{S}|,|\mathcal{T}|\}=\{2,4\}$. Hence 
	$|\mathcal{F}||\mathcal{G}|\leq(\delta+3)^2\binom{n-t-1}{k-t-1}^2$ and $\min\{|\mathcal{F}|,|\mathcal{G}|\}\leq(\delta+3)\binom{n-t-1}{k-t-1}$, where the first inequality follows from $(\delta+2)(\delta+4)<(\delta+3)^2$. This finishes the proof of (i). 

(ii)\;In this case, $\mathcal{T}\setminus\mathcal{Y}$ is non-empty, and it consists of some $t$-subsets of $\cap\mathcal{S}$. 
	
	Suppose first $|\mathcal{S}|\geq2$. Now $|\cap\mathcal{S}|=t$, and thus $\mathcal{T}\setminus\mathcal{Y}=\{\cap\mathcal{S}\}$. Put $X=\cap\mathcal{S}$. Then  $\mathcal{Y}$ is non-empty as there exists some set in $\mathcal{T}$ that does not contain $X$. For each $Y\in\mathcal{Y}\setminus\mathcal{Y}[X]$, since $|Y\cap S|\geq t$ for all $S\in\mathcal{S}$, we have $|X\cap Y|=t-1$ and $\cup(\mathcal{S}(X))\subseteq Y$. It follows that $|\mathcal{S}|=2$ and $\mathcal{Y}\setminus\mathcal{Y}[X]\subseteq\{(\cup\mathcal{S})\setminus\{x\}:x\in X\}$ has size at most $t$. Hence
	\begin{align*}
		|\mathcal{F}|&\leq\left|\bigcup_{i=0}^{q-t-2}\mathcal{\mathcal{F}}[\mathcal{X}_i]\right|+|\mathcal{F}[\mathcal{S}]|\leq(\delta+2)\binom{n-t-1}{k-t-1},\;\;\mbox{and}\\
		|\mathcal{G}|&\leq\left|\bigcup_{i=0}^{q-t-2}\mathcal{\mathcal{G}}[\mathcal{Y}_i]\right|+|\mathcal{G}[\mathcal{Y}\setminus\mathcal{Y}[X]]|+|\mathcal{G}[X]|\leq(\delta+t+(k-t+1))\binom{n-t-1}{k-t-1}.
	\end{align*}
To derive the bound on $|\mathcal{G}|$, we used $\mathcal{T}\setminus\mathcal{Y}=\{X\}$ and  $\mathcal{G}[\mathcal{T}]=\mathcal{G}[\mathcal{Y}]\cup\mathcal{G}[\mathcal{T}\setminus\mathcal{Y}]$, and the bound $|\mathcal{G}[X]|\leq(k-t+1)\binom{n-t-1}{k-t-1}$ follows from Lemma \ref{lemmaind}. Thus $|\mathcal{F}||\mathcal{G}|\leq(\delta+2)(\delta+t+b)\binom{n-t-1}{k-t-1}^2$.
	
	Suppose $\mathcal{S}=\{T\}$ is a singleton. Now certainly $\mathcal{T}\setminus\mathcal{Y}$ consists of several $t$-subsets of $T$. We claim that $\mathcal{Y}=\emptyset$. To the contrary, pick a $Y\in\mathcal{Y}$. Then $Y\neq T$ as $\mathcal{T}$ is an antichain, and so $|Y\cap T|=t$. However, $Y\cap T$ is still a $t$-cover of $\mathcal{S}=\{T\}$, which contradicts that $T$ is minimal.  Hence our claim is true and thus  $\mathcal{T}\subseteq\binom{T}{t}$. Further, the minimality gives $|\mathcal{T}|\geq2$. We will focus on the structure of   $$\mathcal{T}^\uparrow:=\mathcal{T}_{q-t-2}\setminus\mathcal{Y}_{q-t-2}.$$
By using Lemma \ref{lemmafingerprintproperty} (ii) repeatedly, and noting that $\mathcal{T}_{q-t-2}=\mathcal{Y}_{q-t-2}\cup\mathcal{T}^\uparrow$, we obtain
\begin{equation}\label{equlemmafin-stru12}
\mathcal{G}\subseteq\left(\bigcup_{i=0}^{q-t-2}\mathcal{\mathcal{G}}[\mathcal{Y}_i]\right)\cup\mathcal{G}[\mathcal{T}^\uparrow].
\end{equation}
Note that $\mathcal{T}^\uparrow\subseteq\binom{[n]}{\leq(t+1)}$, and it is an antichain. Hence, for each $X\in \mathcal{T}$, the family $\mathcal{T}^\uparrow[X]$ is either the singleton $\{X\}$ or a $(t+1)$-uniform sunflower with kernel $X$. So we proceed by considering $\mathcal{T}\cap\mathcal{T}^\uparrow$. 
	
	If $\mathcal{T}\cap\mathcal{T}^\uparrow=\emptyset$, then by the minimality, for all $X\in\mathcal{T}$, there exists $W\in\mathcal{S}_{q-t-2}$ with $X\nsubseteq W$, then clearly $W\in\mathcal{X}_{q-t-2}$. Since $|W\cap R|\geq t$ for each $R\in\mathcal{T}^\uparrow[X]$, we have $|W\cap X|=t-1$ and  $\cup(\mathcal{T}^\uparrow(X))\subseteq W$. It follows from $|W|=t+2$ that $|\mathcal{T}^\uparrow[X]|\leq3$. This yields  $|\mathcal{T}^\uparrow|\leq3|\mathcal{T}|\leq3(t+1)$, and then we obtain from (\ref{equlemmafin-stru11}) and (\ref{equlemmafin-stru12}) that  
	$$|\mathcal{G}|\leq\left|\left(\bigcup_{i=0}^{q-t-2}\mathcal{\mathcal{G}}[\mathcal{Y}_i]\right)\cup\mathcal{G}[\mathcal{T}^\uparrow]\right|\leq\left(\delta+3(t+1)\right)\binom{n-t-1}{k-t-1}.$$
	For $|\mathcal{F}|$, we simply use  $$|\mathcal{F}|\leq\left|\left(\bigcup_{i=0}^{q-t-2}\mathcal{\mathcal{F}}[\mathcal{X}_i]\right)\cup\mathcal{F}[T]\right|\leq\left(\delta+1\right)\binom{n-t-1}{k-t-1}.$$
	Hence  $|\mathcal{F}||\mathcal{G}|\leq(\delta+1)(\delta+3t+3)\binom{n-t-1}{k-t-1}^2$.
	
	Next, suppose $\mathcal{T}\cap\mathcal{T}^\uparrow\neq\emptyset$. Now $T\notin\mathcal{T}^\uparrow$ as $\mathcal{T}^\uparrow$ is an antichain containing some $t$-subset from  $\mathcal{T}\cap\mathcal{T}^\uparrow$. Suppose  $\mathcal{T}\cap\mathcal{T}^\uparrow=\{X_1\}$ is a singleton. By the minimality again, $\mathcal{X}_{q-t-2}\neq\emptyset$. Fix a $W\in\mathcal{X}_{q-t-2}$. Of course $X_1\subseteq W$, and we may write $W=X_1\cup\{w_1,w_2\}$, where neither $w_1$ nor $w_2$ lie in $T$ as $T\nsubseteq W$. Since $T\notin\mathcal{T}^\uparrow$ and $W$ intersects every member of $\mathcal{T}^\uparrow$ in at least $t$ elements, every set in $\mathcal{T}^\uparrow\setminus\{X_1\}$ is of the form $X\cup\{w_i\}$, where $X\in\mathcal{T}\setminus\{X_1\}$ and $i\in\{1,2\}$. Therefore, we get
	\begin{align*}
	|\mathcal{G}|\leq\left|\bigcup_{i=0}^{q-t-2}\mathcal{\mathcal{G}}[\mathcal{Y}_i]\right|+|\mathcal{G}[\mathcal{T}^\uparrow\setminus\{X_1\}]|+|\mathcal{G}[X_1]|\leq(\delta+2t+b)\binom{n-t-1}{k-t-1}, 
	\end{align*}	
and hence $|\mathcal{F}||\mathcal{G}|\leq(\delta+1)(\delta+2t+b)\binom{n-t-1}{k-t-1}^2$.
	
	It remains to consider the case in which $|\mathcal{T}\cap\mathcal{T}^\uparrow|\geq2$. Now every set from $\mathcal{S}_{q-t-2}$ contains $T$, and thus $\mathcal{S}_{q-t-2}$ consists exactly of $T$ as it is an antichain. This also gives $\mathcal{X}_{q-t-2}=\emptyset$. So we apply Lemma \ref{lemmafingerprintproperty} (iii) to $\mathcal{F}$ and $i=q-t-2$, and use Lemma \ref{lemmafin-stru0} (i) to derive  $$|\mathcal{F}|\leq\left|\left(\bigcup_{i=0}^{q-t-3}\mathcal{\mathcal{F}}[\mathcal{X}_i]\right)\cup\mathcal{F}[\mathcal{S}_{q-t-2}]\right|\leq\left(\delta'+1\right)\binom{n-t-1}{k-t-1}.$$ 
	Since $\mathcal{S}_{q-t-2}=\{T\}$ is a singleton and every set from $\mathcal{T}$ is a minimal $t$-cover of $T$, we have $\mathcal{T}_{q-t-2}\subseteq\binom{[n]}{t}$ and further $\mathcal{T}=\mathcal{T}^{\uparrow}\subseteq\binom{[n]}{t}$. Since $T$ is a $t$-cover of $\mathcal{T}^\uparrow$, it is also one of $\mathcal{G}[\mathcal{T}^\uparrow]$, and then with $H$ ranging over $\binom{T}{t}$, 
	$$|\mathcal{G}[\mathcal{T}_{q-t-2}]|=|\mathcal{G}[\mathcal{T}^\uparrow]|\leq\sum_{H}|\mathcal{G}[H]|\leq(t+1)\left(\binom{n-t}{k-t}-\binom{n-k-1}{k-t}\right),$$
	where the second inequality follows from $|\cap\mathcal{F}|<t$ and Lemma \ref{lemmaind}. Combining these with (\ref{equlemmafin-stru12}) and (\ref{equbinom'}) yield $|\mathcal{G}|\leq(\delta'(1/r+1/b)+t+1)\left(\binom{n-t}{k-t}-\binom{n-k-1}{k-t}\right)$, and thus $$|\mathcal{F}||\mathcal{G}|\leq M_2\binom{n-t-1}{k-t-1}\left(\binom{n-t}{k-t}-\binom{n-k-1}{k-t}\right).$$ This finishes the proof.
\end{proof}
\begin{lemma}\label{lemmafin-stru2}
With the notation in Assumption \ref{assumption}, suppose further that $N=q-t$. Set $s^*=|\mathcal{S}^*|$, $t^*=|\mathcal{T}^*|$, and assume that $\alpha r\geq b$ for some $\alpha\in(0,1/3)$. Then there exist $X\in\binom{[n]}{t}$ and $j\in[q-t-1]$ such that $j=\min\{i\in[q-t-1]:\mathcal{S}_{i}=\mathcal{T}_{i}=\{X\}\}$, and the following hold. 
\begin{itemize}
\item[\rm(i)]If $\max\{\tau_t(\mathcal{F}),\tau_t(\mathcal{G})\}\geq t+2$, then $\min\{|\mathcal{F}|,|\mathcal{G}|\}\leq(\alpha b+\delta)\binom{n-t-1}{k-t-1}$.
\item[\rm(ii)]If $\tau_t(\mathcal{F})=\tau_t(\mathcal{G})=t+1$ and, for  $C:=(1+\alpha)\left(\alpha+\frac{2(1-\alpha)+\delta}{b}\right)$,
\begin{equation}\label{equlemmafin-stru22}
\min\{|\mathcal{F}|,|\mathcal{G}|\}>C\left(\binom{n-t}{k-t}-\binom{n-k-1}{k-t}\right),
\end{equation}
then $s^*,t^*\geq3$ and  $|\mathcal{F}||\mathcal{G}|\leq|\mathcal{H}(n,k,t)|^2$, with equality if and only if $\mathcal{F}=\mathcal{H}(X,K,L)$ and $\mathcal{G}=\mathcal{H}(X,L,K)$ for some  $K,L\in\binom{[n]}{k+1}$ with $X\subseteq K\cap L$ and $|K\cap L|\geq t+2$.
\end{itemize}
\end{lemma}
\begin{proof}
Since $N=q-t$, we have $\mathcal{X}_{q-t-1}\subsetneqq\mathcal{S}_{q-t-1}$ and $\mathcal{Y}_{q-t-1}\subsetneqq\mathcal{T}_{q-t-1}$. Then both $\mathcal{S}_{q-t-1}$ and $\mathcal{T}_{q-t-1}$ contain at least one $t$-subset, and then $\mathcal{S}_{q-t-1}=\mathcal{T}_{q-t-1}=\{X\}$ because these two families are cross $t$-intersecting antichains.  So we may consider the minimal index $j$ with this property. To see that $j\geq1$, note that $\mathcal{F}$ and $\mathcal{G}$ are non-trivial, implying that such an index must be non-zero. 

(i)\;Suppose without loss of generality that $\tau_t(\mathcal{F})\geq t+2$, then from Lemma \ref{lemmakey} and $\alpha r\geq b$, we get $$|\mathcal{G}[\mathcal{T}_{j-1}\setminus\mathcal{Y}_{j-1}]|\leq|\mathcal{G}[X]|\leq b^2\binom{n-t-2}{k-t-2}\leq\alpha b\binom{n-t-1}{k-t-1},$$
and so $|\mathcal{G}|\leq(\alpha b+\delta)\binom{n-t-1}{k-t-1}$ from Lemmas \ref{lemmafingerprintproperty} (iii) and \ref{lemmafin-stru0} (i). 

(ii)\; To begin with, we may suppose that every member of $\mathcal{F}\cup\mathcal{G}$ contains all but at most one element of $X$. Indeed, suppose by symmetry that there exists $W\in\mathcal{F}$ with $|W\cap X|=s\leq t-2$. Then every set from  $\mathcal{G}[X]$ shares at least $t-s$ elements with $W\setminus X$,  yielding $\mathcal{G}[X]=\bigcup_{H}\mathcal{G}[X\cup H]$ with $H$ ranging over $\binom{W\setminus X}{t-s}$, and hence from (\ref{equltrt}),
\begin{align}
	|\mathcal{G}[\mathcal{T}_{j-1}\setminus\mathcal{Y}_{j-1}]|&\leq|\mathcal{G}[X]|\leq\binom{k-s}{t-s}\binom{n-2t+s}{k-2t+s}\leq(k-t+1)^{t-s}/r^{t-s}\binom{n-t}{k-t}\nonumber\\
	&= b^{t-s}/r^{t-s}\binom{n-t}{k-t}\leq b^2/r\binom{n-t-1}{k-t-1}\leq\alpha b\binom{n-t-1}{k-t-1}.\label{equthmdivanalog}
\end{align}
Then $|\mathcal{G}|\leq(\alpha b+\delta)\binom{n-t-1}{k-t-1}$, and further (\ref{equbinom'}) gives
 $$|\mathcal{G}|\bigg/\left(\binom{n-t}{k-t}-\binom{n-k-1}{k-t}\right)\leq(\alpha b+\delta)(1/r+1/b)\leq(\alpha+\delta/b)(1+\alpha)<C.$$
 
Since $\tau_t(\mathcal{F})=\tau_t(\mathcal{G})=t+1$, both $\mathcal{S}^*$ and $\mathcal{T}^*$ are non-empty. Recall from Lemma \ref{lemmafin-stru0} (ii) that $\mathcal{S}^*=\mathcal{S}_0\cap\binom{[n]}{t+1}$ and $\mathcal{T}^*=\mathcal{T}_0\cap\binom{[n]}{t+1}$ are the set of $t$-covers of $\mathcal{G}$ and $\mathcal{F}$ with size $t+1$, respectively. Before $(\mathcal S_j,\mathcal T_j)$ is produced, only sets of size at least
$q-(j-1)\geq t+2$ are removed. Hence no member of $\mathcal S^*$ or
$\mathcal T^*$ has been removed by that time.  Moreover, each member of
$\mathcal S^*$ contains some member of
$\mathcal S_{j-1}\setminus\mathcal X_{j-1}$, and hence all of them contain $X$ as $\mathcal S_j=\mathcal T_j=\{X\}$. Therefore, $X\subseteq\cap\mathcal{S}^*$. Similarly, we have $X\subseteq\cap\mathcal{T}^*$. In particular, both $\mathcal{S}^*$ and $\mathcal{T}^*$ are sunflowers with kernel $X$ as they are $(t+1)$-uniform. Since $\mathcal{F}$ and $\mathcal{G}$ are maximal, every $k$-subset containing some member of $\mathcal{S}^*$ (of $\mathcal{T}^*$) belongs to $\mathcal{F}$ (to $\mathcal{G}$). It follows that 
$$\mathcal{G}[\mathcal{T}^*]=\left\{G\in\binom{[n]}{k}:T\subseteq G\;\mbox{for some}\;T\in\mathcal{T}^*\right\}.$$
Fix a $W_0\in\mathcal{F}\setminus\mathcal{F}[X]$. Then $|W_0\cap X|=t-1$ by our assumption above, and hence $(\cup\mathcal{T}^*)\setminus X\subseteq W_0$.  Therefore, we deduce that  $$\mathcal{G}[X]\setminus\mathcal{G}[\mathcal{T}^*]\subseteq\left\{G\in\binom{[n]}{k}:X\subseteq G,\;G\cap (W_0\setminus(\cup\mathcal{T}^*))\neq\emptyset\right\}.$$
Let $G\in\mathcal{G}\setminus\mathcal{G}[X]$, then it does not contain any member of $\mathcal{T}_{j-1}\setminus\mathcal{Y}_{j-1}$, and hence
\begin{equation*}
\mathcal{G}\setminus\mathcal{G}[X]\subseteq\bigcup_{i=0}^{j-1}\mathcal{G}[\mathcal{Y}_{i}].
\end{equation*}
Here is another characterization for the part, that is, by noting that $|G\cap X|=t-1$ and $(\cup\mathcal{S}^*)\setminus X\subseteq G$ as we have argued for $W_0$, 
$$\mathcal{G}\setminus\mathcal{G}[X]\subseteq\left\{G\in\binom{[n]}{k}:G\cap(\cup\mathcal{S}^*)=(\cup\mathcal{S}^*)\setminus\{x\}\;\mbox{for some}\;x\in X\right\}.$$
Note that  $|W_0\setminus(\cup\mathcal{T}^*)|=k-(t-1)-t^*=b-t^*$. 

Since $\mathcal{G}$ is non-trivial, the family  $\mathcal{G}\setminus\mathcal{G}[X]$ above is non-empty, which implies that $s^*\leq b$. Similarly, we get $t^*\leq b$. For each $w\in W_0\setminus(\cup\mathcal{T}^*)$, the set $X\cup\{w\}$ does not lie in $\mathcal{T}^*$, and so it is not a $t$-cover of $\mathcal{F}$. By Lemma \ref{lemmaind}, we obtain that $$|\mathcal{G}[X\cup\{w\}]|\leq(k-t+1)\binom{n-t-2}{k-t-2}=b\binom{n-t-2}{k-t-2}\;\mbox{for all}\;\;w\in W_0\setminus(\cup\mathcal{T}^*).$$ 
When $t^*\in\{1,2\}$, we have
\begin{align*}
	|\mathcal{G}|&\leq|\mathcal{G}[\mathcal{T}^*]|+|\mathcal{G}[X]\setminus\mathcal{G}[\mathcal{T}^*]|+\left|\bigcup_{i=0}^{j-1}\mathcal{G}[\mathcal{Y}_{i}]\right|\leq\left(t^*+\frac{(b-t^*)b}{r}+\varepsilon^{q-t-j-1}\delta\right)\binom{n-t-1}{k-t-1}\\
	&\leq\left(2+\frac{(b-2)b}{r}+\delta\right)\left(\frac{1}{r}+\frac{1}{b}\right)\left(\binom{n-t}{k-t}-\binom{n-k-1}{k-t}\right),
\end{align*}
where the factor $\left(2+\frac{(b-2)b}{r}+\delta\right)\left(\frac{1}{r}+\frac{1}{b}\right)$ is, by $\alpha r\geq b$, at most $C$. The same argument yields $|\mathcal{F}|\leq C\left(\binom{n-t}{k-t}-\binom{n-k-1}{k-t}\right)$ for $s^*\leq2$, and hence (\ref{equlemmafin-stru22}) does not hold when $t^*\leq2$ or $s^*\leq2$. It follows that $s^*,t^*\geq3$. Now we use the bounds
\begin{align*}
	|\mathcal{G}|&\leq|\mathcal{G}[\mathcal{T}^*]|+|\mathcal{G}[X]\setminus\mathcal{G}[\mathcal{T}^*]|+|\mathcal{G}\setminus\mathcal{G}[X]|\\
	&\leq\left(\binom{n-t}{k-t}-\binom{n-t-t^*}{k-t}\right)+b(b-t^*)\binom{n-t-2}{k-t-2}+t\binom{n-s^*-t}{k-s^*-t+1}
\end{align*}
and 
\begin{align*}
	|\mathcal{F}|&\leq|\mathcal{F}[\mathcal{S}^*]|+|\mathcal{F}[X]\setminus\mathcal{F}[\mathcal{S}^*]|+|\mathcal{F}\setminus\mathcal{F}[X]|\\
	&\leq\left(\binom{n-t}{k-t}-\binom{n-t-s^*}{k-t}\right)+b(b-s^*)\binom{n-t-2}{k-t-2}+t\binom{n-t^*-t}{k-t^*-t+1}.
\end{align*}

We proceed by estimating the sum of their sizes. Put 
$$g(s)=\left(\binom{n-t}{k-t}-\binom{n-t-s}{k-t}\right)+b(b-s)\binom{n-t-2}{k-t-2}+t\binom{n-s-t}{k-s-t+1},\;s\in[3,b].$$
For $4\leq s\leq b$, we have 
\begin{align*}
	g(s)-g(s-1)&=\binom{n-t-s}{k-t-1}-b\binom{n-t-2}{k-t-2}-t\binom{n-t-s}{k-t-s+2}\\
	&\geq\left(\binom{n-t-s}{k-t-1}-t\binom{n-t-s}{k-t-2}\right)-b\binom{n-t-2}{k-t-2}\\
	&\geq\left(\left(\frac{n-k-s+2}{k-t-1}-t\right)\left(1-\frac{(k-t-2)(s-2)}{n-t-2}\right)-b\right)\binom{n-t-2}{k-t-2}>0.
\end{align*}
In the third step we used Lemma \ref{lemmabinom} (ii) for $m=n-t-2,a=k-t-2$ and $b=s-2$. To see the last inequality, note that $r>ek$ gives $n>e(k-t)k+t$, then we get  $\frac{n-k-s+2}{k-t-1}>2b+t$ from $4\leq s\leq b$, and so the factor before $\binom{n-t-2}{k-t-2}$ is larger than $2(1-1/e)b-b>0$. Hence 
$$g(s)\leq g(b)=|\mathcal{H}(n,k,t)|,$$
with equality precisely if $s=b$. Therefore, using the AM-GM inequality, we obtain
$$|\mathcal{F}||\mathcal{G}|\leq\left(\frac{|\mathcal{F}|+|\mathcal{G}|}{2}\right)^2\leq\left(\frac{g(s^*)+g(t^*)}{2}\right)^2\leq|\mathcal{H}(n,k,t)|^2.$$

The final step is to determine when $|\mathcal{F}||\mathcal{G}|=|\mathcal{H}(n,k,t)|^2$. Suppose equality holds, then all inequalities used in the estimates of \(|\mathcal F|\) and
\(|\mathcal G|\) must be equalities. Precisely, we have  $s^*=t^*=b$ and $|W_0\setminus(\cup\mathcal{T}^*)|=b-t^*=0$, implying that $W_0\subseteq \cup\mathcal{T}^*$ and $\mathcal{G}[X]\setminus\mathcal{G}[\mathcal{T}^*]=\emptyset$. Meanwhile, this together with $|W_0\cap X|=t-1$ gives $\mathcal{T}^*=\{X\cup\{w\}:w\in W_0\setminus X\}$. Similarly, fix a  $Y_0\in\mathcal{G}\setminus\mathcal{G}[X]$. Then $\mathcal{F}[X]\setminus\mathcal{F}[\mathcal{S}^*]=\emptyset$ and $\mathcal{S}^*=\{X\cup\{y\}:y\in Y_0\setminus X\}$ . It follows that
\begin{equation*}
	\mathcal{G}=\left\{G\in\binom{[n]}{k}:X\subseteq G,\;G\cap(W_0\setminus X)\neq\emptyset\right\}\cup\left\{(Y_0\cup X)\setminus\{x\}:x\in X\right\}
\end{equation*}
and
\begin{equation*}
	\mathcal{F}=\left\{F\in\binom{[n]}{k}:X\subseteq F,\;F\cap(Y_0\setminus X)\neq\emptyset\right\}\cup\left\{(W_0\cup X)\setminus\{x\}:x\in X\right\}.
\end{equation*}
Finally, it is necessary that $|(W_0\cup X)\cap(Y_0\cup X)|\geq t+2$. So we obtain the optimal families by setting $K=W_0\cup X$ and $L=Y_0\cup X$. This completes the proof.
\end{proof}
We are now in a position to prove Theorems   \ref{thmmin-max} and \ref{thmcrosshm}. Before presenting them, let us note that, by a routine counting argument,
\begin{align}
	|\mathcal{A}(n,k,t)|&=(t+2)\binom{n-t-1}{k-t-1}-(t+1)\binom{n-t-2}{k-t-2}\;\;\mbox{and}\label{equsizea}\\
	|\mathcal{H}(n,k,t)|&=\binom{n-t}{k-t}-\binom{n-k-1}{k-t}+t.\label{equsizeh}
\end{align}
\noindent{\bf Proof of Theorem \ref{thmmin-max}.}\;Set $r=(n-t)/(k-t)$, $b=k-t+1$, $\alpha=1/(2e)<0.184$, $\varepsilon=0.5$ and $\delta=9\binom{t+2}{2}/((1-\varepsilon)r)$. Then  $\varepsilon r\geq ek$ and $\alpha r\geq b$. Also note that $r\geq30(t+2)$, and then $\delta\leq0.3(t+1)$ and $|\mathcal{A}(n,k,t)|>(t+2-(t+1)/r)\binom{n-t-1}{k-t-1}>(t+1.95)\binom{n-t-1}{k-t-1}$.

First, we may suppose that $\mathcal{F}$ and $\mathcal{G}$ are maximal. We may also suppose that both $\mathcal{F}$ and $\mathcal{G}$ are non-trivial. Indeed, if one of them, say $\mathcal{F}$, is trivial, then by picking $X\subseteq\cap\mathcal{F}$ with $|X|=t$, we obtain that $X\nsubseteq\cap\mathcal{G}$, and then Lemma \ref{lemmaind} gives  $$|\mathcal{F}|=|\mathcal{F}[X]|\leq\binom{n-t}{k-t}-\binom{n-k-1}{k-t}=|\mathcal{H}(n,k,t)|-t<|\mathcal{H}(n,k,t)|.$$ 

Perform the algorithm to the pair $(\mathcal{F},\mathcal{G})$, the uniformity $q:=k$ and the fingerprint $(\mathcal{S}_0,\mathcal{T}_0):=(\mathcal{M}(\mathcal{G}),\mathcal{M}(\mathcal{F}))$. Suppose that $(\mathcal{F},\mathcal{G})$ does not conform to any of the two structures listed in (i) or (ii), and it suffices to prove that the desired inequality holds strictly, that is,  $$\min\{|\mathcal{F}|,|\mathcal{G}|\}<\max\{|\mathcal{A}(n,k,t)|, |\mathcal{H}(n,k,t)|\}.$$ 

Suppose first $N\leq q-t-1$, and assume without loss of generality that $\mathcal{S}_N=\mathcal{X}_N$. When $N\leq q-t-2$, Lemma \ref{lemmafin-stru0} (i) gives $|\mathcal{F}|\leq\left|\bigcup_{i=0}^{N}\mathcal{\mathcal{F}}[\mathcal{X}_i]\right|\leq\delta_N\binom{n-t-1}{k-t-1}$. When $N=q-t-1$, Lemma \ref{lemmafin-stru1} yields that at least one of the families has size smaller than 
$$\max\left\{(\delta+tb/r+2), (3+\delta)\right\}\binom{n-t-1}{k-t-1}.$$ 
It follows from $\delta\leq0.3(t+1)$ and $r\geq2eb$  that  $$\min\{|\mathcal{F}|,|\mathcal{G}|\}<(t+1.95)\binom{n-t-1}{k-t-1}<|\mathcal{A}(n,k,t)|.$$

Suppose $N=q-t$. If  $(\tau_t(\mathcal{F}),\tau_t(\mathcal{G}))=(t+1,t+1)$ and (\ref{equlemmafin-stru22}) holds, then we are done as Lemma \ref{lemmafin-stru2} (ii) yields $\min\{|\mathcal{F}|,|\mathcal{G}|\}\leq\sqrt{|\mathcal{F}||\mathcal{G}|}\leq|\mathcal{H}(n,k,t)|$, and the inequality holds strictly due to the characterization of extremal families. It remains to consider what happens when  $(\tau_t(\mathcal{F}),\tau_t(\mathcal{G}))\neq(t+1,t+1)$, or $(\tau_t(\mathcal{F}),\tau_t(\mathcal{G}))=(t+1,t+1)$, but (\ref{equlemmafin-stru22}) does not hold. By Lemma \ref{lemmafin-stru2} again, we may suppose by symmetry that
\begin{equation*}
	|\mathcal{F}|\leq\max\left\{(\alpha b+\delta)\binom{n-t-1}{k-t-1},\; C\left(\binom{n-t}{k-t}-\binom{n-k-1}{k-t}\right)\right\},
\end{equation*}
where $C=(\alpha+1)\left(\alpha+\frac{2(1-\alpha)+\delta}{b}\right)$. From $\binom{n-t}{k-t}-\binom{n-k-1}{k-t}\leq b\binom{n-t-1}{k-t-1}$ and $Cb>(\alpha b+\delta)$, we obtain that $|\mathcal{F}|\leq Cb\binom{n-t-1}{k-t-1}$. When $k\leq2t+1$, we have  $b\leq t+2$, and then by $\alpha<0.184$ and $\delta\leq0.3(t+1)$,
\begin{align*}
	|\mathcal{F}|\bigg/\binom{n-t-1}{k-t-1}&\leq Cb=(1+\alpha)(\alpha b+2(1-\alpha)+\delta)\\
	&\leq(1+\alpha)((\alpha+0.3)t+2.3)<0.574t+2.73<t+1.95.
\end{align*}
It follows that  $|\mathcal{F}|<|\mathcal{A}(n,k,t)|$. When $k\geq2t+2$, we have  $b\geq t+3$, and then from 
$(1/r+1/b)(\alpha b+\delta)\leq(\alpha+1)(\alpha+\delta/b)<C$, $\delta\leq0.3(t+1)$ and $\alpha<0.184$, we obtain
\begin{align*}
	|\mathcal{F}|\bigg/\left(\binom{n-t}{k-t}-\binom{n-k-1}{k-t}\right)&\leq C<(1+\alpha)\left(\alpha+\frac{2(1-\alpha)+0.3(t+1)}{t+3}\right).
\end{align*}
When $t=2$, the right-hand side is at most  $(1+\alpha)(\alpha+0.4(1-\alpha)+0.18)\leq0.82<1$. When $t\geq3$, it is less than  $(1+\alpha)(\alpha+(1-\alpha)/3+0.3)\leq0.91<1$. Thus $|\mathcal{F}|<|\mathcal{H}(n,k,t)|$.{\hfill$\square$}\vspace{1em}

To prove Theorem \ref{thmcrosshm}, we need two  technical lemmas. 
\begin{lemma}\label{lemmasizeha}
	If $k\leq2t+1$ and $n>2k$, then $|\mathcal{A}(n,k,t)|>|\mathcal{H}(n,k,t)|$. If $k\geq2t+2$ and $n\geq t+(t+3)(t+2)(k-t)$, then $|\mathcal{H}(n,k,t)|>|\mathcal{A}(n,k,t)|$.
\end{lemma}
\begin{proof}
	First, a direct computation gives
	$$|\mathcal A(n,k,t)|-|\mathcal H(n,k,t)|=t\binom{n-t-2}{k-t-1}-\left(\binom{n-t-2}{k-t}-\binom{n-k-1}{k-t}\right)-t.$$
	By Lemma \ref{lemmabinom} (ii), the bracketed term is at most $(k-t-1)\binom{n-t-3}{k-t-1}$. Hence
	\begin{align}
		|\mathcal A(n,k,t)|-|\mathcal H(n,k,t)|
		&\geq
		(2t-k+1)\binom{n-t-3}{k-t-1}
		+t\left(\binom{n-t-3}{k-t-2}-1\right)>0,
	\end{align}
	where the last inequality follows from $k\leq 2t+1$. Suppose $k\geq2t+2$, and write $r=(n-t)/(k-t)$ and $b=k-t+1$ for short. Then $b\geq t+3$, and so  $b(t+2)/(b-t-2)\leq(t+3)(t+2)<r$, that is, $(t+2)(1/r+1/b)<1$. Hence, we derive from (\ref{equsizea}), (\ref{equsizeh}) and (\ref{equbinom'}) that $
	|\mathcal{A}(n,k,t)|/|\mathcal{H}(n,k,t)|<(t+2)\binom{n-t-1}{k-t-1}/(\binom{n-t}{k-t}-\binom{n-k-1}{k-t})<(t+2)(1/r+1/b)<1$.
\end{proof}
\begin{lemma}\label{lemmamono}
	Let $r,q,t>1, \varepsilon\in(0,1)$ be with $q\geq t+2$, $r\geq2e\cdot\max\{(t+2)^2,q^2/(q-t)\}$ and $\varepsilon r\geq eq$, and set  $T_{j}=\binom{j+t}{t}(j+1)^j/r^j$ for $1\leq j\leq q-t$. Then each of the following functions is decreasing on $j\in[2,q-t]$.
	\begin{itemize}
	\item[\rm(i)]$f_1(j):=rT_{j}\left(rT_{j+1}+(1-\varepsilon)(j+1)\binom{j+t}{t}\right)$.
	\item[\rm(ii)]$f_2(j):=\frac{1+\alpha}{b}\cdot\left(\frac{rT_{j+1}}{1-\varepsilon}+a_{j}\right)\left(\frac{(1+\alpha)rT_{j+1}}{(1-\varepsilon)b}+\binom{j+t}{t}\right)$, where $b=q-t+1$, $\alpha=b/r$ and $a_{j}=rT_j/\binom{j+t}{t}$.
	\end{itemize}
\end{lemma}
\begin{proof}
(i)\;From the proof of Lemma \ref{lemmarqt}, we have $T_j/T_{j+1}>r/(e(j+t+1))>1$ for $1\leq j\leq q-t-1$. Hence
\begin{align*}
	\frac{f_1(j)-f_1(j+1)}{(1-\varepsilon)rT_{j+1}}&>\frac{r}{e(j+t+1)}\left(\frac{rT_{j+1}}{1-\varepsilon}+(j+1)\binom{j+t}{t}\right)-\left(\frac{rT_{j+2}}{1-\varepsilon}+(j+2)\binom{j+t+1}{t}\right)\\
	&>\frac{r(j+1)}{e(j+t+1)}\binom{j+t}{t}-(j+2)\binom{j+t+1}{t}\\
	&=\frac{\binom{j+t}{t}(j+1)}{e(j+t+1)}\left(r-\frac{e(j+2)(j+t+1)^2}{(j+1)^2}\right)>0.
\end{align*}
Here the term involving $T_{j+1}$ and $T_{j+2}$ is positive, since
$T_{j+1}/T_{j+2}>r/(e(j+t+2))$ and $r\geq10k$. 
A routine calculation gives
$$\frac{e(j+2)(j+t+1)^2}{(j+1)^2}\leq\max\{4e(t+3)^2/9,e(q-t+1)q^2/(q-t)^2\},$$
and then it is smaller than $r$ as $r\geq2e\cdot\max\{(t+2)^2,q^2/(q-t)\}$. So $f_1(j)$ is decreasing.

(ii)\;Note that
\begin{align*}
	\frac{a_{j}}{a_{j+1}}&=\frac{r(j+1)^{j}}{(j+2)^{j+1}}=\frac{r}{j+1}\cdot\left(1+\frac{1}{j+1}\right)^{-(j+1)}>\frac{r}{e(j+1)}>\varepsilon^{-1}.
\end{align*}
It follows from $T_{j+1}>T_{j+2}$ that
\begin{align*}
	\frac{(f_2(j)-f_2(j+1))b}{1+\alpha}&>\left(\frac{rT_{j+1}}{1-\varepsilon}+a_{j}\right)\binom{j+t}{t}-\left(\frac{rT_{j+2}}{1-\varepsilon}+a_{j+1}\right)\binom{j+t+1}{t}\\
	&=\frac{rT_{j+1}}{1-\varepsilon}\binom{j+t}{t}-\frac{rT_{j+2}}{1-\varepsilon}\binom{j+t+1}{t}+r(T_j-T_{j+1})\\
	&>\frac{T_{j+2}\binom{j+t+1}{t}(j+2)}{e(j+t+2)}\left(r-\frac{e(j+2)(j+t+1)^2}{(j+1)^2}\right)>0,
\end{align*}
and hence $f_2(j)$ is decreasing on $j\in[2,q-t]$ as well. 
\end{proof}
\noindent{\bf Proof of Theorem \ref{thmcrosshm}.}\;Set $r=(n-t)/(k-t)$, $b=k-t+1$, $\alpha=b/r$, $\varepsilon=e/10$, $\delta=9\binom{t+2}{2}/((1-\varepsilon)r)$ and $\delta'=\delta_{k-t-3}=64\binom{t+3}{3}/((1-\varepsilon)r^2)$. By a routine computation, $\varepsilon r\geq ek$, $r\geq15(t+2)^2$, $\alpha\leq0.1$ and 
\begin{equation}\label{equthmcrosshm-2}
\delta<0.42\;\;\;\mbox{and}\;\;\;\delta'<0.1/(t+2).
\end{equation}
 Suppose $\mathcal{F},\mathcal{G}\subseteq\binom{[n]}{k}$ are cross $t$-intersecting. We may suppose that $\mathcal{F}$ and $\mathcal{G}$ are maximal. For simplicity, we write
\begin{equation}\label{equdensity1}
	\rho_1(\mathcal{F},\mathcal{G}):=|\mathcal{F}||\mathcal{G}|\bigg/\left(\binom{n-t}{k-t}-\binom{n-k-1}{k-t}\right)^2
\end{equation}
and 
\begin{equation}\label{equdensity2}
	\rho_2(\mathcal{F},\mathcal{G}):=|\mathcal{F}||\mathcal{G}|\bigg/\left((t+1.95)\binom{n-t-1}{k-t-1}\right)^2. 
\end{equation}
Let us recall from (\ref{equsizea})-(\ref{equsizeh}) that $\binom{n-t}{k-t}-\binom{n-k-1}{k-t}+t$ counts the size of $\mathcal{H}$, and $(t+1.95)\binom{n-t-1}{k-t-1}$ is, by the fact that $(t+1)/r<0.05$, less than $|\mathcal{A}(n,k,t)|$. Hence $\rho_1(\mathcal{F},\mathcal{G})<1$ and $\rho_2(\mathcal{F},\mathcal{G})<1$ implies  $|\mathcal{F}||\mathcal{G}|$ is less than $|\mathcal{H}(n,k,t)|^2$ and $|\mathcal{A}(n,k,t)|^2$, respectively. So it is sufficient to prove that, unless
$(\mathcal F,\mathcal G)$ is one of the two configurations described in (i) or (ii), we have  $\rho_1(\mathcal{F},\mathcal{G})<1$ or $\rho_2(\mathcal{F},\mathcal{G})<1$. To characterize the optimal families, we need also Lemma \ref{lemmasizeha}, which gives $|\mathcal{H}(n,k,t)|>|\mathcal{A}(n,k,t)|$ for $k\geq2t+2$, and $|\mathcal{A}(n,k,t)|>|\mathcal{H}(n,k,t)|$ for $k\leq2t+1$. 

From the fact that $\mathcal{F}$ and $\mathcal{G}$ are maximal and Lemma \ref{lemmamaximal}, $(\mathcal{M}(\mathcal{G}),\mathcal{M}(\mathcal{F}))$ 
is a fingerprint of $(\mathcal{F},\mathcal{G})$. So we perform the algorithm to the pair $(\mathcal{F},\mathcal{G})$, the uniformity $q:=k$ and $$(\mathcal{S}_0,\mathcal{T}_0):=(\mathcal{M}(\mathcal{G}),\mathcal{M}(\mathcal{F})).$$
\noindent{\bf Case 1.\;}$N\leq q-t-2$.
	
	Suppose without loss of generality that $\mathcal{S}_{N}=\mathcal{X}_N$. Then by Lemmas \ref{lemmafin-stru0} (i) and \ref{lemmarqt}, 
\begin{align*}
|\mathcal{F}|&\leq\left|\bigcup_{i=0}^{N}\mathcal{\mathcal{F}}[\mathcal{X}_i]\right|<\delta_N\binom{n-t-1}{k-t-1}=\frac{rT_{q-t-N}}{1-\varepsilon}\binom{n-t-1}{k-t-1}, 
\end{align*}
where $T_{j}=\binom{j+t}{t}(j+1)^j/r^j$ for $1\leq j\leq q-t$. 

To proceed, we consider the $t$-covering number $\tau_t(\mathcal{X}_N)$ of $\mathcal{X}_N$. Fix a $W_0\in\mathcal{X}_N$. Every set in  $\mathcal{G}[\mathcal{T}_N]$ contains some of $\mathcal{T}_N$, and then $t$-intersects $W_0$, yielding that $\tau_t(\mathcal{G}[\mathcal{T}_N])\leq q-N$.  If $\tau_t(\mathcal{X}_N)\geq t+1$, then by Lemma \ref{lemmakey}, 
\begin{align*}
	|\mathcal{G}[\mathcal{T}_N]|\leq(q-N-t+1)\binom{q-N}{t}\binom{n-t-1}{k-t-1}.
\end{align*}
Hence
	\begin{align*}
	|\mathcal{G}|&\leq\sum_{i=0}^{N-1}|\mathcal{\mathcal{G}}[\mathcal{Y}_i]|+|\mathcal{G}[\mathcal{T}_N]|\leq\left(\frac{rT_{q-t-N+1}}{1-\varepsilon}+(q-N-t+1)\binom{q-N}{t}\right)\binom{n-t-1}{k-t-1}.
\end{align*}
Put
$$f_1(j):=rT_{j}\left(rT_{j+1}+(1-\varepsilon)(j+1)\binom{j+t}{t}\right),\;2\leq j\leq q-t.$$
We want to prove $f_1(j)/(1-\varepsilon)^2<(t+1.95)^2$, as we have seen that 
$$|\mathcal{F}||\mathcal{G}|\leq f(q-t-N)\binom{n-t-1}{k-t-1}^2,$$
so the expected inequality implies $\rho_2(\mathcal{F},\mathcal{G})<1$. Since $q=k$,  $r\geq15(t+2)^2$ and $r\geq 10k^2/(k-t)$, by Lemma \ref{lemmamono} (i), we obtain that $f_1(j)$ is decreasing on $j\in[2,q-t]$ and hence, by (\ref{equthmcrosshm-2}),
\begin{align*}
	f_1(j)/(1-\varepsilon)^2&\leq f_1(2)/(1-\varepsilon)^2=\frac{rT_2}{1-\varepsilon}\left(\frac{rT_3}{1-\varepsilon}+3\binom{t+2}{t}\right)\\
	&=\delta\left(\delta'+3\binom{t+2}{t}\right)<0.1+0.7(t+1)(t+2)<(t+1.95)^2,
\end{align*}
as desired.

If $\tau=t$, then we may pick a $t$-subset, say $X$ with $X\subseteq\cap\mathcal{X}_N$. By Lemma \ref{lemmafingerprintproperty} (v), we have
\begin{align*}
|\mathcal{X}_N|&=|\mathcal{X}_N(X)|\leq(q-t-N+1)^{q-t-N}.
\end{align*}
Then $\mathcal{F}[\mathcal{X}_N]$ has size at most
\begin{align*}
|\mathcal{X}_N|\binom{n-(q-N)}{k-(q-N)}\leq r \left(\frac{q-t-N+1}{r}\right)^{q-t-N}\binom{n-t-1}{k-t-1}=a_{q-t-N}\binom{n-t-1}{k-t-1},
\end{align*}
where $a_j:=rT_j/\binom{j+t}{t}=(j+1)^j/r^{j-1}$. Hence
\begin{equation*}
	|\mathcal{F}|\leq\left|\bigcup_{i=0}^{N}\mathcal{\mathcal{F}}[\mathcal{X}_i]\right|\leq\left(\frac{rT_{q-t-N+1}}{1-\varepsilon}+a_{q-t-N}\right)\binom{n-t-1}{k-t-1}.
\end{equation*}
For the size of $\mathcal{G}$, fix \(S\in\mathcal X_N\). Note that every set in $\mathcal{G}[\mathcal{T}_N]$ contains some of $\mathcal{T}_N$, and then $t$-intersects $S$. Hence $\mathcal{G}[\mathcal{T}_N]\subseteq\bigcup_{X}\mathcal{G}[X]$, where $X$ ranges over all $t$-subsets of $S$. This together with Lemma \ref{lemmaind} and (\ref{equbinom}) yields
\begin{align*}
	|\mathcal{G}|&\leq\sum_{i=0}^{N-1}|\mathcal{\mathcal{G}}[\mathcal{Y}_i]|+|\mathcal{G}[\mathcal{T}_N]|\\&\leq\left(\frac{(1+\alpha)rT_{q-t-N+1}}{(1-\varepsilon)b}+\binom{q-N}{t}\right)\left(\binom{n-t}{k-t}-\binom{n-k-1}{k-t}\right).
\end{align*}
Then we need to estimate 
$$f_2(j):=\frac{1+\alpha}{b}\cdot\left(\frac{rT_{j+1}}{1-\varepsilon}+a_{j}\right)\left(\frac{(1+\alpha)rT_{j+1}}{(1-\varepsilon)b}+\binom{j+t}{t}\right),\; 2\leq j\leq q-t.$$
By Lemma \ref{lemmamono} (ii), $f_2(j)$ is decreasing on $j\in[2,q-t]$. When $k\geq2t+2$, we have $b=k-t+1\geq t+3$. This together with $r\geq15(t+2)^2$, $\alpha<0.1$ and  $\delta'\leq0.1/(t+2)$ gives
\begin{align*}
\rho_1(\mathcal{F},\mathcal{G})&\leq f_2(2)=\frac{1+\alpha}{b}\cdot\left(\frac{rT_{3}}{1-\varepsilon}+a_{q-t-2}\right)\left(\frac{(1+\alpha)rT_3}{(1-\varepsilon)b}+\binom{t+2}{t}\right)\\
&<\frac{1.1}{b}\cdot\left(\delta'+\frac{9}{r}\right)\left(\frac{1.1\delta'}{b}+\binom{t+2}{t}\right)<\frac{1}{(t+2)^2}\cdot\left(0.1+\binom{t+2}{2}\right)<1.
\end{align*}
When $k\leq 2t+1$, we have $b\leq t+2$, and then by (\ref{equbinom'}),
\begin{align*}
|\mathcal{F}||\mathcal{G}|\bigg/\binom{n-t-1}{k-t-1}^2&\leq\rho_1(\mathcal{F},\mathcal{G})b^2\leq1.1\left(\delta'+\frac{9}{r}\right)\left(1.1\delta'+b\binom{t+2}{t}\right)\\
&<\frac{1}{t+2}\cdot\left(0.1+(t+2)\binom{t+2}{t}\right)<0.6(t+2)(t+1)<(t+1.95)^2,
\end{align*}
and thus $|\mathcal{F}||\mathcal{G}|<|\mathcal{A}(n,k,t)|^2$.

\noindent{\bf Case 2.\;}$N=q-t-1$.

By symmetry, we may assume that $\mathcal S_{q-t-1}=\mathcal X_{q-t-1}$. 
We then apply Lemma \ref{lemmafin-stru1} to estimate $\rho_1(\mathcal{F},\mathcal{G})$ or $\rho_2(\mathcal{F},\mathcal{G})$. First, since $\delta<0.42$, $r>10b$ and $t\geq2$, we have $(3+\delta)^2<(t+1.95)^2$,  $$(\delta+tb/r+2)(\delta+t+2)<(0.1t+2.42)(t+2.42)<(t+1.95)^2\;\;\mbox{and} $$
$$(\delta+1)(\delta+3(t+1))<(t+1.95)^2.$$
The remaining quantities involve both $k$ and $t$. Suppose first $k\leq2t+1$. We want to verify 
$$\max\{(\delta+2)(\delta+t+b),\;(\delta+1)(\delta+2t+b),\;bM_2\}\leq(t+1.95)^2,$$
where the quantity $bM_2$ arises from $\binom{n-t}{k-t}-\binom{n-k-1}{k-t}\leq b\binom{n-t-1}{k-t-1}$. Using $\delta\leq0.42$, $t\geq2$  and $b=k-t+1\leq t+2$, we obtain $$(\delta+2)(\delta+t+b)\leq2.42(2t+2.42)<(t+1.95)^2\;\mbox{and}$$
$$(\delta+1)(\delta+2t+b)<(\delta+1)(\delta+3t+2)<(t+1.95)^2.$$
For $bM_2$, we use $\delta\leq0.42$, $\delta'<0.1/(t+2)$ and $r\geq10b$ to derive
\begin{align*}
	bM_2&=(\delta'+1)(b(t+1)+\delta'(b/r+1))<\left(0.1/(t+1)+1\right)((t+1)(t+2)+0.5)\\
	&=(t+1.1)(t+2)+0.5+0.05/(t+1)<t^2+3.1t+3<(t+1.95)^2.
\end{align*}
Suppose $k\geq2t+2$. We want to verify
$$\max\{(1/r+1/b)^2(\delta+2)(\delta+t+b),\;(1/r+1/b)^2(\delta+1)(\delta+2t+b),\;(1/r+1/b)M_2\}<1.$$
Note that $b=k-t+1\geq t+3$ gives  $$(\delta+2)(\delta+t+b)-(\delta+1)(\delta+2t+b)=b-\delta(t-1)>0.$$
So we need only to check the first and the third terms are less than $1$. Using $\delta<0.5$ and $t\geq2$ and $r\geq10b$, we obtain
$$(1/r+1/b)^2(\delta+2)(\delta+t+b)\leq\frac{(1.1)^2\cdot2.5}{t+3}\cdot\left(\frac{t+0.5}{t+3}+1\right)\leq0.9075<1.$$
Since \(b\ge t+3\) and \(r\ge15(t+2)^2\), we have $\frac{1}{r}+\frac{1}{b}\le \frac{1}{15(t+2)^2}+\frac1{t+3}<\frac1{t+2}$. Therefore, using \(\delta<0.42\) and \(\delta'<0.1/(t+2)\),
\begin{align*}
(1/r+1/b)M_2&=(1/r+1/b)(\delta'+1)(t+1+\delta(1/r+1/b))\\
&<\frac1{t+2}\left(1+\frac{0.1}{t+2}\right)
\left(t+1+\frac{0.42}{t+2}\right)
<1.
\end{align*}
To see the last inequality, set $u=t+2\geq4$, then it is equivalent to verify $(1+0.1/u)(u-1+0.42/u)<u$. The left-hand side equals $u-(0.9u^2-0.32u-0.042)/u^2$ and hence less than $u$, as required.

\noindent{\bf Case 3.}\;$N=q-t$. 

We apply Lemma \ref{lemmafin-stru2} with $\alpha=b/r$. It states that  $\mathcal{S}_{q-t}=\mathcal{T}_{q-t}=\{X\}$ for some $t$-subset $X$, and there is a minimal index $j$ with this property, that is,
	\begin{equation*}
		j=\min\{i\in[q-t-1]:\mathcal{S}_{i}=\mathcal{T}_{i}=\{X\}\}.
	\end{equation*}
By Lemma \ref{lemmafin-stru2}, if the families satisfy the assumptions in (ii), then we obtain  $|\mathcal{F}||\mathcal{G}|\leq|\mathcal{H}(n,k,t)|^2$ as well as a characterization of families achieving the maximum. It remains to consider the case that  $(\tau_t(\mathcal{F}),\tau_t(\mathcal{G}))\neq(t+1,t+1)$, or $(\tau_t(\mathcal{F}),\tau_t(\mathcal{G}))=(t+1,t+1)$ but (\ref{equlemmafin-stru22}) does not hold. Now
\begin{equation*}
	\min\{|\mathcal{F}|,|\mathcal{G}|\}\leq\max\left\{(\alpha b+\delta)\binom{n-t-1}{k-t-1}, C\left(\binom{n-t}{k-t}-\binom{n-k-1}{k-t}\right)\right\},
\end{equation*}
where $C=(\alpha+1)\left(\alpha+\frac{2(1-\alpha)+\delta}{b}\right)$. Suppose without loss of generality that $\mathcal{G}$ achieves the bounds above. For the size of $\mathcal{F}$, we simply use Lemmas \ref{lemmafin-stru0} (i) and \ref{lemmaind} to derive
\begin{align*}
	|\mathcal{F}|&\leq\sum_{i=0}^{j-1}|\mathcal{\mathcal{F}}[\mathcal{X}_i]|+|\mathcal{F}[\mathcal{S}_{j-1}\setminus\mathcal{X}_{j-1}]|\leq\sum_{i=0}^{j-1}|\mathcal{\mathcal{F}}[\mathcal{X}_i]|+|\mathcal{F}[X]|\\
	&\leq\left(\left(\frac{1}{r}+\frac{1}{b}\right)\varepsilon^{q-t-j-1}\delta+1\right)\left(\binom{n-t}{k-t}-\binom{n-k-1}{k-t}\right).
\end{align*}
Note that $\alpha r\geq b$ gives 
$$(1/r+1/b)(\alpha b+\delta)\leq(\alpha+1)(\alpha+\delta/b)<C.$$
Then from $\delta\leq0.42$, when $b=k-t+1\geq4$, we have 
\begin{align*}
	\rho_1(\mathcal{F},\mathcal{G})&\leq\left(\left(\frac{1}{r}+\frac{1}{b}\right)\delta+1\right)\cdot(1+\alpha)\left(\alpha+\frac{2(1-\alpha)+\delta}{b}\right)\\
	&<\left(\frac{1.1\cdot0.42}{b}+1\right)\cdot1.1\cdot\left(0.1+\frac{2.32}{b}\right)<0.84<1.
\end{align*}
When $b=3$, we have $\alpha=b/r\leq1/(5(t+2)^2)\leq0.0125$, and then $$\rho_1(\mathcal{F},\mathcal{G})<(0.42(1+\alpha)/3+1)(1+\alpha)(\alpha/3+2.42/3)<0.94<1.$$
This finishes the proof of Theorem \ref{thmcrosshm}.{\hfill$\square$}\vspace{1em}

\section{Stability via $t$-diversity}\label{secfinstru2}
In this section, we prove Theorems \ref{thmmaxdiv}, \ref{thmdiv} and \ref{thmstability-t-div}. Let us note that a useful feature of the parameter $\gamma_t$ is that  $$\gamma_t(\mathcal{H})\leq\gamma_t(\mathcal{H}')\;\;\mbox{for}\;\;\mathcal{H}\subseteq\mathcal{H}'\subseteq\binom{[n]}{k}.$$
So in the proof of these theorems, we may deal with maximal families.\vspace{1em}

\noindent{\bf Proof of Theorem \ref{thmmaxdiv}.}\;We may suppose that $(\mathcal{F},\mathcal{G})$ is maximal. We may also suppose that the families are non-trivial, as $\gamma_t$ takes zero for a trivial family. Set $r=(n-t)/(k-t)$, $b=k-t+1$, $\varepsilon=0.5$ and $\delta=9\binom{t+2}{2}/((1-\varepsilon)r)$. Since $r\geq18(t+2)$, we have $\delta\leq0.5(t+1)$. Perform the algorithm to the pair $(\mathcal{F},\mathcal{G})$, the uniformity $q:=k$ and $(\mathcal{S}_0,\mathcal{T}_0):=(\mathcal{M}(\mathcal{G}),\mathcal{M}(\mathcal{F}))$. Our goal is to prove
\begin{equation}\label{equthmmaxdiv1}
	\min\{\gamma_t(\mathcal{F}),\;  \gamma_t(\mathcal{G})\}\leq\gamma_t(\mathcal{A}(n,k,t)) =t\binom{n-t-2}{k-t-1},
\end{equation}
with equality only if $\mathcal{F}=\mathcal{G}=\mathcal{A}(Z)$ for some $(t+2)$-subset $Z$. A useful bound is
\begin{align}
	\gamma_t(\mathcal{A}(n,k,t))&=t\left(1-\frac{k-t-1}{n-t-1}\right)\binom{n-t-1}{k-t-1}\nonumber\\
	&>(t-0.1)\binom{n-t-1}{k-t-1}>\delta\binom{n-t-1}{k-t-1}.\label{equthmmaxdiv3}
\end{align}
\noindent{\bf Case 1.}\;$N\neq q-t-1$.

If $N\leq q-t-2$, then we may assume without loss of generality that $\mathcal{S}_N=\mathcal{X}_N$. Then Lemma \ref{lemmafingerprintproperty} (iii) yields 
$$\mathcal{F}\subseteq\left(\bigcup_{j=0}^{N-1}\mathcal{F}[\mathcal{X}_j]\right)\cup\mathcal{F}[\mathcal{S}_N]=\bigcup_{j=0}^{N}\mathcal{F}[\mathcal{X}_j].$$
By Lemma \ref{lemmafin-stru0} (i) and (\ref{equthmmaxdiv3}), it holds that 
$$|\mathcal{F}|\leq\delta\binom{n-t-1}{k-t-1}<\gamma_t(\mathcal{A}(n,k,t)),$$
and of course (\ref{equthmmaxdiv1}) holds. 

If $N=q-t$, then  $\mathcal{S}_{q-t}=\mathcal{T}_{q-t}=\{X\}$ for some $X\in\binom{[n]}{t}$. We have further from Lemma \ref{lemmafin-stru2} that $\mathcal{S}_{q-t-1}=\mathcal{T}_{q-t-1}=\{X\}$. It follows from Lemma \ref{lemmafingerprintproperty} (ii) that
\begin{equation*}
	\mathcal{F}\subseteq\left(\bigcup_{i=0}^{q-t-2}\mathcal{\mathcal{F}}[\mathcal{X}_i]\right)\cup\mathcal{F}[X],
\end{equation*}
and hence Lemmas \ref{lemmafin-stru0} (i) and \ref{lemmarqt} yield
\begin{equation*}
	\gamma_t(\mathcal{F})\leq|\mathcal{F}\setminus\mathcal{F}[X]|\leq\left|\bigcup_{i=0}^{q-t-2}\mathcal{\mathcal{F}}[\mathcal{X}_i]\right|\leq\delta\binom{n-t-1}{k-t-1}<\gamma_t(\mathcal{A}(n,k,t)).
\end{equation*}
\noindent{\bf Case 2.}\;$N=q-t-1$.

 For simplicity, we write  $\mathcal{S}=\mathcal{S}_{q-t-1}$, and write $\mathcal{T},\mathcal{X}$ and $\mathcal{Y}$ in the same manner. Suppose without loss of generality that $\mathcal{S}=\mathcal{X}$. If $\mathcal{T}\neq\mathcal{Y}$, then  $\mathcal{T}\setminus\mathcal{Y}\subseteq\binom{[n]}{t}$ is non-empty, and hence contains some $t$-subset, say $X$. It follows that $X\subseteq\cap\mathcal{S}$, and so $\mathcal{F}\setminus\mathcal{F}[X]\subseteq\bigcup_{i=0}^{q-t-2}\mathcal{\mathcal{F}}[\mathcal{X}_i]$.  Thus we arrive at $
\gamma_t(\mathcal{F})<\gamma_t(\mathcal{A}(n,k,t))$ again. 

In what follows, we suppose $\mathcal{S}=\mathcal{X}$ and $\mathcal{T}=\mathcal{Y}$. If $\mathcal{F}, \mathcal{G}\subseteq\mathcal{A}(Z)$ for some  $Z\in\binom{[n]}{t+2}$, then the desired inequality is immediate. So we suppose that there is no $(t+2)$-subset $Z$ such that $\mathcal{F}, \mathcal{G}\subseteq\mathcal{A}(Z)$. Since  $\mathcal{S}=\mathcal{X}$ and $\mathcal{T}=\mathcal{Y}$, we have from Lemma \ref{lemmafin-stru1} that  $\mathcal{S}^*\subseteq\mathcal{S}$ and $ \mathcal{T}^*\subseteq\mathcal{T}$, where $\mathcal{S}^*$ and $\mathcal{T}^*$ consist of the set of $t$-covers of $\mathcal{G}$ and $\mathcal{F}$ with size $t+1$, respectively. The remaining proof divides into two subcases.

\noindent{\bf Case 2.1.}\;$\mathcal{S},\mathcal{T}\subseteq\binom{Z}{t+1}$ for some $Z\in\binom{[n]}{t+2}$.

Now $|\mathcal{S}^*|\leq2$ or $|\mathcal{T}^*|\leq2$ by the proof of Lemma \ref{lemmafin-stru1} (i) and our assumption that $\mathcal{F}\nsubseteq\mathcal{A}(Z)$ or $\mathcal{G}\nsubseteq\mathcal{A}(Z)$. Suppose without loss of generality that $|\mathcal{S}^*|\leq2$. Note that  $|\cap\mathcal{S}^*|\geq t$ provided $\mathcal{S}^*\neq\emptyset$ as $\mathcal{S}^*\subseteq\mathcal{S}\subseteq\binom{[n]}{t+2}$. If $\mathcal S^*\neq\emptyset$, choose a $t$-subset
$X\subseteq\cap\mathcal S^*$; if $\mathcal S^*=\emptyset$, choose an
arbitrary $t$-subset $X\subseteq Z$. By Lemma \ref{lemmafin-stru0} (ii), $\mathcal{S}^*$ consists of $t$-covers of $\mathcal{G}$ with size $t+1$. Hence every set from $\mathcal{S}\setminus\mathcal{S}^*$ is not a $t$-cover of $\mathcal{G}$ and then, by Lemma \ref{lemmaind}, contained in at most $b\binom{n-t-2}{k-t-2}$ members of $\mathcal{F}$. Since $\mathcal{S}\subseteq\binom{Z}{t+1}$, at most $t$ of its members do not contain $X$. It follows that
\begin{align*}
	\gamma_t(\mathcal{F})&\leq|\mathcal{F}\setminus\mathcal{F}[X]|\leq\left|\bigcup_{i=0}^{q-t-2}\mathcal{\mathcal{F}}[\mathcal{X}_i]\right|+|\mathcal{F}[\mathcal{S}\setminus\mathcal{S}^*]\setminus\mathcal{F}[X]|\\
	&\leq(\delta+tb/r)\binom{n-t-1}{k-t-1}<\gamma_t(\mathcal{A}(n,k,t)),
\end{align*}
where in the last step we used  (\ref{equthmmaxdiv3}), $\delta\leq0.5(t+1)$ and $r\geq2ek>2eb$. 

\noindent{\bf Case 2.2.}\;$\mathcal{S}\nsubseteq\binom{Z}{t+1}$ or $\mathcal{T}\nsubseteq\binom{Z}{t+1}$ whenever $Z\in\binom{[n]}{t+2}$. 

Suppose by symmetry that $|\mathcal{S}|\geq|\mathcal{T}|$. Then by Lemmas \ref{lemmafin-stru1} (ib) and  \ref{lemmastrue-alike}, we obtain that $|\mathcal{S}|,|\mathcal{T}|\geq2$, and there exists $I\in\binom{[n]}{t-1}$ and two families $\mathcal{P},\mathcal{Q}\subseteq\binom{[n]}{2}$ conforming to one of the following structures such that 
$$\mathcal{S}\subseteq\{I\cup P:P\in\mathcal{P}\}\;\;\mbox{and}\;\;\mathcal{T}\subseteq\{I\cup Q:Q\in\mathcal{Q}\}.$$
\begin{itemize} 	\item[\rm(a)]$(\{\{1,2\},\{3,4\},\{1,4\}\},\{\{1,3\},\{2,4\},\{1,4\}\})$.
	\item[\rm(b)]$(\{\{1,2\},\{3,4\},\{1,4\},\{2,3\}\},\{\{1,3\},\{2,4\}\})$.\vspace{1em}
\end{itemize}
For simplicity, we write $X=I\cup\{1\},S=I\cup\{1,2\},S'=I\cup\{3,4\},R=I\cup\{1,4\}$, $T=I\cup\{1,3\},T'=I\cup\{2,4\}$ and $Y=I\cup\{2,3\}$. 

First, suppose (a) holds.  Now $S'\subseteq F$ for all  $F\in\mathcal{F}[\mathcal{S}]\setminus\mathcal{F}[X]$. If $S'\notin\mathcal{S}^*$, then it is not a $t$-cover of $\mathcal{G}$, and so $|\mathcal{F}[S']|\leq b\binom{n-t-2}{k-t-2}$ from Lemma \ref{lemmaind}. This yields $\gamma_t(\mathcal{F})\leq|\mathcal{F}\setminus\mathcal{F}[X]|<\gamma_t(\mathcal{A}(n,k,t))$ again. Hence we suppose  $S'\in\mathcal{S}^*$. Similarly, we may assume $S\in\mathcal{S}^*$ by considering those not containing $I\cup\{4\}$. By symmetry, we may assume $$\{T,T'\}\subseteq\mathcal{T}^*$$ as otherwise we can derive that $\gamma_t(\mathcal{G})$ is less than $\gamma_t(\mathcal{A}(n,k,t))$. Now we fix an $F\in\mathcal{F}$ that contains neither $X$ nor any set from $\mathcal{S}$. Then $|F\cap T|\geq t$, and so $|F\cap X|=t-1$ and $3\in F$. If $F\cap X=I$, then $F\cap\{2,4\}=\{2\}$ as $S'\nsubseteq F$ and $|F\cap T'|\geq t$; if $F\cap X\neq I$, then $\{2,4\}\subseteq F$. Particularly, in the former case, we observe that  $Y\subseteq F$, and hence there are at most $|\mathcal{F}[Y]|$ such sets in $\mathcal{F}$. Since $Y\notin\mathcal{S}$ and $\mathcal{S}^*\subseteq\mathcal{S}$, it is not a $t$-cover of $\mathcal{G}$, and so Lemma \ref{lemmaind} yields $|\mathcal{F}[Y]|\leq(k-t+1)\binom{n-t-2}{k-t-2}$. Therefore, 
\begin{align*}
	\gamma_t(\mathcal{F})&\leq|\mathcal{F}\setminus\mathcal{F}[X]|\leq|\mathcal{F}[S']|+|\mathcal{F}[Y]|+\left|\{F\in\mathcal{F}[\{2,3,4\}]:F\cap X=X\setminus\{x'\},x'\neq1\}\right|\\
	&\leq\binom{n-t-1}{k-t-1}+(k-t+1)\binom{n-t-2}{k-t-2}+(t-1)\binom{n-t-3}{k-t-2},
\end{align*}
and consequently 
\begin{align*}
	\gamma_t(\mathcal{A}(n,k,t))-\gamma_t(\mathcal{F})&\geq(t-1)\binom{n-t-3}{k-t-1}-(k-t+2)\binom{n-t-2}{k-t-2}\\
	&=\left((t-1)-\frac{(k-t+2)(k-t-1)(n-t-2)}{(n-k)(n-k-1)}\right)\binom{n-t-3}{k-t-1}>0.
\end{align*}
To see the last inequality, note that $n>2e(k-t)k$, then $n-t-2<2(n-k-1)$ and $(k-t+2)(k-t-1)<0.25(n-k)$, and so the quantity is positive. 

Suppose that (b) holds. Then $\mathcal{T}=\{T,T'\}$ as $|\mathcal{T}|\geq2$, and hence $|\mathcal{T}^*|\leq2$. By the same argument as in Case 2.1,  we may assume $\{T,T'\}\subseteq\mathcal{T}^*$, namely, $\mathcal{T}^*=\{T,T'\}$. Let $F\in\mathcal{F}\setminus\mathcal{F}[X]$. If $I\subseteq F$, then $1\notin F$ as $X=I\cup\{1\}\nsubseteq F$. It follows from $|F\cap T|\geq t$ and $|F\cap T'|\geq t$ that $3\in F$ and $\{2,4\}\cap F\neq\emptyset$. If $I\nsubseteq F$, then $[4]\subseteq F$ and $|I\setminus F|=1$. By combining these, we obtain that
\begin{align*}
\gamma_t(\mathcal{F})&\leq|\mathcal{F}\setminus\mathcal{F}[X]|\leq2\binom{n-t-2}{k-t-1}-\binom{n-t-3}{k-t-2}+(t-1)\binom{n-t-3}{k-t-2}\\
&=2\binom{n-t-2}{k-t-1}+(t-2)\binom{n-t-3}{k-t-2}\leq t\binom{n-t-2}{k-t-1}=\gamma_t(\mathcal{A}(n,k,t)).
\end{align*}
In particular, this yields $\gamma_t(\mathcal{F})<\gamma_t(\mathcal{A}(n,k,t))$ for $t\geq3$. Finally, suppose $t=2$. Let us prove  $\gamma_t(\mathcal{G})<\gamma_t(\mathcal{A}(n,k,t))$ provided $\gamma_t(\mathcal{F})=\gamma_t(\mathcal{A}(n,k,t))$. Put $I=\{5\}$ to ease notation. Now $X=\{1,5\}$, and 
\begin{align}
\mathcal{F}\setminus\mathcal{F}[\{1,5\}]=&\left\{F\in\binom{[n]}{k}:F\cap\{1,3,5\}=\{3,5\},\;F\cap\{2,4\}\neq\emptyset\right\}\nonumber\\
&\cup\left\{F\in\binom{[n]}{k}:F\cap[5]=[4]\right\}.\label{equthmdivmax2}
\end{align}

We next consider the set $\{3,5\}$. Note that
\begin{equation}\label{equthmdivmax3}
|\mathcal F\setminus\mathcal F[\{3,5\}]|
\geq\gamma_2(\mathcal F)=2\binom{n-4}{k-3}.
\end{equation} 
On the other hand,
$$
\mathcal F\setminus\mathcal F[\{3,5\}]=
\left(\mathcal F[X]\setminus\mathcal F[\{3,5\}]\right)\cup\left\{F\in\binom{[n]}{k}:F\cap[5]=[4]\right\}.
$$
By our assumption that $\mathcal{T}^*=\{T,T'\}$, the set  $T'=\{2,4,5\}$ is a $2$-cover of $\mathcal F$, and so every
$F\in\mathcal F[X]\setminus\mathcal F[\{3,5\}]$ contains $\{1,5\}$, avoids
$3$, and meets $\{2,4\}$. Therefore
$$
|\mathcal F\setminus\mathcal F[\{3,5\}]|\leq|\mathcal F[X]\setminus\mathcal F[\{3,5\}]|+\binom{n-5}{k-4}\leq2\binom{n-4}{k-3}=\gamma_2(\mathcal F).
$$
This together with (\ref{equthmdivmax3}) yields 
$$
\left\{F\in\binom{[n]}k:F\cap\{1,3,5\}=\{1,5\},\;F\cap\{2,4\}\ne\emptyset\right\}
\subseteq\mathcal F.
$$
By combining this with (\ref{equthmdivmax2}), we obtain that
\begin{align*}
\mathcal{G}\setminus\mathcal{G}[\{1,5\}]\subseteq\left\{G\in\binom{[n]}{k}:G\cap[5]=[4],\;\mbox{or}\;1\notin G\;\mbox{and}\;\{2,4,5\}\subseteq G\right\}.
\end{align*}
Hence $$\gamma_2(\mathcal{G})\leq|\mathcal{G}\setminus\mathcal{G}[X]|\leq\binom{n-5}{k-4}+\binom{n-4}{k-3}<2\binom{n-4}{k-3}=\gamma_2(\mathcal{A}(n,k,2)).$$
This finishes the proof. {\hfill$\square$}\vspace{1em}

\noindent{\bf Proof of Theorem \ref{thmdiv}.}\;Set $r=(n-t)/(k-t)$ and $b=k-t+1$. By $n\geq7(k-t+1)k^2$,  
$$r=\frac{n-t}{k-t}\geq7k^2+\frac{7k^2-t}{k-t}\geq7k(k+1).$$

We may suppose that $(\mathcal{F},\mathcal{G})$ is maximal. By Lemma \ref{lemmabinom} (iii), we have
\begin{align*}
	|\mathcal{L}(n,k,t,u,v)|&>\binom{n-t}{k-t}-\binom{n-v-t}{k-t}\geq v\binom{n-t-1}{k-t-1}-\binom{v}{2}\binom{n-t-2}{k-t-2}\\&\geq\left(v-\binom{v}{2}/r\right)\binom{n-t-1}{k-t-1}=v\left(1-\frac{v-1}{2r}\right)\binom{n-t-1}{k-t-1}.
\end{align*}
Similarly, we have $|\mathcal{L}(n,k,t,v,u)|>u\left(1-\frac{u-1}{2r}\right)\binom{n-t-1}{k-t-1}$. So by our assumption, it follows that
\begin{equation}\label{equthmdiv0}
|\mathcal{F}|>v\left(1-\frac{v-1}{2r}\right)\binom{n-t-1}{k-t-1}\;\;\;\mbox{and}\;\;\;|\mathcal{G}|>u\left(1-\frac{u-1}{2r}\right)\binom{n-t-1}{k-t-1}.
\end{equation}
In particular, both of them have size larger than $2.9\binom{n-t-1}{k-t-1}$. 

\noindent{\bf Case 1.}\;$\max\{|\cap\mathcal{F}|,|\cap\mathcal{G}|\}\geq t$.

Suppose without loss of generality that $|\cap\mathcal{F}|\geq t$. Then $\tau_t(\mathcal{F})=t$ and of course $\gamma_t(\mathcal{F})=0$. If $|\cap\mathcal{G}|\geq t$, then we are done. So we suppose $\tau_t(\mathcal{G})\geq t+1$ and consider  $\gamma_t(\mathcal{G})$. Let $X\subseteq\cap\mathcal{F}$ with $|X|=t$. By the maximality, all $k$-subsets of $[n]$ containing $X$ belong to $\mathcal{G}$, implying that $|\mathcal{G}[X]|=\binom{n-t}{k-t}$. We claim that $\tau_t(\mathcal{G})=t+1$. To the contrary, assume that $\tau_t(\mathcal{G})\geq t+2$, then Lemma \ref{lemmakey} gives 
\begin{align*}
|\mathcal{F}|&\leq b^{\tau_t(\mathcal{G})-t}\binom{n-\tau_t(\mathcal{G})}{k-\tau_t(\mathcal{G})}\leq b^{\tau_t(\mathcal{G})-t}/r^{\tau_t(\mathcal{G})-t-1}\binom{n-t-1}{k-t-1}\leq\binom{n-t-1}{k-t-1}
\end{align*}
as $r\geq b^2$. This contradicts that $|\mathcal{F}|>\binom{n-t}{k-t}-\binom{n-v-t}{k-t}>(v-0.1)\binom{n-t-1}{k-t-1}$. Hence our claim is true. Now we consider the family $\mathcal{S}^*$ of $t$-covers of $\mathcal{G}$ with size $t+1$. Since   $\mathcal{F}$ and $\mathcal{G}$ are maximal and $X\subseteq\cap\mathcal{F}$, we have $X\subseteq\cap\mathcal{S}^*$, and so $\mathcal{S}^*$ is a sunflower with kernel $X$. For simplicity, we write $S=X\cup\{w_S\}$ for $S\in\mathcal{S}^*$. For all $G\in\mathcal{G}\setminus\mathcal{G}[X]$, it holds that $|X\cap G|=t-1$ and $\{w_S:S\in\mathcal{S}^*\}\subseteq G$, and then
\begin{equation}\label{equthmdiv1}
\gamma_t(\mathcal{G})\leq|\mathcal{G}\setminus\mathcal{G}[X]|\leq t\binom{n-s^*-t}{k-s^*-t+1},
\end{equation}
where $s^*:=|\mathcal{S}^*|$. 

To proceed, we estimate the size of $\mathcal{F}$. Since $X\subseteq\cap\mathcal{F}$, we have $\mathcal{F}=\mathcal{F}[\mathcal{S}^*]\cup(\mathcal{F}[X]\setminus\mathcal{F}[\mathcal{S}^*])$. By the maximality, $|\mathcal{F}[\mathcal{S}^*]|=\binom{n-t}{k-t}-\binom{n-t-s^*}{k-t}$. For all $F\in\mathcal{F}[X]\setminus\mathcal{F}[\mathcal{S}^*]$, we have $F\cap(\cap\mathcal{S}^*)=X$ and hence   $F\cap\{w_S:S\in\mathcal{S}^*\}=\emptyset$. Fix  $G_0\in\mathcal{G}\setminus\mathcal{G}[X]$, and set $H_0=G_0\setminus(\cup\mathcal{S}^*)$. Then  $F\cap H_0\neq\emptyset$, and hence $$\mathcal{F}[X]\setminus\mathcal{F}[\mathcal{S}^*]\subseteq\bigcup_{w\in H_0}\mathcal{F}[X\cup\{w\}].$$ Note that $|H_0|=k-t+1-s^*$. For each $w\in H_0$, the $(t+1)$-subset $X\cup\{w\}$ is not a $t$-cover of $\mathcal{G}$ as $w\notin\cup\mathcal{S}^*$, then we get from Lemma \ref{lemmaind} that  $|\mathcal{F}[X\cup\{w\}]|\leq(k-t+1)\binom{n-t-2}{k-t-2}$. Therefore, if $s^*\leq v-1$, then from $r\geq7b^2$, 
\begin{align*}
|\mathcal{F}|&\leq s^*\binom{n-t-1}{k-t-1}+(b-s^*)b\binom{n-t-2}{k-t-2}\\
&\leq(s^*+(b-s^*)b/r)\binom{n-t-1}{k-t-1}<(v-0.1)\binom{n-t-1}{k-t-1},
\end{align*}
which contradicts (\ref{equthmdiv0}). Hence $s^*\geq v$, and particularly  $\gamma_t(\mathcal{G})<t\binom{n-v-t}{k-v-t+1}$ when $s^*\geq v+1$. Indeed, we have further that  (\ref{equthmdiv1}) holds strictly when $s^*=v$. To see this, note that if $|\mathcal{G}\setminus\mathcal{G}[X]|=t\binom{n-s^*-t}{k-s^*-t+1}$, then $\mathcal{G}$ contains every  $k$-subset that intersects $\cup\mathcal{S}^*$ in all but exactly one element of $X$. It follows that every  $F\in\mathcal{F}$ must intersect $(\cup\mathcal{S}^*)\setminus X=\{w_S:S\in\mathcal{S}^*\}$, and thus contain some set from $\mathcal{S}^*$. Hence $\mathcal{F}[X]\setminus\mathcal{F}[\mathcal{S}^*]=\emptyset$ and consequently $|\mathcal{F}|=|\mathcal{F}[\mathcal{S}^*]|=\binom{n-t}{k-t}-\binom{n-t-v}{k-t}$, which is impossible as $|\mathcal{F}|$ is larger than that quantity. 

\noindent{\bf Case 2.}\;$\max\{|\cap\mathcal{F}|,|\cap\mathcal{G}|\}<t$.

Set $\alpha=b/r$, $\varepsilon=e/(7k)$ and $\delta=9\binom{t+2}{2}/((1-\varepsilon)r)$. Note that $\varepsilon r\geq ek$. It is routine to check from $k\geq t+2\geq4$ that $\delta<0.73$. Perform the algorithm to the pair $(\mathcal{F},\mathcal{G})$, the uniformity $q:=k$ and the fingerprint $(\mathcal{S}_0,\mathcal{T}_0):=(\mathcal{M}(\mathcal{G}),\mathcal{M}(\mathcal{F}))$. 

\noindent{\bf Case 2.1.\;}$N\leq q-t-1$. 

If $N=q-t-2$, then by Lemmas \ref{lemmafingerprintproperty} (iii) and \ref{lemmafin-stru0} (i), 
\begin{align*}
	|\mathcal{F}|&\leq\left|\bigcup_{i=0}^{N}\mathcal{\mathcal{F}}[\mathcal{X}_i]\right|\leq\delta\binom{n-t-1}{k-t-1}<0.73\binom{n-t-1}{k-t-1}, 
\end{align*} 
which is impossible as $|\mathcal{F}|>2.9\binom{n-t-1}{k-t-1}$. So $N=q-t-1$. Now by Lemma \ref{lemmafin-stru1}, one of the following holds.
\begin{itemize}
\item[\rm(a)]$\mathcal{F}=\mathcal{G}\cong\mathcal{A}(n,k,t)$.
\item[\rm(b)]$\min\{|\mathcal{F}|,|\mathcal{G}|\}\leq(\delta+tb/r+2)\binom{n-t-1}{k-t-1}$.
\item[\rm(c)]$\mathcal{S}_{q-t-1}=\mathcal{X}_{q-t-1}$ and $\mathcal{T}_{q-t-1}=\mathcal{Y}_{q-t-1}$, and  $(\mathcal{S}_{q-t-1},\mathcal{T}_{q-t-1})$ conforms to some pair described in Lemma \ref{lemmastrue-alike}.
\end{itemize}
First, (b) does not happen as $r\geq7k^2$ and $\delta<0.73$ yields $\delta+tb/r+2<2.9$. Suppose (a) holds, then certainly $|\mathcal{A}(n,k,t)|\geq\max\{|\mathcal{L}(n,k,t,u,v)|,|\mathcal{L}(n,k,t,v,u)|\}$. Let us verify that this holds only for $u,v\leq t+2$. Set by symmetry that  $v\geq t+3$, then from  $|\mathcal{A}(n,k,t)|=\binom{n-t}{k-t}-\binom{n-t-2}{k-t}+t\binom{n-t-2}{k-t-1}$, we obtain that 
\begin{align*}
	|\mathcal{L}(n,k,t,u,v)|-|\mathcal{A}(n,k,t)|&\geq\binom{n-t-2}{k-t}-\binom{n-t-v}{k-t}-t\binom{n-t-2}{k-t-1}+t\binom{n-u-t}{k-u-t+1}\\
	&\geq(v-t-2)\binom{n-t-3}{k-t-1}-t\binom{n-t-3}{k-t-2}-\binom{v-2}{2}\binom{n-t-4}{k-t-2}\\
	&>\left((v-t-2)-\frac{t(k-t-1)}{n-k-1}-\binom{v-2}{2}/r\right)\binom{n-t-3}{k-t-1}>0,
\end{align*}
where in the second step, we used Lemma \ref{lemmabinom} (iii) to derive
$$\binom{n-t-2}{k-t}-\binom{n-t-v}{k-t}\geq(v-2)\binom{n-t-3}{k-t-1}-\binom{v-2}{2}\binom{n-t-4}{k-t-2}.$$
To see the last inequality, note that $r=(n-t)/(k-t)\geq7(k-t+1)k$, then the factor of $\binom{n-t-3}{k-t-1}$ is larger than $1-1/(k-t)-1/2$ and hence positive. Hence (a) occurs only when $u,v\leq t+2$.

It remains to consider whether (c) holds. Write $\mathcal{S}_{q-t-1}=\mathcal{S}$ and $\mathcal{T}_{q-t-1}=\mathcal{T}$ for simplicity. By Lemma \ref{lemmafingerprintproperty} (iii) again, we get
$$\mathcal{F}\subseteq\left(\bigcup_{i=0}^{q-t-2}\mathcal{F}[\mathcal{X}_i]\right)\cup\mathcal{F}[\mathcal{S}].$$
Hence a uniform bound and Lemma \ref{lemmafin-stru0} (i) yield $|\mathcal{F}|\leq(\delta+|\mathcal{S}|)\binom{n-t-1}{k-t-1}$. The same argument gives $|\mathcal{G}|\leq(\delta+|\mathcal{T}|)\binom{n-t-1}{k-t-1}$. So necessarily $|\mathcal{S}|,|\mathcal{T}|\geq3$, and then Lemma \ref{lemmastrue-alike} gives $|\mathcal{S}|=|\mathcal{T}|=3$. However,  $\max\{u,v\}\geq4$ and (\ref{equthmdiv0}) yields $\max\{|\mathcal{F}|,|\mathcal{G}|\}>3.9\binom{n-t-1}{k-t-1}$. Hence (c) does not hold as $\delta<0.73$. 

\noindent{\bf Case 2.2.}\;$N=q-t$.

By Lemma \ref{lemmafin-stru2}, there exist $X\in\binom{[n]}{t}$ and $j\in[q-t-1]$ with $j=\min\{i\in[q-t-1]:\mathcal{S}_{i}=\mathcal{T}_{i}=\{X\}\}$. If $(\tau_t(\mathcal{F}),\tau_t(\mathcal{G}))\neq(t+1,t+1)$, then Lemma \ref{lemmafin-stru2} (i) gives  $\min\{|\mathcal{F}|,|\mathcal{G}|\}\leq(\alpha b+\delta)\binom{n-t-1}{k-t-1}$, where $\alpha b+\delta<2.9$ as $\alpha b=b^2/r<1/7$ and $\delta<0.73$, which also contradicts (\ref{equthmdiv0}). So we deduce that
$$(\tau_t(\mathcal{F}),\tau_t(\mathcal{G}))=(t+1,t+1).$$
It follows that both $\mathcal{S}^*$ and $\mathcal{T}^*$ are non-empty. Since $q-(j-1)\geq t+2$, no member of $\mathcal{S}^*$ is removed in the procedure. Then each of them contains a member of $\mathcal{S}_{j-1}\setminus\mathcal{X}_{j-1}$ as a subset. This yields $X\subseteq\cap\mathcal{S}^*$. Similarly, we get $X\subseteq\cap\mathcal{T}^*$. We observe that $|W\cap X|\geq t-1$ for all  $W\in(\mathcal{F}\setminus\mathcal{F}[X])\cup(\mathcal{G}\setminus\mathcal{G}[X])$. Indeed, assume without loss of generality that there exists $W\in\mathcal{F}\setminus\mathcal{F}[X]$ such that $|W\cap X|\leq t-2$, then by the same argument as we deduced (\ref{equthmdivanalog}), and using (\ref{equbinom'}) and $\alpha=b/r$, we obtain that $$|\mathcal{G}[\mathcal{T}_{j-1}\setminus\mathcal{Y}_{j-1}]|\leq|\mathcal{G}[X]|\leq\alpha(\alpha+1)\left(\binom{n-t}{k-t}-\binom{n-k-1}{k-t}\right).$$
This leads to $|\mathcal{G}[\mathcal{T}_{j-1}\setminus\mathcal{Y}_{j-1}]|<2\binom{n-t-1}{k-t-1}$, and hence we get  $|\mathcal{G}|<(2+\delta)\binom{n-t-1}{k-t-1}<2.9\binom{n-t-1}{k-t-1}$, which is impossible. 

Next, we claim that $s^*,t^*\geq3$. To prove this, it is sufficient to verify (\ref{equlemmafin-stru22}) displayed in Lemma \ref{lemmafin-stru2}. Since $\delta<0.73, \alpha=b/r<1/(7k)<0.036$ and $\alpha b=b^2/r<1/7$, the factor $C=(1+\alpha)\left(\alpha+\frac{2(1-\alpha)+\delta}{b}\right)$ there satisfies
$$Cb<(1+\alpha)(\alpha b+2+\delta)<1.036\cdot2.873<2.977<v\left(1-\frac{v-1}{2r}\right).$$
The last inequality holds trivially for $v\geq4$. When $v=3$, we used $r\geq7k(k+1)\geq140$. Therefore, from (\ref{equthmdiv0}) and Lemma \ref{lemmabinom} (ii), $$|\mathcal{F}|>v\left(1-\frac{v-1}{2r}\right)\binom{n-t-1}{k-t-1}>Cb\binom{n-t-1}{k-t-1}\geq C\left(\binom{n-t}{k-t}-\binom{n-k-1}{k-t}\right),$$
and so (\ref{equlemmafin-stru22}) holds. Hence our claim is true. We now adapt the argument of Case 1 to describe the structure of $\mathcal{F}$ and $\mathcal{G}$. First, $\mathcal{S}^*$ is a sunflower with kernel $X$. Write $V=\cup\mathcal{S}^*$ for short. Note that $|V|=s^*+t$. For all $G\in\mathcal{G}\setminus\mathcal{G}[X]$, we have $|G\cap X|=t-1$, and then  
\begin{equation*}
\mathcal{G}\setminus\mathcal{G}[X]\subseteq\left\{G\in\binom{[n]}{k}:G\cap V=V\setminus\{x\}\;\text{for some}\;x\in X\right\}.
\end{equation*}
Hence $\gamma_t(\mathcal{G})\leq t\binom{n-s^*-t}{k-s^*-t+1}$.

 For $\mathcal{F}$, by the maximality, $\mathcal{F}[\mathcal{S}^*]$ consists of all subsets from $\binom{[n]}{k}$ that contains some member of $\mathcal{S}^*$. By the structure of $\mathcal{S}^*$, 
\begin{equation}\label{equthmdiv7}
\mathcal{F}[\mathcal{S}^*]=\left\{F\in\binom{[n]}{k}:X\subseteq F,\;F\cap(V\setminus X)\neq\emptyset\right\}.
\end{equation}
Thus $|\mathcal{F}[\mathcal{S}^*]|=\binom{n-t}{k-t}-\binom{n-t-s^*}{k-t}$. In addition,
\begin{equation*}
\mathcal{F}[X]\setminus\mathcal{F}[\mathcal{S}^*]\subseteq\left\{F\in\binom{[n]}{k}:F\cap V=X, F\cap(G\setminus V)\neq\emptyset\;\mbox{for all}\;G\in\mathcal{G}\setminus\mathcal{G}[X]\right\}.
\end{equation*}
 The same argument for $\mathcal{G}$ yields, by setting $U=\cup\mathcal{T}^*$,
\begin{equation}\label{equthmdiv8}
	\mathcal{F}\setminus\mathcal{F}[X]\subseteq\left\{F\in\binom{[n]}{k}:F\cap U=U\setminus\{x\}\;\text{for some}\;x\in X\right\}.
\end{equation}
Then we deduce $s^*\geq v$. To see this, note that $t^*\geq3$ yields $\binom{n-t^*-t}{k-t^*-t+1}\leq\binom{n-t-3}{k-t-2}$, and then 
\begin{align*}
	|\mathcal{F}|&\leq(s^*+(b-s^*)b/r)\binom{n-t-1}{k-t-1}+t\binom{n-t^*-t}{k-t^*-t+1}\\
	&<(v-1+(b^2+t)/r)\binom{n-t-1}{k-t-1}<(v-0.1)\binom{n-t-1}{k-t-1} 
\end{align*}
for $s^*\leq v-1$, which is impossible. It follows that $$\gamma_t(\mathcal{G})\leq|\mathcal{G}\setminus\mathcal{G}[X]|\leq t\binom{n-s^*-t}{k-s^*-t+1}\leq t\binom{n-v-t}{k-v-t+1}.$$ Similarly, we have $t^*\geq u$ and $\gamma_t(\mathcal{F})\leq|\mathcal{F}\setminus\mathcal{F}[X]|\leq t\binom{n-u-t}{k-u-t+1}$.

Finally, let us characterize equality. Assume without loss of generality that $\gamma_t(\mathcal{F})=t\binom{n-u-t}{k-u-t+1}$. Then $t^*=u$ and $\gamma_t(\mathcal{F})=|\mathcal{F}\setminus\mathcal{F}[X]|$, and consequently the two families in (\ref{equthmdiv8}) are the same. This yields 
$$\mathcal{F}\supseteq\mathcal{L}(X,U,V),$$
 and then $\mathcal{G}[X]\setminus\mathcal{G}[\mathcal{T}^*]=\emptyset$. 
Hence
\begin{equation}\label{equthmdiv9}
|\mathcal{G}|=|\mathcal{G}[\mathcal{T}^*]|+|\mathcal{G}\setminus\mathcal{G}[X]|\leq\binom{n-t}{k-t}-\binom{n-t-t^*}{k-t}+t\binom{n-s^*-t}{k-s^*-t+1}.
\end{equation}
Then from  $|\mathcal{G}|\geq\binom{n-t}{k-t}-\binom{n-u-t}{k-t}+t\binom{n-v-t}{k-v-t+1}$ and $u=t^*$, we obtain that $v\geq s^*$, and thus $v=s^*$. Moreover, we have equality in (\ref{equthmdiv9}), and this yields
\begin{equation*}
\mathcal{G}=\mathcal{L}(X,V,U).
\end{equation*} 
Hence $|U\cap V|\geq t+2$ as  $\mathcal{F}\supseteq\mathcal{L}(X,U,V)$. Now $(\mathcal{L}(X,U,V), \mathcal{L}(X,V,U))$ forms a pair of maximal cross $t$-intersecting families, and consequently $\mathcal{F}=\mathcal{L}(X,U,V)$. Similarly, once $\gamma_t(\mathcal{G})=t\binom{n-v-t}{k-v-t+1}$, then $(t^*,s^*)=(u,v)$, $|U\cap V|\geq t+2$ and $(\mathcal{F}, \mathcal{G})=(\mathcal{L}(X,U,V), \mathcal{L}(X,V,U))$. This finishes the proof. {\hfill$\square$}\vspace{1em}

We treat the case of $u=v=3$ in the following proposition.
\begin{proposition}\label{prop}
	Let $k\geq t+2\geq4$, $n\geq7(k-t+1)k^2$, and let  $\mathcal{F},\mathcal{G}\subseteq\binom{[n]}{k}$ be cross $t$-intersecting. Suppose that $\min\{|\mathcal{F}|,|\mathcal{G}|\}\geq|\mathcal{L}(n,k,t,3,3)|$, then one of the following holds.
	\begin{itemize}
		\item[\rm(i)]$\mathcal{F}, \mathcal{G}\subseteq\mathcal{A}(Z)$ for some $Z\in\binom{[n]}{t+2}$.
		\item[\rm(ii)]
		$\max\{\gamma_t(\mathcal{F}),\gamma_t(\mathcal{G})\}\leq\gamma_t(\mathcal{L}(n,k,t,3,3))$. Moreover, if $\gamma_t(\mathcal{F})=\gamma_t(\mathcal{L}(n,k,t,3,3))$ or $\gamma_t(\mathcal{G})=\gamma_t(\mathcal{L}(n,k,t,3,3))$, then $\mathcal{F}=\mathcal{L}(X,U,V)$ and $\mathcal{G}=\mathcal{L}(X,V,U)$ for some  $X\in\binom{[n]}{t}$, $U\in\binom{[n]}{t+3}$ and $V\in\binom{[n]}{t+3}$ with $X\subseteq U\cap V$ and $|U\cap V|\geq t+2$.
		\item[\rm(iii)]$(\mathcal{F},\mathcal{G})\cong(\mathcal{C}_1,\mathcal{C}_2)$, where
		$$\mathcal{C}_i=\left\{F\in\binom{[n]}{k}:T\subseteq F\;\mbox{for some}\;T\in\mathcal{T}_i,\;\mbox{or}\;[t+3]\setminus F=\{j\}\;\mbox{for some}\;j\in[5,t+3]\right\}$$
		for $i=1,2$, and $\mathcal{T}_1=\{[5,t+3]\cup P:P\in\{\{1,2\},\{3,4\},\{1,4\}\}\}$ and $\mathcal{T}_2=\{[5,t+3]\cup P:P\in\{\{1,3\},\{2,4\},\{1,4\}\}\}$.
	\end{itemize}
\end{proposition}
\begin{proof}
We adapt the notation in the proof of Theorem \ref{thmdiv}, and further set $u=v=3$. We need only to complete the item (c) in Case 2.1, and the proof of the remaining cases will be omitted as there is nothing new. Suppose (c) holds, that is, $(\mathcal{S},\mathcal{T})$ conforms to some pair described in Lemma \ref{lemmastrue-alike}. Our goal is to prove that (iii) holds. By (\ref{equthmdiv0}), we have $\max\{|\mathcal{F}|,|\mathcal{G}|\}\geq2.9\binom{n-t-1}{k-t-1}$. This together with Lemma \ref{lemmastrue-alike},  $|\mathcal{F}|\leq(\delta+|\mathcal{S}|)\binom{n-t-1}{k-t-1}$ and $|\mathcal{G}|\leq(\delta+|\mathcal{T}|)\binom{n-t-1}{k-t-1}$ yields $|\mathcal{S}|=|\mathcal{T}|=3$, and there are a $(t-1)$-subset $I$ and four elements outside it, say  $1,2,3,4$ after relabeling, such that
\begin{equation}\label{equthmdiv2}
	\mathcal{S}=\{I\cup P:P\in\{\{1,2\},\{3,4\},\{1,4\}\}\}\;\;\mbox{and}\;\;\mathcal{T}=\{I\cup Q:Q\in\{\{1,3\},\{2,4\},\{1,4\}\}\}.
\end{equation}
If some member of $\mathcal{T}$ is not a $t$-cover of $\mathcal{F}$, then Lemma \ref{lemmaind} yields $|\mathcal{F}|<(\delta+2+b/r)\binom{n-t-1}{k-t-1}$, which contradicts (\ref{equthmdiv0}). Hence $\mathcal{T}$ is a set of $t$-covers of $\mathcal{F}$, and similarly $\mathcal{S}$ is a set of $t$-covers of $\mathcal{G}$. Write $\mathcal{U}=\binom{[n]}{k}$ for short. Since $(\mathcal{S},\mathcal{T})$ is a pair of cross $t$-intersecting families and $(\mathcal{F},\mathcal{G})$ is maximal, we have
\begin{align*}
	\mathcal{F}&\supseteq\mathcal{U}[\mathcal{S}]\;\;\mbox{and}\;\;\mathcal{G}\supseteq\mathcal{U}[\mathcal{T}].
\end{align*}
By a routine counting argument, we get 
$$|\mathcal{U}[\mathcal{S}]|=|\mathcal{U}[\mathcal{T}]|=3\binom{n-t-1}{k-t-1}-2\binom{n-t-2}{k-t-2}=|\mathcal{L}(n,k,t,3,3)|-(t-1)\binom{n-t-3}{k-t-2}.$$ 
Since $|\mathcal{F}|\geq|\mathcal{L}(n,k,t,3,3)|$, we have $\mathcal{F}\setminus\mathcal{U}[\mathcal{S}]\neq\emptyset$. We claim that $\mathcal{F}\setminus\mathcal{U}[\mathcal{S}]=\mathcal{F}\setminus\mathcal{F}[I]$. Indeed, assume to the contrary that $I\subseteq F$ for some $F\in\mathcal{F}\setminus\mathcal{U}[\mathcal{S}]$. Then $F$ contains none of $\{1,2\},\{3,4\}$ and $\{1,4\}$. Note that $I\cup\{2,4\}\in\mathcal{T}$ is a $t$-cover of $\mathcal{F}$, we have $\{2,4\}\cap F\neq\emptyset$. It follows that $1\notin F$. However, since $F$ also intersects both $\{1,3\}$ and $\{1,4\}$, this leads to $\{3,4\}\subseteq F$, which is a contradiction. Hence our claim is true. Now every member of $\mathcal{F}\setminus\mathcal{U}[\mathcal{S}]$ does not contain $I$, and then each of them must contain $\{1,2,3,4\}$ and all but exactly one element of $I$. Hence
\begin{align*}
	|\mathcal{L}(n,k,t,3,3)|&\leq|\mathcal{F}|=|\mathcal{U}[\mathcal{S}]|+|\mathcal{F}\setminus\mathcal{F}[I]|\\
	&\leq|\mathcal{U}[\mathcal{S}]|+(t-1)\binom{n-t-3}{k-t-2}=|\mathcal{L}(n,k,t,3,3)|,
\end{align*}
and consequently $|\mathcal{F}|=|\mathcal{L}(n,k,t,3,3)|$ and
\begin{equation*}
	\mathcal{F}=\mathcal{U}[\mathcal{S}]\cup\left\{F\in\mathcal{U}:|I\setminus F|=1,\;[4]\subseteq F\right\}.
\end{equation*}
By the same argument, we obtain $|\mathcal{G}|=|\mathcal{L}(n,k,t,3,3)|$ and $$\mathcal{G}=\mathcal{U}[\mathcal{T}]\cup\left\{G\in\mathcal{U}:|I\setminus G|=1,\;[4]\subseteq G\right\}.$$
Finally, it is routine to check that
\begin{align*}
\gamma_t(\mathcal{F})&=\gamma_t(\mathcal{G})=|\mathcal{G}\setminus\mathcal{G}[I\cup\{1\}]|=|\mathcal{U}[I\cup\{2,4\}]\setminus\mathcal{U}[\{1\}]|+|\mathcal{G}\setminus\mathcal{U}[\mathcal{T}]|\\
&=\binom{n-t-2}{k-t-1}+(t-1)\binom{n-t-3}{k-t-2}>t\binom{n-t-3}{k-t-2}=\gamma_t(\mathcal{L}(n,k,t,3,3)).
\end{align*}
This finishes the proof.
\end{proof}

\noindent{\bf Proof of Theorem \ref{thmstability-t-div}.}\;Let $n\geq7(k-t+1)k^2$ and $d\geq2$, and let $\mathcal{F}\subseteq\binom{[n]}{k}$ be a $t$-intersecting family with 
\begin{align}
	|\mathcal{F}|&\geq\binom{n-t}{k-t}-\binom{n-t-d}{k-t}+t\binom{n-t-d}{k-t-d+1}.\label{equthmstability-t-div2}
\end{align}
Recall that the right-hand side of \eqref{equthmstability-t-div2} equals $|\mathcal{L}(n,k,t,d,d)|$ for $3\leq d\leq k-t+1$. 

For $d=2$, we get directly from Corollary \ref{corothmmaxdiv} that $$\gamma_t(\mathcal{F})\leq t\binom{n-t-2}{k-t-1}=\gamma_t(\mathcal{A}(n,k,t))=\gamma_t(\mathcal{L}(n,k,t,2,2)),$$
as $\mathcal{A}(n,k,t)=\mathcal{L}(n,k,t,2,2)$. Moreover, if equality holds, then $\mathcal{F}\subseteq\mathcal{A}(Z)$ for some $Z\in\binom{[n]}{t+2}$, and further $\mathcal{F}\cong\mathcal{A}(n,k,t)$ as $|\mathcal{F}|\geq|\mathcal{L}(n,k,t,2,2)|=|\mathcal{A}(n,k,t)|$. 

For $d\in[3,k-t+1]$, the desired result follows directly from Theorem \ref{thmdiv} and Proposition \ref{prop} by setting $\mathcal{G}=\mathcal{F}$.

It remains to suppose $d\geq k-t+2$. Now (\ref{equthmstability-t-div2}) yields
\begin{equation}\label{equthmstability-t-div3}
|\mathcal{F}|\geq\binom{n-t}{k-t}-\binom{n-k-2}{k-t}=|\mathcal{H}(n,k,t)|+\left(\binom{n-k-2}{k-t-1}-t\right)>|\mathcal{H}(n,k,t)|.
\end{equation}
If $\mathcal{F}$ is trivial, then $\gamma_t(\mathcal{F})=0=t\binom{n-t-d}{k-t-d+1}$ and we have done. So we assume that $\mathcal{F}$ is non-trivial. Then $\gamma_t(\mathcal{F})>t$. To see this, assume to the contrary that $|\mathcal{F}\setminus\mathcal{F}[X]|=\gamma_t(\mathcal{F})\leq t$ for some $X\in\binom{[n]}{t}$. By fixing a $G\in\mathcal{F}$ with $|G\cap X|<t$ and using Lemma \ref{lemmaind}, we obtain $$|\mathcal{F}|=|\mathcal{F}[X]|+|\mathcal{F}\setminus\mathcal{F}[X]|\leq\left(\binom{n-t}{k-t}-\binom{n-k-1}{k-t}\right)+t=|\mathcal{H}(n,k,t)|,$$
which contradicts (\ref{equthmstability-t-div3}). Hence $\gamma_t(\mathcal{F})>t$. We apply Theorem \ref{thmdiv} for $k\geq t+3$ to the pair  $(\mathcal{F},\mathcal{F})$ with $u=v=k-t+1$; for $k=t+2$, we apply Proposition \ref{prop} to $(\mathcal{F},\mathcal{F})$. Note that alternative (iii) of Proposition \ref{prop} cannot
occur, since the two components of $(\mathcal{F},\mathcal{F})$
are identical. Since $|\mathcal{F}|>|\mathcal{H}(n,k,t)|=|\mathcal{L}(n,k,t,u,v)|$ and $\gamma_t(\mathcal{F})>t$, we have $k-t+1\leq t+2$ and $\mathcal{F}\subseteq\mathcal{A}(Z)$ for some $(t+2)$-subset $Z$. We need to verify  $d\leq t+2$. Assume for contradiction that $d\geq t+3$. Then $\binom{n-t}{k-t}-\binom{n-t-d}{k-t}\geq \binom{n-t}{k-t}-\binom{n-2t-3}{k-t}$, and further, by Lemma \ref{lemmabinom} (iii), it is at least 
\begin{align*}
\left((t+3)-\binom{t+3}{2}\cdot\frac{k-t}{n-t}\right)\binom{n-t-1}{k-t-1}>(t+2)\binom{n-t-1}{k-t-1}.
\end{align*}
So $|\mathcal{F}|\geq\binom{n-t}{k-t}-\binom{n-t-d}{k-t}>(t+2)\binom{n-t-1}{k-t-1}$. On the other hand, $|\mathcal{F}|\leq|\mathcal{A}(Z)|<(t+2)\binom{n-t-1}{k-t-1}$, a contradiction. Thus $d\leq t+2$, finishing the proof.{\hfill$\square$}
\section{Concluding remarks}\label{secrmk}
\subsection{Applications of Theorem \ref{thmcrossekr}}
Let us derive two Erd\H{o}s--Ko--Rado type results for sufficiently spread objects.

First, we give an Erd\H{o}s--Ko--Rado type theorem for signed sets. A \emph{$q$-signed $k$-set} on $[n]$ is a pair $(A,f)$, where $A\in\binom{[n]}{k}$ and $f$ is a function from $A$ to $[q]$. We denote by $\mathcal{S}_{n,k,q}$ the set of such pairs. A family $\mathcal{F}\subseteq\mathcal{S}_{n,k,q}$ is said to be $t$-intersecting if for all $(A,f),(B,g)\in\mathcal{F}$, there are at least $t$ elements in $A\cap B$ on which $f$ and $g$ agree. Similarly, we may define cross $t$-intersecting families for signed sets. An analogue of Erd\H{o}s--Ko--Rado theorem for signed sets states that, for all $n\geq k$ and $q\geq2$, a $1$-intersecting family of $q$-signed $k$-sets on $[n]$ has size at most $\binom{n-1}{k-1}q^{k-1}$, and if $(n,q)\neq(k,2)$, then every optimal family consists of all $(A,f)$ with $x\in A$ and $f(x)=y$ for some fixed $x\in[n]$ and $y\in[q]$. The result was proved by Deza and Frankl \cite{Deza-Frankl-1983} based on the shifting technique, and by Bollob\'{a}s and Leader \cite{Bollobas-Leader} based on Katona's cycle. For  results on signed sets, see e.g.,  \cite{Borg-2009,Borg-Leader,Yao-Lv-Wang}.  Here is a set interpretation to signed sets, namely, associated a $q$-signed $k$-set to its `graph'  $\{(a,f(a)):a\in A\}$, which is a $k$-subset of $\Omega=[n]\times[q]$. Thus we may identify $\mathcal{S}_{n,k,q}$ with a family $\mathcal{A}\subseteq\binom{\Omega}{k}$. Then for each $t$-subset $X\subseteq\Omega$, the family $\mathcal{A}[X]$ is either empty or has size $\binom{n-t}{k-t}q^{k-t}$. Hence $\mathcal{A}$ is weakly $(r,t)$-spread for all  $r\leq q(n-t)/(k-t)$, and so we obtain immediately the following result by applying Theorem \ref{thmcrossekr}.
\begin{theorem}
	Suppose $n>k\geq\ell\geq t\geq1$ and $q\geq2$. The following hold.
	\begin{itemize}
		\item[\rm(i)]Suppose $n\geq t+2e(k-t)k/q$. If $\mathcal{F}\subseteq\mathcal{S}_{n,k,q}$ is $t$-intersecting, then $|\mathcal{F}|\leq\binom{n-t}{k-t}q^{k-t}$, with equality if and only if $\mathcal{F}=\{(A,f)\in\mathcal{S}_{n,k,q}:X\subseteq A,\;f|_X=f_0\}$ for some $X\in\binom{[n]}{t}$ and $f_0:[n]\to[q]$.
		\item[\rm(ii)]Suppose $n\geq t+2e\cdot\max\{(t+2)^2(k-t),k^2\}/q$. If $\mathcal{F}\subseteq\mathcal{S}_{n,k,q}$ and $\mathcal{G}\subseteq\mathcal{S}_{n,\ell,q}$ are cross $t$-intersecting, then $|\mathcal{F}||\mathcal{G}|\leq\binom{n-t}{k-t}\binom{n-t}{\ell-t}q^{k+\ell-2t}$, with equality if and only if $\mathcal{F}=\{(A,f)\in\mathcal{S}_{n,k,q}:X\subseteq A,\;f|_X=f_0\}$ and $\mathcal{G}=\{(B,g)\in\mathcal{S}_{n,\ell,q}:X\subseteq B,\;g|_X=f_0\}$ for some $X\in\binom{[n]}{t}$ and $f_0:[n]\to[q]$.
	\end{itemize}
\end{theorem}
Let us note that it was asserted in \cite{Deza-Frankl-1983} that (i) holds
whenever the classical Erd\H{o}s--Ko--Rado theorem holds, namely when
\(n\geq n_0(k,t)=(t+1)(k-t+1)\). The current bound can be much smaller than $n_0(k,t)$ provided that the alphabet $[q]$ is large. The result in (i) was also proved by Borg \cite{Borg-2009} for sufficiently large $q$ depending on $k$ and $t$. 

Second, we consider the Erd\H{o}s--Ko--Rado theorem for $k$-partitions. A $k$-partition of $[n]$ is a collection of $k$ pairwise disjoint non-empty subsets whose union is $[n]$. A family of $k$-partitions of $[n]$ is called $t$-intersecting if any two of its members share at least $t$ blocks. The Erd\H{o}s--Ko--Rado theorem for $k$-partitions states that, if $n$ is sufficiently large depending on $k$ and $t$, then every $t$-intersecting family of $k$-partitions of $[n]$ has size at most the Stirling number of the second kind 
\begin{equation*}
	\spn{n-t}{k-t}=\frac{1}{(k-t)!}\sum_{j=0}^{k-t}(-1)^j\binom{k-t}{j}(k-t-j)^{n-t}.
\end{equation*}
Moreover, every family achieving the maximum must consist of all $k$-partitions which contain $t$  fixed singletons. In 2000, Erd\H{o}s and  Sz\'{e}kely proved the upper bound for sufficiently large $n$ depending on $k$ and $t$. Then the theorem was established by Kupavskii \cite{Kupavskii-2026} for $n\geq2k\log_2n$ using the spread approximation method, and by the authors \cite{Wen-Lv-2026} for $n\geq t+1+(k-t+1)\log_2((t+1)(k-t+1))$ using the $t$-cover method. It is natural to find a set interpretation of a $k$-partition, that is, a $k$-subset of the power set of $[n]$. We denote by $\mathcal{A}$ the collection of all corresponding $k$-sets. A key property proved by Kupavskii \cite{Kupavskii-2026} is that, for $k\geq t+2$, the family $\mathcal{A}$ is weakly $(r,t)$-spread for $r=2^{(n-t-1)/(k-t-1)}-1$. Thus Theorem \ref{thmcrossekr} yields an improvement when $k$ is large with respect to $t$.
\begin{theorem}
Let $k\geq t+2$ and $\mathcal{F}$ be a  $t$-intersecting family of $k$-partitions of $[n]$. If $n\geq t+1+(k-t-1)\log_2(2ek+1)$, then $
	|\mathcal{F}|\leq\spn{n-t}{k-t}$. 
Moreover, equality holds precisely if $\mathcal{F}$ consists of all $k$-partitions with $t$ fixed singletons.
\end{theorem}
\subsection{Extremal problems for large cross $t$-intersecting families}
In Theorem \ref{thmmin-max}, we establish a max-min result for cross $t$-intersecting families for $n=\Omega(k^2)$. We have attempted to prove it for $n=\Omega(kt)$, but the present method does not seem to reach this range. Precisely, there are two reasons. First, due to the use of Lemma \ref{lemmakey}, the estimate on the maximum $t$-degree does not work for $n<(k-t+1)^2$. Another obstacle is the lack of further structural information in the case $N\leq q-t-2$. Nevertheless, it might be interesting to consider whether the theorem holds in that range, and it seems likely that a different approach is required. In Theorem  \ref{thmcrosshm}, we improve the bound on \(n\) to
$\Omega(\max\{t^2(k-t),k^2\})$ for the product version of the
Hilton--Milner type theorem. It would be of considerable interest to determine the maximum of $|\mathcal F||\mathcal G|$ for $n\geq Ckt$.
\begin{problem}
Is there an absolute constant $C>0$ such that the following holds
		for all $k\geq t+2\geq4$ and $n\geq Ckt$?
		If $\mathcal F,\mathcal G\subseteq\binom{[n]}k$ are cross
		$t$-intersecting with $|\cap\mathcal F|<t$ and
		$|\cap\mathcal G|<t$, then
		$$
		|\mathcal F||\mathcal G|
		\leq
		\max\{|\mathcal A(n,k,t)|^2,|\mathcal H(n,k,t)|^2\}.
		$$
\end{problem}
\subsection{Families with large size and large $t$-diversity}
We note that for fixed $u,v\in[3,k-t+1]$, Theorem \ref{thmdiv} does not hold for $n=O((k-t)t^2)$. Similarly, for fixed $d\in[3,k-t+1]$, Theorem \ref{thmstability-t-div} does not hold for $n=O((k-t)t^2)$. To see this, consider the Frankl family
\begin{equation*}
\mathcal{F}_2(n,k,t)=\left\{F\in\binom{[n]}{k}:|F\cap[t+4]|\geq t+2\right\}.
\end{equation*}
For $u,v\in[3,k-t+1]$, by a routine computation, 
\begin{align*}
|\mathcal{F}_2(n,k,t)|-|\mathcal{L}(n,k,t,u,v)|&\geq\frac{t(t+5)}{2}\binom{n-t-4}{k-t-2}-v\binom{n-t-4}{k-t-1}\\
&=\left(\frac{t(t+5)}{2}-\frac{v(n-k-2)}{k-t-1}\right)\binom{n-t-4}{k-t-2}.
\end{align*}
Moreover, we have 
$$\gamma_t(\mathcal{F}_2(n,k,t))=\frac{t(t+7)}{2}\binom{n-t-4}{k-t-2}+t\binom{n-t-4}{k-t-3}>t\binom{n-t-3}{k-t-2}.$$
Then for $n<k+2+0.5(k-t-1)t(t+5)/v$, it holds that $$|\mathcal{F}_2(n,k,t)|>|\mathcal{L}(n,k,t,u,v)|\;\;\;\mbox{and}\;\;\gamma_t(\mathcal{F}_2(n,k,t))>\gamma_t(\mathcal{L}(n,k,t,u,v)).$$
Therefore, for fixed $u,v\in[3,k-t+1]$, Theorem \ref{thmdiv} fails for $n<0.5(k-t-1)t^2/\max\{u,v\}$, as witnessed by $(\mathcal{F}_2(n,k,t),\mathcal{F}_2(n,k,t))$. So for $k=O(t)$, the order $(k-t)k^2$ of the present lower bounds on $n$ is best possible. However, when $t$ is fixed and $k$ is large, our method does not work when $n$ is only a constant multiple of $k$. It is therefore natural to ask whether a linear lower bound on $n$ is sufficient in this range. In view of the family $\mathcal{A}(n,k,t)$, we restrict attention to $d\geq t+3$ and formulate the following conjecture.
\begin{conjecture}
For every $t\geq2$, there is a constant $C=C(t)$ such that the following holds for all $k\geq t+2$ and $n\geq Ck$. Let $d\in[t+3,k-t+1]$, and let
$\mathcal{F}\subseteq\binom{[n]}{k}$ be a $t$-intersecting family. If $|\mathcal{F}|\geq|\mathcal{L}(n,k,t,d,d)|$, then $\gamma_t(\mathcal{F})\leq\gamma_t(\mathcal{L}(n,k,t,d,d))$.
\end{conjecture}

		\section*{Acknowledgments}
B. Lv is supported by National Natural Science Foundation of China (12571347 \& 12131011), and Beijing Natural Science Foundation (1252010).
\
\addcontentsline{toc}{chapter}{Bibliography}

{
	}

\begin{thebibliography}{99}
		\setlength{\itemsep}{-1pt}
		\bibitem{Ahlswede-Khachatrian-1996}
		R. Ahlswede and L.H. Khachatrian, The complete nontrivial-intersection theorem for systems of finite sets, J. Combin. Theory Ser. A 76 (1996) 121--138.
					
		\bibitem{Ahlswede-Khachatrian-1997}
		R. Ahlswede and L.H. Khachatrian, The complete intersection theorem for systems of finite sets, European J. Combin. 18 (1997) 125--136.
		
		\bibitem{sunflower} R. Alweiss, S. Lovett, K. Wu, J. Zhang, Improved bounds for the sunflower lemma, Ann. of Math. 194 (3) (2021) 795--815.
		
		\bibitem{Balogh-Mubayi} J. Balogh and D. Mubayi, A new short proof of a theorem of Ahlswede and
		Khachatrian, J. Combin. Theory Ser. A 115 (2008) 326--330.
		
		\bibitem{Blokhuis-etal-2010}
		A. Blokhuis, A. Brouwer, A. Chowdhury, P. Frankl, T. Mussche, B. Patkós
		and T. Szőnyi, A Hilton--Milner theorem for vector spaces, Electron. J. Combin. 17
		(2010) \#R71.
		
		\bibitem{Bollobas-Leader} B. Bollob\'{a}s  and  I. Leader,  An  Erd\H{o}s--Ko--Rado  theorem  for  signed  sets,  Comput. Math. Appl. 34 (1997) 9--13.
		
		\bibitem{Borg-2009} P. Borg, On $t$-intersecting families of signed sets and permutations, Discrete Math. 309 (2009) 3310--3317.
		
		\bibitem{Borg-Leader} P. Borg and I. Leader, Multiple cross-intersecting families of signed sets, J. Combin. Theory Ser. A 117 (2010) 583--588.
		
		\bibitem{Lv-2021}
		M. Cao, B. Lv and K. Wang, The structure of large non-trivial $t$-intersecting families of
		finite sets, European J. Combin. 97 (2021) 103373.
		
		\bibitem{Cao-Lu-Lv-Wang-2024} M. Cao, M. Lu, B. Lv and K. Wang, Nearly extremal non-trivial cross $t$-intersecting families and $r$-wise $t$-intersecting families, European J. Combin. 120 (2024) 103958.
					
		\bibitem{Deza-Frankl-1983} M. Deza and P. Frankl, The Erd\H{o}s--Ko--Rado theorem--22 years later, SIAM J. Algebraic Discrete Methods 4 (1983) 419--431.
					
		\bibitem{Dinur-Friedgut}I. Dinur and E. Friedgut,  Intersecting families are essentially contained in juntas, Combin. Probab. Comput. 18 (2009) 107--122.
		
		\bibitem{Ellis-2019} A. Ellis, N. Keller and N. Lifshitz, Stability versions of  Erd\H{o}s--Ko--Rado type theorems via isoperimetry, J. Eur. Math. Soc. 21 (2019) 3857--3902.
		
		\bibitem{Ellis-book}D. Ellis, Intersection problems in extremal combinatorics: theorems, techniques and questions 
		old and new, in: Surveys in Combinatorics 2022, in: London Math. Soc. Lecture Note Ser., vol. 481, 
		Cambridge Univ. Press, Cambridge, 2022, pp. 115--173.
		
		\bibitem{Ellis-2024} D. Ellis, N. Keller and N. Lifshitz, Stability for the complete intersection theorem, and the
		forbidden intersection problem of Erd\H{o}s and S\'{o}s, J. Eur. Math. Soc. 26 (2024)  1611--1654.
		
		\bibitem{Erdos-Ko-Rado-1961} P. Erd\H{o}s, C. Ko and R. Rado, Intersection theorems for systems of finite sets, Quart. J. Math. Oxf. 2 (12) (1961) 313--320.
	
	\bibitem{Erdos-Szekely-2000}
	P.L. Erd\H{o}s and L.A. Sz\'{e}kely, Erd\H{o}s--Ko--Rado theorems of higher order, in: I. Alth{\"o}fer, N. Cai, G. Dueck, L. Khachatrian, M.S. Pinsker, A. S\'{a}rk{\"o}zy, I. Wegener and Z. Zhang
	(Eds.), Numbers, Information and Complexity, Springer US, Boston, MA, 2000, 117--124.
	

	\bibitem{Frankl-1976} P. Frankl, The Erd\H{o}s--Ko--Rado theorem is true for $n = ckt$, in: Combinatorics, Vol. I, Proc. Fifth Hungarian Colloq., Keszthely, 1976, in: Colloq. Math. Soc. J\'{a}nos Bolyai, vol. 18, North-Holland, 1978, 365--375.
					
	\bibitem{Frankl-1978} P. Frankl, On intersecting families of finite sets, J. Combin. Theory Ser. A 24 (1978) 146--161.
					
	\bibitem{Frankl-1987} P. Frankl, Erd\H{o}s--Ko--Rado theorem with conditions on the maximal degree, J. Combin. Theory Ser. A 46 (1987) 252--263.
					
	\bibitem{Frankl-shifting} P. Frankl, The shifting technique in extremal set theory, in: Surveys in Combinatorics, in: London 
		Math. Soc. Lecture Note Ser., vol. 123, Cambridge Univ. Press, Cambridge, 1987, pp. 81--110.
		
	\bibitem{Frankl-2017} P. Frankl, Antichains  of fixed diameter, Moscow J. Combin. Number Theory 7 (2017) 189--219.
		
	\bibitem{Frankl-2020} P. Frankl, Maximum degree and diversity in intersecting hypergraphs, J. Combin. Theory Ser. B 144 (2020) 81--94.
	
	\bibitem{Frankl--Furedi-1991} P. Frankl and Z. F\"{u}redi, Beyond the Erd\H{o}s--Ko--Rado theorem, J. Combin. Theory Ser. A 56 (1991) 182--194.
		
    \bibitem{Frankl-Kupavskii-2020} P. Frankl and A. Kupavskii, Sharp results concerning disjoint cross-intersecting families, European J. Combin. 86 (2020) 103089.
				
	\bibitem{Frankl-Kupavskii-2021} P. Frankl and A. Kupavskii, Diversity, J. Combin. Theory Ser. A 182 (2021) 105468.	
		
	\bibitem{Frankl-Kupavskii-2025}P. Frankl and A. Kupavskii, The Hajnal and Rothschild problem,  arXiv:2502.06699.
		
	\bibitem{Frankl-Tokushige-2016} P. Frankl and N. Tokushige, Invitation to intersection problems for finite sets, J.	Combin. Theory Ser. A 144 (2016) 157--211.
					
	\bibitem{Frankl-Tokushige-book}P. Frankl and N. Tokushige, Extremal Problems for Finite Sets, American Mathematical Society, 2018.
					
	\bibitem{Frankl-Wang-2023} P. Frankl and J. Wang, A product version of the Hilton--Milner theorem, J.	Combin. Theory Ser. A 200 (2023) 105791.
				
	\bibitem{Frankl-Wang-2024} P. Frankl and J. Wang, A product version of the Hilton--Milner--Frankl theorem, Sci. China Math. 67 (2024) 455--474.
		
	\bibitem{Frankl-Wang-2024-diversity}P. Frankl and J. Wang, Improved bounds on the maximum diversity of intersecting families, European J. Combin. 118 (2024) 103885.
		
	\bibitem{Frankl-Wang-2026+} P. Frankl and J. Wang, A product version of the Hilton--Milner theorem II, 	arXiv:2605.09246.
	
	\bibitem{Furedi-1995}Z. F\"{u}redi, Cross-intersecting  families of finite sets, J.	Combin. Theory Ser. A 72 (1995) 332--339.
					
	\bibitem{Godsil-Meagher-book}C. Godsil and K. Meagher, Erd\H{o}s--Ko--Rado Theorems: Algebraic Approaches, Cambridge University Press, 2015.
					
	\bibitem{Hajnal-Rothschild}A. Hajnal and B. Rothschild, A generalization of the Erd\H{o}s--Ko--Rado theorem on finite set systems, J. Combin. Theory Ser. A 15 (1973), 359--362.
		
	\bibitem{Hilton-1976} A.J.W. Hilton, The Erd\H{o}s--Ko--Rado theorem with valency conditions, Unpublished Manuscript, 1976.
		
	\bibitem{Hilton-1977}A.J.W. Hilton, An intersection theorem for a collection of families of subsets of a finite set, J. Lond. Math. Soc. (2) 15 (1977) 369--376.
		
	\bibitem{Hilton-Milner-1967} A.J.W. Hilton and E.C. Milner, Some intersection theorems for systems of finite sets, Quart. J. Math. Oxf. 2 (18) (1967) 369--384.
	
	\bibitem{Huang-2019}H. Huang, Two extremal problems on intersecting families, European J. Combin. 76 (2019) 1--9.
		
	\bibitem{Keevash} P. Keevash, Shadows and intersections: Stability and new proofs, Adv. Math. 218 (2008) 1685--1703.
		
	\bibitem{Keevash-Long-2020} P. Keevash and E. Long, Stability for vertex isoperimetry in the cube, J. Combin. Theory Ser. B 145 (2020) 113--144.
		
	\bibitem{Keevash-2021} P. Keevash, N. Lifshitz, E. Long and D. Minzer, Global hypercontractivity and its applications, arXiv:2103.04604.
		
	\bibitem{Keevash-2023}P. Keevash, N. Lifshitz, E. Long and D. Minzer, Forbidden intersections for codes, J. Lond. Math. Soc. (2) 108 (2023) 2037--2083.
		
	\bibitem{Keevash-2024}P. Keevash, N. Lifshitz, E. Long and D. Minzer, Hypercontractivity for global functions and sharp thresholds, J. Amer. Math. Soc. 37 (2024) 245--279.
	
	\bibitem{KKLS} N. Keller, A. Kupavskii, N. Lifshitz and O. Sheinfeld, A complete intersection theorem for large permutation groups, arXiv:2607.00318.
	
	\bibitem{Keller-Lifshitz}N. Keller and N. Lifshitz, The junta method for hypergraphs and the Erd\H{o}s–Chv\'{a}tal simplex conjecture, Adv. Math. 392 (2021) 107991.
	
	\bibitem{KLMS} N. Keller, D. Minzer, E. Long and O. Sheinfeld, On $t$-intersecting families of permutations, Adv. Math. 445 (2024) 109650.
		
	\bibitem{Kupavskii-2018}A. Kupavskii, Diversity of uniform intersecting families, European J. Combin. 74 (2018) 39--47.  
		
	\bibitem{Kupavskii-2024}A. Kupavskii, An almost complete $t$-intersection theorem for permutations. arXiv:2405.07843.
		
	\bibitem{Kupavskii-2026} A. Kupavskii, Erd\H{o}s--Ko--Rado type results for partitions via spread approximations, European J. Combin. 132 (2026) 104288.
		
	\bibitem{Kupavskii-Zakharov-2018} A. Kupavskii and D. Zakharov, Regular bipartite graphs and intersecting families, J. Combin. Theory Ser. A 155 (2018) 180--189.
		
	\bibitem{Kupavskii-Zakharov-2024} A. Kupavskii and D. Zakharov, Spread approximations for forbidden intersections problems, Adv. Math. 445 (2024) 109653. 
				
    \bibitem{Lemons-Palmer} N. Lemons and C. Palmer, The unbalance of set systems, Graphs  Combin. 24 (2008) 361--365.

    \bibitem{Mors} M. M\"{o}rs, A generalization of a theorem of Kruskal, Graphs Combin. 1 (N1) (1985) 167--183.
	
	\bibitem{Saengrungkongka}P. Saengrungkongka, extremal $t$-intersecting families of permutations for large $t$, arXiv:2605.26051.
	
	\bibitem{Wen-Lv-2026} J. Wen and B. Lv, Erd\H{o}s--Ko--Rado theorem and Hilton--Milner type theorem for $k$-partitions, J. Combin. Theory Ser. A 223 (2026) 106219.
					
	\bibitem{Wilson-1984} R.M. Wilson, The exact bound in the Erd\H{o}s--Ko--Rado theorem, Combinatorica 4 (1984) 247--257.
	
	\bibitem{Yao-Lv-Wang} T. Yao, B. Lv and K. Wang, Large non-trivial  $t$-intersecting families of signed sets, Australas. J. Combin. 89 (2024) 32--48.		
\end{thebibliography}
\end{document}